\documentclass[11pt, reqno]{amsart}

\usepackage[utf8]{inputenc}
\usepackage[T1]{fontenc}
\usepackage{amsmath, mathtools, amsthm}
\usepackage{amssymb,url,xspace}
\usepackage[mathscr]{eucal}
\usepackage[dvipsnames]{xcolor}
\usepackage{dsfont} % for mathds
\usepackage{cases}
\usepackage{tikz}
\usetikzlibrary{arrows.meta}
\usepackage{color}
\usepackage[hidelinks]{hyperref}
\usepackage[capitalise]{cleveref}
\hypersetup{
	colorlinks,
	citecolor=red!70!black,
	filecolor=black,
	linkcolor=blue!70!black,
	urlcolor=black
}
\usepackage{enumitem}

\newtheorem{assumption}{Assumption}

\usepackage{dsfont} % for mathds
\usepackage{xfrac} % for 1/2
\usepackage{enumitem} % for different enumerate
\usepackage{empheq} % for dcases
\usepackage{yfonts} % for special I
\usepackage[pagewise]{lineno}

\calclayout

\newtheorem{theorem}{Theorem}[section]

\newtheorem{lemma}[theorem]{Lemma}
\newtheorem{proposition}{Proposition}

\theoremstyle{definition}
\newtheorem{definition}[theorem]{Definition}
\newtheorem{remark}{Remark}

\title[Networks of geometrically exact beams]{Networks of geometrically exact beams: well-posedness and stabilization}

\author{Charlotte Rodriguez}

\thanks{
  \textbf{Funding:} This project has received funding from the European Union’s Horizon 2020 research and innovation programme under the Marie Sklodowska-Curie grant agreement No.765579-ConFlex.}

\address {Charlotte Rodriguez \newline \indent
    {Department Mathematik, Lehrstuhl f\"ur Angewandte Mathematik, \newline \indent
    Friedrich-Alexander-Universit\"at Erlangen-N\"urnberg},
    \newline \indent
    {Cauerstr. 11, 91058 Erlangen, Germany}
  }
\email{\texttt{charlotte.rodriguez@fau.de}}
  
\date{\today}

\begin{document}

\maketitle

\begin{abstract}

In this work, we are interested in tree-shaped networks of freely vibrating beams which are geometrically exact (GEB) -- in the sense that large motions (deflections, rotations) are accounted for in addition to shearing -- and linked by rigid joints. For the intrinsic GEB formulation, namely that in terms of velocities and internal forces/moments, we derive transmission conditions and show that the network is locally in time well-posed in the classical sense. Applying velocity feedback controls at the external nodes of a star-shaped network, we show by means of a quadratic Lyapunov functional and the theory developed by Bastin \& Coron in \cite{BC2016} that the zero steady state of this network is exponentially stable for the $H^1$ and $H^2$ norms. The major obstacles to overcome in the intrinsic formulation of the GEB network, are that the governing equations are semilinar, containing a quadratic nonlinearity, and that linear lower order terms cannot be neglected.

\end{abstract}

\tableofcontents

\addtocontents{toc}{\protect\setcounter{tocdepth}{1}}

%%%%%%%%%%%%%%%%%%%%%%%%%%%%%%%%%%%%%%%%%%%%

\section{Introduction}
\label{sec:intro}

Multi-link flexible structure are of paramount importance in practice, as attests the growing use of large spacecraft structures, trusses, robot arms, solar panels, antennae and so on \cite{chen_serial_EBbeams, LLS, flotow_spacecraft}.
Their dynamic behavior can be modeled by networks of many interconnected flexible elements such as strings, beams, membranes, shells or plates.
Here, we are concerned with networks of so-called \emph{geometrically exact beams}.

Various one-dimensional models have been developed for beams made of \emph{linear elastic materials} -- i.e. small strains. The \emph{Euler-Bernoulli model} describes a beam whose cross-sections remain perpendicular to the centerline. The \emph{Timoshenko model} allows for shearing. A common point between these models is that they account for motions which are negligeable in comparison to the overall dimensions of the beam -- small displacements of the centerline and small rotations of the cross sections.

However, nowadays the use of modern highly flexible light weight structures -- such as robotic arms \cite{grazioso2018robot}, flexible aircraft wings \cite{palacios10aircraft} or wind turbine blades \cite{wang2014windturbine} -- asks for models taking into account not only shear deformation, but also motions of large magnitude.
Such beam models -- called geometrically exact -- have nonlinear governing equations, as is the case for the \emph{geometrically exact beam model} (GEB) originating from the work of Reissner \cite{reissner1981finite} and Simo \cite{simo1985finite}. 
For a freely vibrating beam -- meaning that the applied external forces and moments are set to zero -- of length $\ell>0$ evolving in $\mathbb{R}^3$, the governing equations of the GEB model read%
\footnote{\label{foot:cross_prod}
Here, $u \times \zeta$ denotes the cross product between any $u, \zeta \in \mathbb{R}^3$, and we shall also write $\widehat{u} \,\zeta = u \times \zeta$, meaning that $\widehat{u}$ is the skew-symmetric matrix 
\begin{equation*}
\widehat{u} = \begin{bmatrix}
0 & -u_3 & u_2 \\
u_3 & 0 & -u_1 \\
-u_2 & u_1 & 0
\end{bmatrix}, 
\end{equation*}
and for any skew-symmetric $\mathbf{u} \in \mathbb{R}^{3 \times 3}$, the vector $\mathrm{vec}(\mathbf{u}) \in \mathbb{R}^3$ is such that $\mathbf{u} = \widehat{\mathrm{vec}(\mathbf{u})}$.
}
\begin{equation} \label{eq:GEB}
\partial_t \left( \begin{bmatrix}
\mathbf{R} & 0\\ 0 & \mathbf{R}
\end{bmatrix} \mathbf{M}
\begin{bmatrix}
V \\ W
\end{bmatrix}
\right) = \partial_x \begin{bmatrix}
\mathbf{R} \Phi \\ \mathbf{R} \Psi \end{bmatrix} + \begin{bmatrix}
0\\ (\partial_x \mathbf{p}) \times (\mathbf{R} \Phi)
\end{bmatrix},
\end{equation}
where the unknown states, for $x\in[0, \ell]$ and $t \geq 0$, are the position $\mathbf{p}(x,t) \in \mathbb{R}^3$ of the beam's centerline and a rotation matrix $\mathbf{R}(x,t)\in \mathrm{SO}(3)$ giving the orientation of the cross sections of the beam, both expressed in some fixed coordinate system. Here, $\mathrm{SO}(3)$ denotes the special orthogonal group, namely, the set of unitary real matrices of size $3$, with determinant equal to $1$ -- which one may also call \emph{rotation} matrices.
On the other hand, $V(x,t), W(x,t), \Phi(x,t),\Psi(x,t)\in \mathbb{R}^3$ denote the \emph{linear velocity}, \emph{angular velocity}, \emph{internal forces} and \emph{internal moments} of the beam respectively, all expressed in a moving coordinate system attached to the centerline of the beam -- a so-called \emph{body-attached basis} -- and are defined by (see Footnote \ref{foot:cross_prod})
\begin{equation} \label{eq:single_beam_VWPhiPsi}
\begin{bmatrix}
V \\ W
\end{bmatrix}
= \begin{bmatrix}
\mathbf{R}^\intercal \partial_t \mathbf{p}\\ \mathrm{vec}\left( \mathbf{R}^\intercal \partial_t \mathbf{R} \right)
\end{bmatrix}, \qquad 
\begin{bmatrix}
\Phi \\ \Psi
\end{bmatrix}= \mathbf{C}^{-1} \begin{bmatrix}
\mathbf{R} ^\intercal \partial_x \mathbf{p}  - e_1 \\ 
\mathrm{vec}\left(\mathbf{R} ^\intercal \partial_x \mathbf{R}  - R^\intercal \tfrac{\mathrm{d}}{\mathrm{d}x} R\right)
\end{bmatrix},
\end{equation}
where $e_1 = (1, 0, 0)^\intercal$.
In the above governing system and definitions, the \emph{mass matrix} $\mathbf{M}(x)\in \mathbb{R}^{6\times 6}$ and \emph{flexibility matrix} $\mathbf{C}(x)\in \mathbb{R}^{6\times 6}$ depend on the geometrical and material properties of the beam, while $R(x)\in \mathrm{SO}(3)$ depends on the initial form of the beam, as it may be pre-curved and twisted before deformation.
Another way of describing geometrically exact beams consists in taking as unknowns so-called \emph{intrinsic} variables -- the velocities $V, W$ and internal forces/moments $\Phi, \Psi$ -- expressed in the body-attached basis.
These are stored in the unknown state $y(x,t) \in \mathbb{R}^{12}$ whose dynamics are then given by a system of the form
\begin{equation}\label{eq:IGEB}
\partial_t y + A(x) \partial_x y + \overline{B}(x) y = \overline{g}(x, y),
\end{equation}
where the coefficients $A,\overline{B}$ and the source $\overline{g}$ depend on $\mathbf{M}, \mathbf{C}$ and $R$.
One should be aware that the matrix $\overline{B}(x)$ is indefinite and, up to the best of our knowledge, may not be assumed arbitrarily small -- in particular cases, the norm of this matrix can be explicitly computed and seen to be away from zero for realistic beam parameters. Furthermore, the function $\overline{g}$ is nonlinear -- quadratic -- with respect to the unknown, which allows for local, but not global, Lipschitz properties.
System \eqref{eq:IGEB} is the \emph{intrinsic geometrically exact beam model} (IGEB), which originates from the work of Hodges \cite{hodges1990, hodges2003geometrically}.
More details are provided in Section \ref{sec:model_mainRes} and Section \ref{sec:beam_model}. In fact, as pointed out in \cite[Sec. 2.3.2]{weiss99}, one may see \eqref{eq:GEB} and \eqref{eq:IGEB} as being related by the nonlinear transformation (see \eqref{eq:single_beam_VWPhiPsi})
\begin{equation} \label{eq:transfo}
\mathcal{T} \colon (\mathbf{p}, \mathbf{R}) \longmapsto y= 
\begin{bmatrix} V\\ W\\ \Phi\\ \Psi\end{bmatrix}.
%\left[\begin{smallmatrix} V\\ W\\ \Phi\\ \Psi\end{smallmatrix}\right].
\end{equation}
Considering the IGEB model raises the number of governing equations from six to twelve, but with the advantage of dealing with a first-order hyperbolic system (as $A(x)$ is an hyperbolic matrix\footnote{All eigenvalues of $A(x)$ are real and one may find $12$ associated independent eigenvectors.}) which is only semilinear;
and a large literature -- beyond the context of beam models -- exists on such models. In particular, a systematic study of one-dimensional hyperbolic systems -- well-posedness, control, stabilization -- has been developed by Li \cite{Li_Duke85} and Bastin \& Coron \cite{BC2016}.

\subsection{Our contributions}
\label{subsec:contrib}

In this article, we are concerned with tree-shaped networks of freely vibrating geometrically exact beams. \textcolor{black}{Such networks have not yet been considered in the literature.}
Each beam's dynamics are governed by the IGEB model, which is of the form \eqref{eq:IGEB}, and the beams are connected through rigid joints.
We investigate the local in time well-posedness of the system and, in the case of star-shaped networks, the exponential stabilization of steady states by means of velocity feedback controls applied at the nodes.
The importance of this type of study lies in the need for engineering to eliminate vibrations in such structures \cite{Matsuoka1995, UchiyamaKonno1991}.
More precisely, in this article, 
\begin{itemize}
\item from the continuity of displacements and the balance of forces/moments at the joint, we derive the transmission conditions for a tree-shaped network of beams governed by the IGEB model (see Subsection \ref{subsec:derivation_nodalCond});
\item in Theorem \ref{thm:well-posedness}, we show that the network system (given by \eqref{eq:syst_y} below), with or without feedback, admits a unique local in time solution in $C^0_t H^1_x$ for $H^1$ initial data (resp. $C^0_t H^2_x $ for $H^2$ initial data and more regular coefficients) -- such a solution is consequently $C^0_{x,t}$ (resp. $C^1_{x,t}$);
\item in Theorem \ref{thm:stabilization}, for a star-shaped network, we show that if velocity feedback controls are applied at all external nodes, then the zero steady state is locally exponentially stable for the $H^1$ and $H^2$ norms.
\end{itemize}
We stress that this work also provides an extension of the stabilization study realized for a single beam in our previous work \cite{RL2019} to a wider class of beams and velocity feedback controls (see also Remark \ref{rem:extention}). More precisely, concerning the former point, the formulation accounts for material anisotropy and varying material/geometrical properties along the beam -- as, here, we consider the general IGEB model of \cite{hodges2003geometrically}.
%accounting for any slender beam made of a linear elastic material.
%%% (meaning that the beam may be anisotropic, with geometrical coupling)

\subsection{Outline} 
The next section (Section \ref{sec:model_mainRes}) unfolds as follows.
In Subsection \ref{subsec:the_model}, we start by presenting the system describing the network (System \eqref{eq:syst_y} below), before stating the main results in Subsection \ref{subsec:main_results}.

Then, the aim of Section \ref{sec:beam_model} is to clarify the meaning of the different elements of the model.
% (unknowns, coefficients, nodal conditions).
In Subsection \ref{subsec:beam_description}, we explain how the beam is described, thus clarifying the meaning of the unknown states of the GEB and IGEB models. Then, in Subsection \ref{subsec:derivation_nodalCond}, we derive the nodal conditions of System \eqref{eq:syst_y}.

Section \ref{sec:riem_wellp} is concerned with the well-posedness result. We start by writing \eqref{eq:syst_y} in \emph{diagonal form} in Subsection \ref{subsec:riem}, before proving Theorem \ref{thm:well-posedness} in Subsection \ref{subsec:wellposedness}.

Finally, Section \ref{sec:stab} is centered on the stabilization result.
In Subsection \ref{subsec:stab_P}, we prove Theorem \ref{thm:stabilization} from the point of view of \eqref{eq:syst_y} -- that is, the \emph{physical system}. Afterwards, in Subsection \ref{subsec:stab_D}, we discuss on this proof seen from the point of view of the diagonal system.

\subsection{Notation}
Let $m, n, k \in \mathbb{N}$ and $M \in \mathbb{R}^{n\times n}$. Here, the identity and null matrices are denoted by $\mathds{I}_n \in \mathbb{R}^{n \times n}$ and $\mathds{O}_{n, m} \in \mathbb{R}^{n \times m}$, and we use the abbreviation $\mathds{O}_{n} = \mathds{O}_{n, n}$. The transpose and determinant of $M$ are denoted by $M^\intercal$ and $\mathrm{det}(M)$. By $\|M\|$, we denote the operator norm induced by the Euclidean norm $|\cdot |$.
The symbol $\mathrm{diag}(\, \cdot \, , \ldots, \, \cdot \, )$ denotes a (block-)diagonal matrix composed of the arguments.

\section{The model and main results}
\label{sec:model_mainRes}

\subsection{The model}
\label{subsec:the_model}

We start by introducing some notation inspired by \cite{AlabauPerrollazRosier2015}.
Consider an oriented tree containing $N$ edges. The \emph{edges} are indexed by $i \in \mathcal{I} = \{1, \ldots, N\}$, while the \textit{nodes} are indexed by $n \in \mathcal{N} = \{0, \ldots, N\}$, and we interchangeably use the expressions "node of index $n$" (resp. "edge of index $i$"), and "node $n$" (resp. "edge $i$") for short.
The set of nodes is partitioned as $\mathcal{N} = \mathcal{N}_S \cup \mathcal{N}_M$, where $\mathcal{N}_S$ and $\mathcal{N}_M$ are the set of \emph{simple} and \emph{multiple} nodes, respectively.

For any $i \in \mathcal{I}$, the $i$-th edge, of length $\ell_i$, is identified with the interval $[0, \ell_i]$ whose endpoints $x=0$ and $x=\ell_i$ are called \textit{initial point} and \textit{ending point} of this edge.
Without loosing generality, we assume that the node $n=0$ is a simple node and is the initial point of the edge $i=1$, and we assume that for any $i \in \mathcal{I}$ the edge $i$ has for ending point the node with the same index $n=i$. We refer to Fig. \ref{fig:tree_star} for visualization.

For any node $n \in \mathcal{N}$, we denote by $k_n$ the number of edges incident to this node.
For any multiple node $n \in \mathcal{N}_M$, we denote by $\mathcal{I}_n$ the set of indices of all edges starting at this node; we also denote the elements of $\mathcal{I}_n$ by (see Fig. \ref{fig:notation_mult_node})
\begin{equation} \label{eq:nota_In_indices}
\mathcal{I}_n = \{i_2, i_3, \ldots, i_{k_n}\}, \qquad \text{with} \quad i_2 < i_3 < \ldots < i_{k_n}.
\end{equation}
Note that, for $\#S$ denoting the cardinality of any set $S$, 
\begin{equation} \label{eq:def_kn}
k_n = 
\begin{cases}
 1& \text{if }n\in \mathcal{N}_S,\\
\#\mathcal{I}_n + 1, & \text{if }n\in \mathcal{N}_M.
\end{cases}
\end{equation}
The purpose of this tree is to specify how a collection of $N$ beams are connected to each other, namely, which beam is connected to which beam and at which endpoint ($x=0$ or $x = \ell_i$, for $i \in \mathcal{I}$).

\begin{figure}
\includegraphics[scale=0.8]{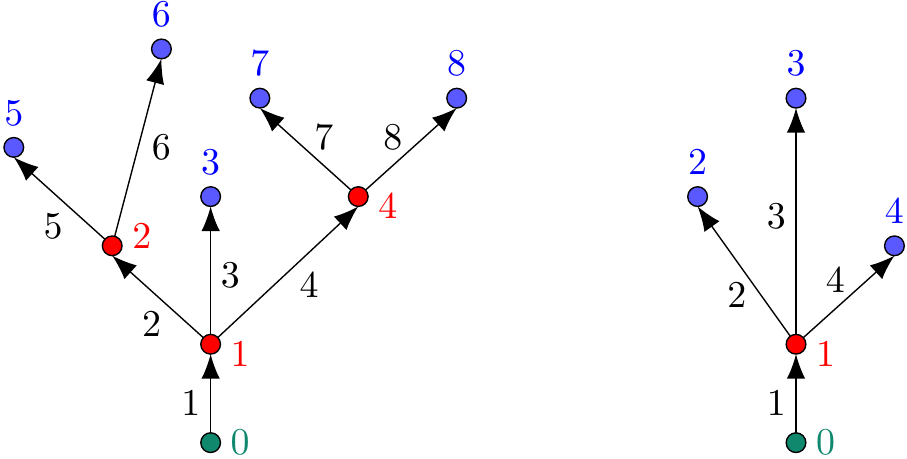}
\caption{Left: tree-shaped network with $N=8$ edges,
% a depth equal to 3, 
$\mathcal{N}_S = \{0, 3, 5, 6, 7\}$ and $\mathcal{N}_M = \{ 1, 2, 4 \}$. Right: star-shaped network with $N=4$ edges,
% a depth equal to 2, 
$\mathcal{N}_S = \{0, 2, 3\}$ and $\mathcal{N}_M = \{ 1 \}$.}
\label{fig:tree_star}
\end{figure}

Let $i \in \mathcal{I}$. To the $i$-th edge corresponds a beam characterized by its length $\ell_i>0$, the so-called \emph{mass matrix} $\mathbf{M}_i$ and \emph{flexibility matrix} $\mathbf{C}_i$, and the \emph{initial curvature-twist matrix} $\mathbf{E}_i$. While\footnote{We may assume that $\mathbf{M}_i, \mathbf{C}_i, \mathbf{E}_i$ are of higher regularity: $C^k([0, \ell_i]; \mathbb{R}^{6\times 6})$ for $k\geq 2$.} $\mathbf{M}_i, \mathbf{C}_i \in C^1([0, \ell_i]; \mathbb{R}^{6 \times 6})$ depend on the geometry and material of the beam, $\mathbf{E}_i \in C^1([0, \ell_i]; \mathbb{R}^{6 \times 6})$ depends on the initial form of the beam. For any $x \in [0, \ell_i]$, the matrix $\mathbf{E}_i(x)$ is indefinite (see \eqref{eq:def_E_bold} for details), and $\mathbf{M}_i(x)$ and $\mathbf{C}_i(x)$ are both assumed positive definite.
This beam -- of index $i$ -- is described by the unknown state $y_i \colon [0, \ell_i]\times [0, T] \rightarrow \mathbb{R}^{12}$ which has the form \begin{equation}\label{eq:unknown_physical}
y_i = \begin{bmatrix}
v_i \\ z_i
\end{bmatrix}.
\end{equation}
It consists of the \emph{linear and angular velocities} $v_i \colon [0, \ell_i]\times [0, T] \rightarrow \mathbb{R}^6$ and the \emph{internal forces and moments} $z_i \colon [0, \ell_i]\times [0, T] \rightarrow \mathbb{R}^6$ of the beam. The precise meaning of these variables is given in Section \ref{sec:beam_model}. 

\begin{figure}
\includegraphics[scale=0.8]{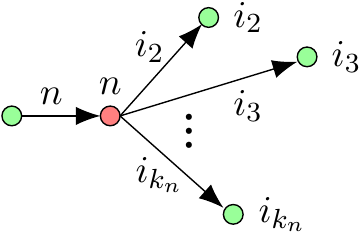}
\caption{A multiple $n$ and the incident edges and nodes.}
\label{fig:notation_mult_node}
\end{figure}

For the network, the unknown state, $y = (y_1, y_2, \ldots, y_N)$, consists of the unknowns $y_i$ ($i \in \mathcal{I}$) of the $N$ beams contained in this network.
Its dynamics are given by the system
\begin{equation} \label{eq:syst_y}
\begin{dcases}
\partial_t y_i + A_i(x) \partial_x y_i + \overline{B}_i(x) y_i = \overline{g}_i (x,y_i) & \text{in }(0, \ell_i) \times(0, T), \ i \in \mathcal{I}\\
\overline{R}_{i}(0) v_i(0, t) = \overline{R}_{n}(\ell_n) v_n(\ell_n, t) & t \in (0,T), \ i \in \mathcal{I}_n, \ n \in \mathcal{N}_M\\
\overline{R}_{n}(\ell_n) z_n(\ell_n, t) - {\textstyle \sum_{i \in \mathcal{I}_n}} \overline{R}_{i}(0) z_i(0, t) \\
\hspace{2.175cm} = - \overline{R}_n(\ell_n) K_n v_n(\ell_n, t) & t \in (0,T), \ n \in \mathcal{N}_M \\
z_n(\ell_n, t) = - K_n v_n(\ell_n, t) & t \in (0,T), \ n \in \mathcal{N}_S \setminus \{0\}\\
z_1(0, t) = K_0 v_1(0, t) & t \in (0,T)\\
y_i(x,0) = y_i^0(x) & x \in (0, \ell_i), \ i \in \mathcal{I}.
\end{dcases}
\end{equation}
Let us now describe this system, starting with the governing equations.
Let $i \in \mathcal{I}$. Each beam's dynamics are governed by the IGEB model which has been briefly described in Section \ref{sec:intro} for a single beam -- here a subindex $i$ is added to the coefficients $A_i$, $\overline{B}_i$ and source $\overline{g}_i$, since the beams may have different material/geometrical properties and initial forms. As mentioned before, this is a system of twelve equations forming a one-dimensional semilinear hyperbolic system.
The coefficients $A_i, \overline{B}_i \in C^1([0, \ell_i];\mathbb{R}^{12 \times 12})$ are defined by
\begin{equation} \label{eq:def_Ai_barBi}
A_i = \begin{bmatrix}
\mathds{O}_6 & -\mathbf{M}_i^{-1}\\
-\mathbf{C}_i^{-1} & \mathds{O}_6
\end{bmatrix}, \qquad \overline{B}_i = \begin{bmatrix}
\mathds{O}_6 & - \mathbf{M}^{-1}_i\mathbf{E}_i\\
\mathbf{C}_i^{-1}\mathbf{E}_i^\intercal & \mathds{O}_6
\end{bmatrix}.
\end{equation}
In these definitions, one observes that, while both $A_i$ and $\overline{B}_i$ depend on the geometry and material of the beam, $\overline{B}_i$ also depends on the initial form of the beam. 
Latter on, we will see that $A_i(x)$ is hyperbolic for all $x \in [0, \ell_i]$. 
As we also pointed out earlier, for any $x \in [0, \ell_i]$, the matrix $\overline{B}_i(x)$ is indefinite and, up to the best of our knowledge, may not be assumed arbitrarily small.
The nonlinear function $\overline{g}_i \in C^1([0, \ell_i]\times \mathbb{R}^{12}; \mathbb{R}^{12})$ is defined by 
\begin{linenomath}
\begin{equation*}
\overline{g}_i(x, \mathbf{u}) = \overline{\mathcal{G}}_i(x, \mathbf{u})\mathbf{u},
\end{equation*}
\end{linenomath}
for all $x \in [0, \ell_i]$ and $ \mathbf{u}=(\mathbf{u}_1^\intercal, \mathbf{u}_2^\intercal, \mathbf{u}_3^\intercal,\mathbf{u}_4^\intercal)^\intercal \in \mathbb{R}^{12}$, with $\mathbf{u}_j \in \mathbb{R}^3$ for $j\in\{1, \ldots, 4\}$, where the function $\overline{\mathcal{G}}_i \colon [0, \ell_i]\times \mathbb{R}^{12} \rightarrow \mathbb{R}^{12}$ is defined by (see Footnote \ref{foot:cross_prod})
\begin{linenomath}
\begin{equation*}
\overline{\mathcal{G}}_i(x,\mathbf{u}) = - \mathrm{diag}(\mathbf{M}_i(x) , \mathbf{C}_i(x))^{-1} \begin{bmatrix}
\widehat{\mathbf{u}}_2 & \mathds{O}_3 & \mathds{O}_3 & \widehat{\mathbf{u}}_3\\
\widehat{\mathbf{u}}_1 & \widehat{\mathbf{u}}_2 & \widehat{\mathbf{u}}_3 & \widehat{\mathbf{u}}_4 \\
\mathds{O}_3 & \mathds{O}_3 & \widehat{\mathbf{u}}_2 & \widehat{\mathbf{u}}_1\\
\mathds{O}_3 & \mathds{O}_3 & \mathds{O}_3 & \widehat{\mathbf{u}}_2
\end{bmatrix} \mathrm{diag}(\mathbf{M}_i(x) , \mathbf{C}_i(x)).
\end{equation*}
\end{linenomath}
Note that $\overline{g}_i$ is quadratic and $C^\infty$ with respect to the argument $\mathbf{u}$ (see also Remark \ref{rem:form_sources}). While $\bar{g}_i(x,\cdot)$ is locally Lipschitz in $\mathbb{R}^{12}$ for any $x \in [0, \ell_i]$, and $\bar{g}_i$ is locally Lipschitz in $H^1(0, \ell_i; \mathbb{R}^{12})$, no global Lipschitz property is available.

Let us now describe the nodal conditions, which are derived in Section \ref{sec:beam_model}.
We start with the transmission conditions for the multiple nodes. At these nodes, it is assumed that the beams remain attached to each other through time and without rotating -- i.e. we consider \textit{rigid} joints. For the variables $y_i$ ($i \in \mathcal{I}$) these assumptions amount to imposing the following \textit{continuity conditions}: for all $n \in \mathcal{N}_M$,
\begin{equation} \label{eq:cont_cond}
\overline{R}_{i}(0) v_i(0, t) = \overline{R}_{n}(\ell_n) v_n(\ell_n, t), \qquad \text{for all }i \in \mathcal{I}_n, \ t \in [0, T].
\end{equation}
Above, for any $i \in \mathcal{I}$, the function $\overline{R}_{i}\in C^2([0, \ell_i], \mathbb{R}^{6 \times 6})$ is defined by
\begin{equation} \label{eq:def_barRi}
\overline{R}_{i} = \mathrm{diag}(R_{i}, R_{i}).
\end{equation}
where $R_{i}\in C^2([0, \ell_i], \mathrm{SO}(3))$ depends on the initial form of the beam (see Section \ref{sec:beam_model}) -- this function was denoted $R$ for a single beam in Section \ref{sec:intro}.
Additionally, for all $n \in \mathcal{N}_M$,
\begin{equation} \label{eq:Kirchhoff_cond}
\overline{R}_{n}(\ell_n) z_n(\ell_n, t) - \sum_{i \in \mathcal{I}_n} \overline{R}_{i}(0) z_i(0, t) = - \overline{R}_n(\ell_n) K_n v_n(\ell_n, t), \quad \text{for all }t \in [0, T]
\end{equation}
provides the condition of balance of forces/moments, also called \emph{Kirchhoff condition}.
In \eqref{eq:Kirchhoff_cond}, two situations may be accounted for: either  a feedback is applied at this node, in which case $K_n \in \mathbb{R}^{6 \times 6}$ is a positive definite symmetric matrix; or $K_n = \mathds{O}_6$ and no external load is applied at this node -- the latter corresponds to the classical Kirchhoff condition. 
At simple nodes $n\in\mathcal{N}_S$, either the velocity feedback control
\begin{linenomath}
\begin{align} 
\label{eq:neum_cond_n}
z_n(\ell_n, t) &=- K_n v_n(\ell_n, t),\qquad \text{for all }t \in [0, T], \text{ if }n\neq 0,\\
z_1(0, t) &= K_0 v_1(0, t), \qquad \quad \ \, \text{for all }t \in [0, T], \text{ if }n= 0 \label{eq:neum_cond_0}
\end{align}
\end{linenomath}
%%%%\begin{eqnarray} 
%%%%\label{eq:neum_cond_n}
%%%%&z_n(\ell_n, t) =- K_n v_n(\ell_n, t),\quad &\text{for all }t \in [0, T], \text{ if }n\neq 0,\\
%%%%&z_1(0, t) = K_0 v_1(0, t), \qquad &\text{for all }t \in [0, T], \text{ if }n= 0 \label{eq:neum_cond_0}
%%%%\end{eqnarray}
is applied, where $K_n \in \mathbb{R}^{6 \times 6}$ is positive definite and symmetric; or $K_n = \mathds{O}_6$ in \eqref{eq:neum_cond_n}, \eqref{eq:neum_cond_0} and the beam is \emph{free}.
Instead of \eqref{eq:neum_cond_n} or \eqref{eq:neum_cond_0}, one may want to assume that the beam is clamped at this node, which would amount to considering the respective homogeneous Dirichlet conditions
\begin{equation}\label{eq:diri_cond_0n}
v_n(\ell_n, t) = 0, \quad \text{or} \quad v_1(0, t) = 0, \qquad \text{for all }t \in [0, T].
\end{equation}

Finally, the last equation in \eqref{eq:syst_y} describes the initial conditions, with initial datum $y^0 = (y_1^0, y_2^0, \ldots, y_N^0)$ for the whole network.

\subsection{Main results}
\label{subsec:main_results}

We will need to define compatibility conditions for System \eqref{eq:syst_y}. As for the unknown, we write the initial datum $y^0 = (y_1^0, \ldots, y_N^0)$ as
\begin{linenomath}
\begin{equation*}
y_i^0 = \begin{bmatrix}
v_i^0 \\ z_i^0
\end{bmatrix}, \qquad \text{with } v_i^0, z_i^0 \colon [0, \ell_i] \rightarrow \mathbb{R}^6, \quad \text{for all }i \in \mathcal{I}.
\end{equation*}
\end{linenomath}
Let us denote $\mathbf{H}^k_x = \prod_{i=1}^N H^k(0, \ell_i; \mathbb{R}^{12})$, endowed with the associated product norm, for any $k \geq 1$.

\begin{definition}
\label{def:comp_cond}
We say that the initial datum $y^0 \in \mathbf{H}^1_x$ fulfills the zero-order compatibility conditions of System \eqref{eq:syst_y} if
\begin{equation} \label{eq:compat_0} 
\begin{aligned}
& (\overline{R}_{i} v_i^0)(0) = (\overline{R}_{n} v_n^0)(\ell_n), && \text{for }i \in \mathcal{I}_n, \ n \in \mathcal{N}_M,\\
&(\overline{R}_{n} z_n^0)(\ell_n) - \sum_{i \in \mathcal{I}_n} (\overline{R}_{i} z_i^0)(0) = - (\overline{R}_n K_n v_n^0)(\ell_n), \quad && \text{for }n \in \mathcal{N}_M,\\
&z_n^0(\ell_n) = - K_n v_n^0(\ell_n),  &&\text{for } n \in \mathcal{N}_S \setminus \{0\},\\
&z_1^0(0) = K_0 v_1^0(0), && \text{for }n=0,
\end{aligned}
\end{equation} 
holds. We say that $y^0\in \mathbf{H}^2_x$ fulfills the first-order compatibility conditions of \eqref{eq:syst_y} if it fulfills \eqref{eq:compat_0} and, $y^1 \in \mathbf{H}_x^1$ defined by 
\begin{linenomath}
\begin{equation*}
y_i^1 = -  A_i \frac{\mathrm{d}y_i^0}{\mathrm{d}x} - \overline{B}_i y_i^0 + \overline{g}_i(\cdot, y_i^0) = \begin{bmatrix}
v_i^1 \\z_i^1
\end{bmatrix}, \qquad \text{for all }i\in\mathcal{I},
\end{equation*}
\end{linenomath}
also fulfills \eqref{eq:compat_0}, where $v_i^0, z_i^0$ are replaced by $v_i^1, z_i^1$ respectively. 
\end{definition}

%%% CUT £££
%%%In other words, the zero-order compatibility conditions of \eqref{eq:syst_y} require that the initial datum fulfill the boundary conditions of \eqref{eq:syst_y}, while the first-order compatibility condition additionally require that the "time derivative of the initial datum", interpreted as $y^1$, also fulfill the boundary conditions of \eqref{eq:syst_y}.

We now make an assumption on the mass and flexibility matrices, to ensure a certain regularity of the eigenvalues and eigenvectors of $\{A_i\}_{i \in \mathcal{I}}$ with respect to $x$.

\begin{assumption} \label{as:mass_flex}
Let $m \in \{1, 2, \ldots\}$ be given. For any $i \in \mathcal{I}$, let $\Theta_i \colon [0, \ell_i] \rightarrow \mathbb{R}^{6\times 6}$ be defined by
\begin{equation} \label{eq:def_Thetai}
\Theta_i = \mathbf{C}_i^{-\sfrac{1}{2}} \mathbf{M}_i^{-1}\mathbf{C}_i^{-\sfrac{1}{2}};
\end{equation}
it has values in the set of positive definite symmetric matrices (since $\mathbf{M}_i$ and $\mathbf{C}_i$ both also have values in this set). We make the following two assumptions:
\begin{enumerate}[label=\alph*)]
\item \label{item:MCassump_reg}
the mass and flexibility matrices have the regularity 
\begin{equation}\label{eq:reg_mass_flex}
\mathbf{C}_i, \mathbf{M}_i \in C^m([0, \ell_i]; \mathbb{R}^{6 \times 6}), \qquad \text{for all }i\in \mathcal{I};
\end{equation}
in which case $\Theta_i \in C^m([0, \ell_i]; \mathbb{R}^{6 \times 6})$ for all $i \in \mathcal{I}$;

\item %\label{item:MCassump_Thetai}
there exists $U_i$ and $D_i$, both of regularity $C^m([0, \ell_i]; \mathbb{R}^{6 \times 6})$, such that
\begin{equation} \label{eq:diag_Thetai}
\Theta_i(x) = U_i(x)^\intercal D_i(x)^2 U_i(x), \qquad \text{for all }x \in [0, \ell_i],
\end{equation}
where $D_i(x)$ is a positive definite diagonal matrix containing the square roots of the eigenvalues of $\Theta_i(x)$ as diagonal entries, while $U_i(x)$ is a unitary matrix.
\end{enumerate}
\end{assumption}

Whenever $\mathbf{M}_i$ and $\mathbf{C}_i$ ($i\in\mathcal{I}$) fulfill \eqref{eq:reg_mass_flex}, Assumption \ref{as:mass_flex} \ref{item:MCassump_reg} is readily verified if $\mathbf{M}_i, \mathbf{C}_i$ have values in the set of diagonal matrices, or if the eigenvalues of $\Theta_i(x)$ are distinct (one may adapt \cite[Th. 2, Sec. 11.1]{evans2}), for all $i\in \mathcal{I}$ and $x \in [0, \ell_i]$. So is it, clearly, if the mass and flexibility are both constant, meaning that the material and geometrical properties do not vary along the beam.

For any \emph{tree-shaped} network as above, we show the following (local in time) well-posedness result.

\begin{theorem}[Well-posedness for tree-shaped networks] \label{thm:well-posedness}
Let $k \in \{1, 2\}$, suppose that Assumption \ref{as:mass_flex} is fulfilled for $m=k+1$, and assume that $\mathbf{E}_i \in C^k([0, \ell_i]; \mathbb{R}^{6 \times 6})$ for all $i\in\mathcal{I}$. Then, there exists $\delta_0>0$ such that for any $y^0 \in \mathbf{H}^k_x$ satisfying $\|y^0\|_{\mathbf{H}^k_x} \leq \delta_0$ and the $(k-1)$-order compatibility conditions, there exists a unique solution $y \in C^0([0, T), \mathbf{H}^k_x)$ to \eqref{eq:syst_y}, with $T\in (0, +\infty]$. Moreover, if $\|y(\cdot, t)\|_{\mathbf{H}^k_x} \leq \delta_0$ for all $t \in [0, T)$ then $T = + \infty$.
\end{theorem}

\begin{remark}
Theorem \ref{thm:well-posedness} also holds if the beam is clamped at one or several simple nodes, provided that the compatibility conditions \eqref{eq:compat_0} are accordingly changed.
\end{remark}

The proof of Theorem \ref{thm:well-posedness}, given in Section \ref{sec:riem_wellp}, is based on existing results on first-order hyperbolic systems -- the local existence and uniqueness of $C_t^0 H_x^k$ solutions to general one-dimensional semilinear ($k=1$) and quasilinear ($k=2$) hyperbolic systems, which have been addressed by Bastin \& Coron \cite{BC2016, bastin2017exponential}. 
Such results require a certain regularity of the coefficients as well as a specific form of the boundary conditions for the system written in diagonal form (also called \emph{characteristic form} or \emph{Riemann invariants}); namely, the so-called \emph{outgoing information} should be explicitly expressed as a function of the \emph{incoming information} (see Section \ref{sec:riem_wellp}). The main point of the proof of Theorem \ref{thm:well-posedness} is consequently to write \eqref{eq:syst_y} in Riemann invariants and study its transmission conditions.

Next, we consider a stabilization problem, in the sense of the following definition.

\begin{definition}[Local exponential stability]
Let $k \in \{1, 2\}$. The steady state $y\equiv 0$ of \eqref{eq:syst_y} is locally $\mathbf{H}^k_x$ exponentially stable if there exist $\varepsilon>0$, $\beta>0$ and $\eta \geq 1$ such that the following holds. Let  $y^0 \in \mathbf{H}^k_x$ fulfill both $\|y^0\|_{\mathbf{H}^k_x}\leq \varepsilon$ and the $(k-1)$-order compatibility conditions. Then, there exists a unique global in time solution $y \in C^0([0, +\infty); \mathbf{H}^k_x)$ to \eqref{eq:syst_y}. Moreover,
\begin{linenomath}
\begin{equation*}
\|y(\cdot, t)\|_{\mathbf{H}^k_x} \leq \eta  e^{- \beta t } \|y^0\|_{\mathbf{H}^k_x}, \qquad \text{for all } \, t \in [0, +\infty).
\end{equation*}
\end{linenomath}
\end{definition}

For a \emph{star-shaped} network (i.e. $\mathcal{N}_M = \{1\}$) such that the multiple node is free (i.e. $K_1 = \mathds{O}_6$) while velocity feedback controls are applied at all simple nodes (i.e. $K_n$ is symmetric positive definite for all $n \in \mathcal{N}_S$), we show the following result.

\begin{theorem}[Stabilization for star-shaped networks] \label{thm:stabilization}
Let $k \in \{1, 2\}$, suppose that Assumption \ref{as:mass_flex} is fulfilled for $m=k+1$, and that $\mathbf{E}_i \in C^k([0, \ell_i]; \mathbb{R}^{6 \times 6})$ for all $i\in\mathcal{I}$. If $\mathcal{N}_M = \{1\}$, $K_1= \mathds{O}_6$, and $K_n$ is symmetric positive definite for all $n \in \mathcal{N}_S$, then the steady state $y \equiv 0$ of \eqref{eq:syst_y} is locally $\mathbf{H}^k_x$ exponentially stable.
\end{theorem}

To prove Theorem \ref{thm:stabilization}, in Section \ref{sec:stab}, in the spirit of Bastin \& Coron \cite{BC2016, bastin2017exponential}, we find a so-called \emph{quadratic Lyapunov functional}: we work directly with the \emph{physical system} \eqref{eq:syst_y} (instead of this system in diagonal form), and use our knowledge of the energy of the beam and the coefficients $A_i, \bar{B}_i$ $(i\in\mathcal{I})$ to choose this functional.

\begin{remark} \label{rem:extention}
As pointed out in Subsection \ref{subsec:contrib}, Theorem \ref{thm:stabilization} yields an extension of our previous work \cite{RL2019}, when one considers a single beam $($i.e. $\mathcal{N}_M = \emptyset)$ clamped or free $($i.e. $K_0 = \mathds{O}_6)$ at $x = 0$, with a control applied at $x = \ell_1$ $($i.e. $K_1$ symmetric positive definite$)$.
This previous article was concerned with prismatic, isotropic beams, with principal axis aligned with the body-attached basis -- which amounts to assuming that $\mathbf{M}_i, \mathbf{C}_i$ $(i \in \mathcal{I})$ are, in addition, both \emph{constant} and \emph{diagonal}.
\end{remark}

\subsection{Brief state of the art}

Numerous works have been carried out on the stabilization of tree-shaped networks of d'Alembert wave equations \cite{DagerZuazuaBook}, by means of velocity feedback controls applied at some nodes. 
In \cite{ValeinZuazua2009, Zuazua2012}, the control is located in one single simple node and stability properties (e.g. polynomial) are proved by making use of suitable observability inequalities. 
In \cite{NicaiseValein2007} the exponential stabilization is obtained by applying velocity feedback controls with delay at the multiple nodes.
In \cite{AlabauPerrollazRosier2015}, the authors apply transparent boundary conditions at all simple nodes in addition to velocity feedback controls at the multiple nodes, in order to obtain finite time stabilization. For the exponential stabilization of star-shaped networks using spectral methods, we refer to \cite{GuoXu2011} where the controls are applied at the multiple nodes and all simple nodes but one, and \cite{ZhangXu2013} where the controls are applied at all simple nodes but one. In \cite{GugatSigalotti2019}, the authors showed that, for star-shaped networks, finite time stability is achieved by applying velocity feedback controls at all simple nodes, and that exponential stability is still achieved if one of the controls is removed from time to time.
Applying the controls at all simple nodes but one, \cite{LLS} also studies the exponential stabilization of tree-shaped networks of strings, as well as that of beams.

The stabilization of beam networks has also been considered by \cite{HanXu2011}, who applied time-delay controls at all the simple nodes of a star-shaped network of Timoshenko beams, to obtain exponential stability via spectral methods.
In \cite{ZhangXuMastorakis2009}, by means of semigroup theory and spectral analysis, the exponential stability of a tree-shaped network of Euler-Bernoulli beams is proved, when all simple nodes are clamped while velocity feedback controls are applied at the interior nodes. 
Also using spectral methods to study exponential stabilization, \cite{XuHanYung2007} considered serially connected Timoshenko beams, applying velocity feedback controls at all nodes except one simple node, while \cite{HanXu2010} considered a specific star-shaped network of Timoshenko beams, where velocity feedback controls are applied all simple nodes but one.
The interested reader is also referred to the references therein \cite{HanXu2010, HanXu2011, XuHanYung2007, ZhangXuMastorakis2009}.

Stabilization problems for networks of first order hyperbolic systems have also been extensively studied, in particular for the Saint-Venant equations -- e.g. \cite{AlabauPerrollazRosier2015}, as well as \cite{coron2003} and \cite{LeugeringSchmidt2002} which both make use of the Li-Greenberg Theorem \cite[Chap. 5, Th. 1.3]{Li_blue_book} to obtain the exponential decay result.
Since the tree-shaped network system may be rewritten as a single hyperbolic system -- as done for example in \cite{Coron2007} where exponential stabilization is then proved by means of a Lyapunov functional -- the literature on such systems is also of interest here, see \cite{BC2016,bastin2017exponential, gugat2018, hayat2018exponential, HertyYu2018, Li_Duke85}.

\section{Mechanical setting and derivation of nodal conditions}
\label{sec:beam_model}

Let us clarify the definition of the unknowns, coefficients, and the derivation of the nodal conditions.

\subsection{Description of the beam}
\label{subsec:beam_description}

Let $\{e_j\}_{j=1}^3$ $= \{(1, 0, 0)^\intercal$, $(0, 1, 0)^\intercal$, $(0, 0, 1)^\intercal\}$. 
Let $T>0$, $i \in \mathcal{I}$ and consider the $i$-th beam.

This beam is idealized as a \emph{reference line} -- that we also call \emph{centerline} -- and a family of \emph{cross sections}.
At rest, before deformation, the position of the centerline $p_i \colon [0, \ell_i] \rightarrow \mathbb{R}^3$ and the orientation of the cross sections, are both known. The latter is given by the columns $\{b_i^j\}_{j=1}^3$ of a rotation matrix $R_i \colon [0, \ell_i] \rightarrow \mathrm{SO}(3)$. We assume that $b_i^1 = \frac{\mathrm{d}p_i}{\mathrm{d}x}$, implying that $p_i$ is parametrized by its arclength. 
At any time $t>0$, the position $\mathbf{p}_i \colon [0, \ell]\times [0, T] \rightarrow \mathbb{R}^3$ of the centerline and the orientation of the cross sections, given by the columns $\{\mathbf{b}_i^j\}_{j=1}^3$ of a rotation matrix $\mathbf{R}_i \colon [0, \ell_i]\times[0, T] \rightarrow \mathrm{SO}(3)$, are both unknown.
As shear deformation is allowed, $\mathbf{b}_i^1$ is not necessarily tangent to the centerline.

Let $\Omega_{s}^i \subset \mathbb{R}^3$ be the straight, untwisted beam whose centerline is located at $x e_1$ for $x \in [0, \ell_i]$; it may be written as $\Omega_{s}^i = \bigcup_{x \in [0,\ell]}\mathfrak{a}_i(x)$ where $\mathfrak{a}_i(x)$ is the cross section intersecting the centerline at $x e_1$.
Then, the beam before deformation takes the form $\Omega_c^i = \{\overline{p}_i(X) \colon X \in \Omega_s^i\}$ while the beam at time $t>0$ takes the form $\Omega_t^i = \{\overline{\mathbf{p}}_i(X, t) \colon X \in \Omega_s^i\}$, where $\bar{p}_i$ and $\bar{\mathbf{p}}_i$ are defined by $\overline{p}_i(X) = p_i(x) + R_i(x)(\zeta_2 e_2 + \zeta_3 e_3)$ and $\overline{\mathbf{p}}_i(X,t) = \mathbf{p}_i(x,t) + \mathbf{R}_i(x,t)(\zeta_2 e_2 + \zeta_3 e_3)$, using the notation $X = (x,\zeta_2,\zeta_3)^\intercal$ for any $X \in \Omega_{s}$.
We call $\Omega_s^i$, $\Omega_c^i$ and $\Omega_t^i$ the  \emph{straight-reference} configuration, \emph{curved-reference} configuration and \emph{current} configuration of the beam (see Fig. \ref{fig:config_beam}), respectively.

\begin{figure}
\includegraphics[scale=0.8]{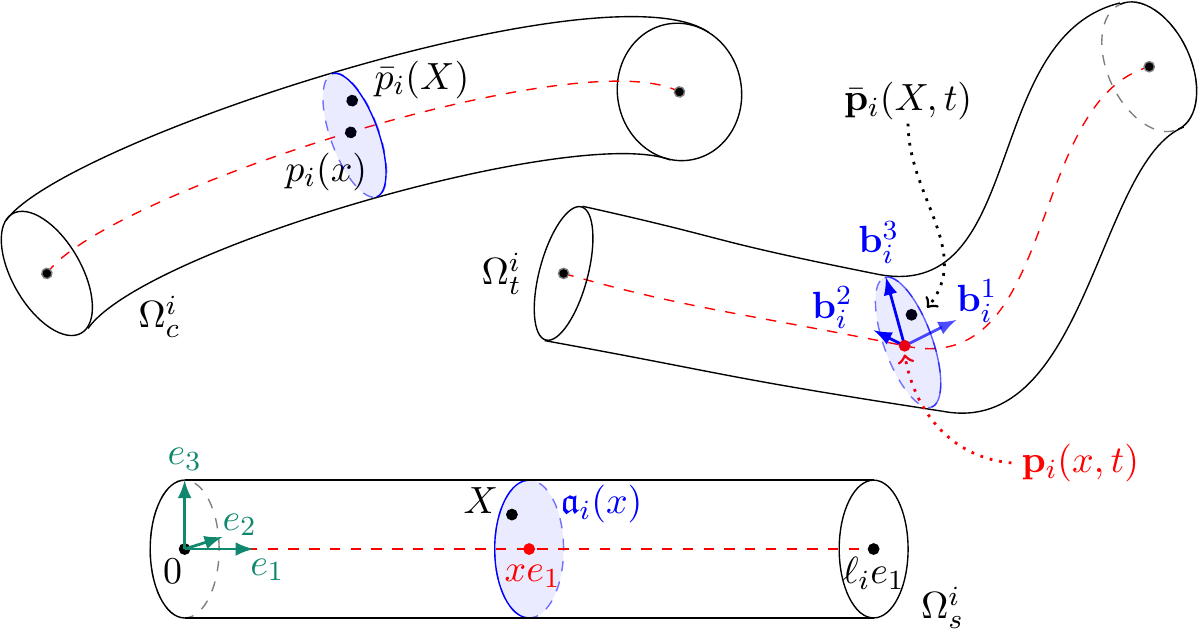}
\caption{The $i$-th beam in its different configurations $\Omega_s^i$, $\Omega_c^i$ and $\Omega_t^i$.}
\label{fig:config_beam}
\end{figure}

\begin{remark}[Body-attached variable]
\label{rem:body_attached_var}
The set $\{\mathbf{b}_i^j(x,t)\}_{j=1}^3$ can be seen as a body-attached (moving with time) basis, with origin $\mathbf{p}_i(x,t)$, for any $x \in [0, \ell_i]$ and $t\in [0, T]$. Hence, here, we consider two kinds of coordinate systems: $\{e_j\}_{j=1}^3$ which is fixed in space and time, and the body-attached basis $\{\mathbf{b}_i^j\}_{j=1}^3$. 
We then make the difference between two kinds of vectors in $\mathbb{R}^3$: global and body-attached.
Consider two vectors $u := \sum_{j=1}^3 u_j e_j$ and $U:=\sum_{j=1}^3 U_j e_j$ of $\mathbb{R}^3$, 
the former being a \emph{global} vector and the latter being the \emph{body-attached} representation of $u$. By this, we mean that the components of $u$ are its coordinates with respect to the global basis $\{e_j\}_{j=1}^3$, while the components of $U$ are coordinates of the vector $u$ with respect to the body-attached basis $\{\mathbf{b}_i^j\}_{j=1}^3$. In other words $u = \sum_{j=1}^3 U_j \mathbf{b}_i^j $. Both vectors are then related by the identity $u = \mathbf{R}_i U$ since $\mathbf{b}_i^j = \mathbf{R}_i e_j$, and we may also call $u$ the \emph{global} representation of $U$.
\end{remark}

In fact, the unknown state of the IGEB model is composed of such \emph{body-attached variables}.
We have seen in \eqref{eq:unknown_physical} that the unknown state $y_i$ of \eqref{eq:syst_y} consists of the velocities $v_i$ and internal forces/moments $z_i$. More precisely, they are (see also Section \ref{sec:intro})
\begin{equation} \label{eq:def_vi_si}
v_i = \begin{bmatrix}
V_i \\ W_i
\end{bmatrix}, \quad z_i = \begin{bmatrix}
\Phi_i \\ \Psi_i
\end{bmatrix}
\end{equation}
where $V_i, W_i, \Phi_i, \Psi_i \colon [0, \ell_i ]\times[0, T] \rightarrow \mathbb{R}^3$ are body-attached variables:
the \emph{linear velocity}, the \emph{angular velocity}, the \emph{internal forces} and the \emph{internal moments} of the beam, respectively.
Then, $v_i, z_i$ are related to $\mathbf{p}_i, \mathbf{R}_i$ as follows (see Footnote \ref{foot:cross_prod}):
\begin{equation}\label{eq:def_VWGU}
\begin{aligned}
V_i &= \mathbf{R}_i^\intercal \partial_t \mathbf{p}_i,\\
W_i &= \mathrm{vec}(\mathbf{R}_i^\intercal \partial_t\mathbf{R}_i),
\end{aligned} \qquad \quad
\begin{aligned}
\begin{bmatrix}
\Phi_i \\ \Psi_i
\end{bmatrix} = \mathbf{C}_i^{-1}  \begin{bmatrix}
\mathbf{R}_i^\intercal \partial_x \mathbf{p}_i  - e_1 \\ \mathrm{vec}\big(\mathbf{R}_i^\intercal \partial_x \mathbf{R}_i - R_i^\intercal \tfrac{\mathrm{d}}{\mathrm{d}x} R_i\big)
\end{bmatrix}.
\end{aligned}
\end{equation}
The \textit{initial curvature-twist matrix} $\mathbf{E}_i$, appearing in the definition of $\overline{B}_i$ (see \eqref{eq:def_Ai_barBi}), is defined by
\begin{equation} \label{eq:def_E_bold}
\mathbf{E}_i = \begin{bmatrix}
\widehat{\Upsilon}_c^i &  \mathds{O}_3\\
\widehat{e}_1 & \widehat{\Upsilon}_c^i
\end{bmatrix}, \qquad \text{with} \ \ \Upsilon_c^i = \mathrm{vec}\big(R_i^\intercal \tfrac{\mathrm{d}}{\mathrm{d}x} R_i \big),
\end{equation}
where $\Upsilon_c^i \colon [0, \ell_i] \rightarrow \mathbb{R}^3$ is the \emph{curvature} of the beam in the curved-reference configuration (i.e. before deformation).
If the beam is straight and untwisted with centerline $p_i(x) = xe_1$ before deformation, then $R_i$ is the identity matrix and $\Upsilon_c^i = 0$.

\subsection{Derivation of the nodal conditions}
\label{subsec:derivation_nodalCond}
Let us now explain how we derived the nodal conditions
\eqref{eq:cont_cond},\eqref{eq:Kirchhoff_cond},\eqref{eq:neum_cond_n},\eqref{eq:neum_cond_0} and \eqref{eq:diri_cond_0n}. Let $T>0$.

\subsubsection{The continuity condition}

Let $n \in \mathcal{N}_M$. 
It is assumed that incident beams -- which have indices in $\mathcal{I}_n \cup \{n\}$ -- stay attached and that the angles between them remain the same, at all times. In terms of positions and rotations, this writes as
\begin{linenomath}
\begin{align} 
\mathbf{p}_i(0, t)  = \mathbf{p}_n(\ell_n, t), \qquad \qquad \ \ \, \text{for all }t\in [0, T], \ i \in \mathcal{I}_n, \label{eq:cont_centerline}\\
\mathbf{R}_i(0, t) R_{i}(0)^\intercal = \mathbf{R}_n(\ell_n, t) R_{n}(\ell_n)^\intercal,  \quad \text{for all }t\in [0, T], \ i \in \mathcal{I}_n, \label{eq:rigid_joint}
\end{align}
\end{linenomath}
respectively (see also \cite{strohm_dissert}). Indeed, \eqref{eq:rigid_joint} translates to the fact that the change of angle $\mathbf{R}_i R_i^\intercal$ between the curved-reference and current configurations is the same for all incident beams.
Differentiating in time \eqref{eq:cont_centerline} and \eqref{eq:rigid_joint}, we have
\begin{equation} \label{eq:cont_centerline_rigid_joint_delt}
\partial_t\mathbf{p}_i(0, t)  = \partial_t\mathbf{p}_n(\ell_n, t), \qquad
\partial_t \mathbf{R}_i(0, t) R_{i}(0)^\intercal = \partial_t \mathbf{R}_n(\ell_n, t) R_{n}(\ell_n)^\intercal.
\end{equation}
Left-multiplying the left-hand sides (and the right-hand sides) in \eqref{eq:cont_centerline_rigid_joint_delt} by the transposed left-hand side (resp. right-hand side) of \eqref{eq:rigid_joint}, one obtains
%%%\begin{equation*}
%%%\begin{aligned}
%%%R_{i}(0) \mathbf{R}_i(0, t)^\intercal \partial_t \mathbf{p}_i(0, t) &= R_{n}(\ell_n) \mathbf{R}_n(\ell_n, t)^\intercal\partial_t \mathbf{p}_n(\ell_n, t)\\
%%%R_{i}(0) \mathbf{R}_i(0, t)^\intercal \partial_t \mathbf{R}_i(0, t) R_{i}(0)^\intercal &= R_{n}(\ell_n) \mathbf{R}_n(\ell_n, t)^\intercal \partial_t \mathbf{R}_n(\ell_n, t) R_{n}(\ell_n)^\intercal.
%%%\end{aligned}
%%%\end{equation*}
\begin{eqnarray*}
&R_{i}(0) \mathbf{R}_i(0, t)^\intercal \partial_t \mathbf{p}_i(0, t) = R_{n}(\ell_n) \mathbf{R}_n(\ell_n, t)^\intercal\partial_t \mathbf{p}_n(\ell_n, t)\\
& \quad R_{i}(0) \mathbf{R}_i(0, t)^\intercal \partial_t \mathbf{R}_i(0, t) R_{i}(0)^\intercal = R_{n}(\ell_n) \mathbf{R}_n(\ell_n, t)^\intercal \partial_t \mathbf{R}_n(\ell_n, t) R_{n}(\ell_n)^\intercal.
\end{eqnarray*}
By the definition of $V_i, W_i$ and the invariance of the cross-product in $\mathbb{R}^3$ under rotation, these two systems also write as $R_i(0) V_i(0, t) = R_n(\ell_n) V(\ell_n, t)$ and $R_i(0) W_i(0, t) = R_n(\ell_n) W_n(\ell_n, t)$. As $\overline{R}_{i}$ is defined by \eqref{eq:def_barRi}, we have obtained \eqref{eq:cont_cond}.

\subsubsection{The Kirchhoff condition}

For any $i \in \mathcal{I}$, let us denote by $\phi_i, \psi_i \colon [0, \ell_i]\times [0, T] \rightarrow \mathbb{R}^3$ the \emph{(global) internal forces} and \emph{moments} respectively, and their body-attached counterparts by $\Phi_i, \Psi_i \colon [0, \ell_i]\times [0, T] \rightarrow \mathbb{R}^3$ respectively. As explained in Remark \ref{rem:body_attached_var}, they are related by the identities $\Phi_i = \mathbf{R}_i^\intercal \phi_i$ and $\Psi_i = \mathbf{R}_i^\intercal \psi_i$. 
Similarly, for any node $n \in \mathcal{N}$, denote by $\phi_n^\text{load}, \psi_n^\text{load} \colon [0, T] \rightarrow \mathbb{R}^3$ the \emph{(global) external load} applied at this node, and their body-attached counterparts by $\Phi_n^\text{load}, \Psi_n^\text{load} \colon [0, T] \rightarrow \mathbb{R}^3$, related by the identities
%%% without argument t
%%%\begin{align*}
%%%\Phi_n^\text{load}(t) = \begin{cases}
%%%\mathbf{R}_n(\ell_n, t)^\intercal \phi_n^\text{load}(t), &\text{if }n \neq 0\\
%%% \mathbf{R}_1(0, t)^\intercal \phi_0^\text{load}(t), & \text{if }n=0
%%%\end{cases}, \quad \ \Psi_n^\text{load}(t) = \begin{cases}
%%%\mathbf{R}_n(\ell_n, t)^\intercal \psi_n^\text{load}(t), &\text{if }n \neq 0\\
%%% \mathbf{R}_1(0, t)^\intercal \psi_0^\text{load}(t), & \text{if }n=0.
%%%\end{cases}
%%%\end{align*}
\begin{linenomath}
\begin{equation*}
\Phi_n^\text{load} = \begin{cases}
\mathbf{R}_n(\ell_n, \cdot)^\intercal \phi_n^\text{load}, &\text{if }n \neq 0\\
 \mathbf{R}_1(0, \cdot)^\intercal \phi_0^\text{load}, & \text{if }n=0
\end{cases}, \quad \ \Psi_n^\text{load} = \begin{cases}
\mathbf{R}_n(\ell_n, \cdot)^\intercal \psi_n^\text{load}, &\text{if }n \neq 0\\
 \mathbf{R}_1(0, \cdot)^\intercal \psi_0^\text{load}, & \text{if }n=0.
\end{cases}
\end{equation*}
\end{linenomath}
For any \emph{multiple} node $n \in \mathcal{N}_M$, we require the forces and moments exerted to this node by incident beams to be balanced with the external load applied at this node, meaning that for all $t \in [0, T]$ one has
\begin{equation} \label{eq:node_balance_forces_glob}
\phi_{n}(\ell_n, t) - \sum_{i \in \mathcal{I}_n} \phi_{i}(0, t) = \phi_n^\text{load}(t), \qquad
\psi_{n}(\ell_n, t) - \sum_{i \in \mathcal{I}_n} \psi_{i}(0, t)
=\psi_n^\text{load}(t).
\end{equation}
Using the rigid joint assumption \eqref{eq:rigid_joint} and \eqref{eq:def_vi_si}, we deduce that \eqref{eq:node_balance_forces_glob} is equivalent to
\begin{linenomath}
\begin{equation*}
\overline{R}_{n}(\ell_n) z_n(\ell_n, t)  - \sum_{i \in \mathcal{I}_n} \overline{R}_i(0) z_i(0, t) = \overline{R}_{n}(\ell_n) \begin{bmatrix}
\Phi_n^\text{load}(t)\\ \Psi_n^\text{load}(t)
\end{bmatrix}, \qquad \text{for all }t \in [0, T].
\end{equation*}
\end{linenomath}
As presented in Subsection \ref{subsec:the_model}, either a velocity feedback control is applied at this node $(\Phi_n^\text{load}(t)^\intercal, \Psi_n^\text{load}(t)^\intercal)^ \intercal = - K_n v_n(\ell_n, t)$, with $K_n \in \mathbb{R}^{6 \times 6}$ symmetric positive definite, or no external load is applied at this node, which means that $\phi_n^\mathrm{load} \equiv 0$ and $\psi_n^\mathrm{load} \equiv 0$ but may also be written as $K_n=\mathds{O}_n$. Hence, we have obtained \eqref{eq:Kirchhoff_cond}.

\subsubsection{Conditions at the simple nodes} Let $n \in \mathcal{N}_S$.
Similarly to the Kirchhoff condition, the balance between internal forces/moments and external loads is required. It takes the form
\begin{linenomath}
\begin{align}
\phi_{n}(\ell_n, t) = \phi_n^\text{load}(t), \quad \ \psi_{n}(\ell_n, t) = \psi_n^\text{load}(t), \qquad &\text{for all }t \in [0, T], \ \text{if } n \neq 0 \label{eq:balance_simple_n}\\
-\phi_{1}(0, t) = \phi_0^\text{load}(t), \quad -\psi_{1}(0, t) = \psi_0^\text{load}(t), \qquad &\text{for all }t \in [0, T], \ \text{if } n=0.\label{eq:balance_simple_n=0}
\end{align}
\end{linenomath}
Left-multiplying the systems in \eqref{eq:balance_simple_n} and \eqref{eq:balance_simple_n=0} by $\mathbf{R}_n(\ell_n)^\intercal$ and $\mathbf{R}_1(0)^\intercal$ respectively, and using \eqref{eq:def_vi_si}, we obtain
\begin{linenomath}
\begin{equation*}
z_n(\ell_n, t) = \begin{bmatrix}
\Phi_n^\text{load}(t)\\ \Psi_n^\text{load}(t)
\end{bmatrix}, \quad \text{if }n\neq 0, \qquad \quad  - z_1(0, t) = \begin{bmatrix}
\Phi_0^\text{load}(t)\\ \Psi_0^\text{load}(t)
\end{bmatrix}, \quad \text{if }n=0.
\end{equation*}
\end{linenomath}
The nodal conditions \eqref{eq:neum_cond_n} and \eqref{eq:neum_cond_0} result from applying the following controls 
\begin{linenomath}
\begin{equation*}
\begin{bmatrix}
\Phi_n^\text{load}(t)\\ \Psi_n^\text{load}(t)
\end{bmatrix} 
= 
\begin{cases}
- K_n v_n(\ell_n, t) &  \text{if } n \neq 0\\
- K_0 v_1(0, t) & \text{if }n=0,
\end{cases}
\end{equation*}
\end{linenomath}
with $K_n \in \mathbb{R}^{6 \times 6}$ symmetric positive definite. If no load is applied at the node, meaning that $\phi_n^\mathrm{load} \equiv 0$ and $\psi_n^\mathrm{load} \equiv 0$, then we set $K_n = \mathds{O}_6$.

If the beam is clamped at the node $n$, position and rotation are both independent of time at this node. Namely, for constant $h_n^\mathbf{p} \in \mathbb{R}^3$ and $h_n^\mathbf{R}\in \mathbb{R}^{3 \times 3}$, we set
\begin{equation*}
\begin{aligned}
\mathbf{p}_n(\ell_n, t) &= h_n^\mathbf{p}, \qquad \mathbf{R}_n(\ell_n, t) = h_n^\mathbf{R}, \qquad \text{for all }t \in [0, T], \text{ if }n\neq 0,\\
\mathbf{p}_1(0, t) &= h_0^\mathbf{p}, \qquad \ \hspace{0.075cm} \mathbf{R}_1(0, t) = h_0^\mathbf{R}, \qquad \hspace{0.025cm} \text{for all }t \in [0, T],\text{ if }n= 0,
\end{aligned}
\end{equation*}
which yields \eqref{eq:diri_cond_0n}, by definition of $v_i$ (see \eqref{eq:def_vi_si}-\eqref{eq:def_VWGU}).

\section{Riemann invariants and Proof of Theorem \ref{thm:well-posedness}}
\label{sec:riem_wellp}

In this section, we first write System \eqref{eq:syst_y} in diagonal form, before proving the well-posedness result. Let us first comment on the hyperbolicity of $A_i$ ($i\in\mathcal{I}$).

\subsubsection*{Hyperbolicity of the system}

Let $i\in \mathcal{I}$ and $x \in [0, \ell_i]$. One may quickly verify that $A_i(x)$ (see \eqref{eq:def_Ai_barBi}) has six positive and six negative real eigenvalues, for any $\mathbf{M}_i, \mathbf{C}_i$ having values in the set of positive definite symmetric matrices.
Indeed, one may study the zeros of $\mathrm{det}(\lambda \mathds{I}_{12} - A_i)$, where we drop the argument $x$ for clarity. Some computations yield that it is equal to $\mathrm{det}(\lambda^2 \mathds{I}_6 - (\mathbf{C}_i \mathbf{M}_i)^{-1})$, which reduces the problem to showing that $(\mathbf{C}_i \mathbf{M}_i)^{-1}$ has are real, positive eigenvalues only. The latter matrix also writes as $(\mathbf{C}_i \mathbf{M}_i)^{-1} =  \mathbf{C}_i^{-\sfrac{1}{2}} \Theta_i \mathbf{C}_i^{\sfrac{1}{2}}$, with $\Theta_i$ defined by \eqref{eq:def_Thetai}, implying that it has the same eigenvalues as $\Theta_i$ since $\mathbf{C}_i$ is invertible -- all are real and positive as $\Theta_i$ is symmetric positive definite. 
Hence, $(\mathbf{C}_i \mathbf{M}_i)^{-1}$ is possibly not positive definite, but has necessarily real, positive eigenvalues.

Further to this, Assumption \ref{as:mass_flex}, by ensuring a certain regularity of its eigenvalues and eigenvectors, permits to obtain the following lemma.

\begin{lemma} \label{lem:diag_Ai}
Let Assumption \ref{as:mass_flex} be fulfilled for some $m \in \{1, 2, \ldots\}$ and, for $i \in \mathcal{I}$, let $U_i, D_i \in C^{m}([0, \ell_i]; \mathbb{R}^{6\times 6})$ be the matrices introduced in Assumption \ref{as:mass_flex}.
Then, $A_i \in C^{m}([0, \ell_i]; \mathbb{R}^{12\times 12})$ which is defined in \eqref{eq:def_Ai_barBi}, may be diagonalized as follows. In $[0, \ell_i]$, one has $A_i = L_i^{-1} \mathbf{D}_i L_i$, where $\mathbf{D}_{i}, L_i \in C^m( [0, \ell_i]; \mathbb{R}^{12 \times 12})$ are defined by
\begin{equation} \label{eq:def_bfDi_Li}
\mathbf{D}_i = \mathrm{diag}(-D_{i}, D_{i}), \qquad L_i = \begin{bmatrix}
U_i \mathbf{C}_i^{-\sfrac{1}{2}} & D_{i}U_i \mathbf{C}_i^{\sfrac{1}{2}} \\
U_i \mathbf{C}_i^{-\sfrac{1}{2}} & - D_{i}U_i \mathbf{C}_i^{\sfrac{1}{2}}
\end{bmatrix}.
\end{equation}
\end{lemma}

\begin{proof}
Let $i\in \mathcal{I}$ and $x \in [0, \ell_i]$. Here, we drop again the argument $x$ to lighten the notation. 
Being symmetric and positive definite, $\Theta_i$ may always be written as  $\Theta_i = U_i^\intercal D_i^2 U_i$ where $D_i$ is a positive definite diagonal matrix whose diagonal entries are the square roots of the eigenvalues of $\Theta_i$ and $U_i$ is a unitary matrix, and 
Assumption \ref{as:mass_flex} ensures that such $D_i$, $U_i$ with regularity $C^m([0, \ell_i]; \mathbb{R}^{6 \times 6})$ exist.
Let us define the matrices $P_i = \mathrm{diag}(\mathbf{C}_i^{-1}, \mathbf{C}_i)^{\sfrac{1}{2}}$ and $Q_i = \mathrm{diag}(U_i, U_i)$. Then, a few computations lead to the expression $A_i =  \big(L_{0i}  Q_i P_i \big)^{-1} \mathbf{D}_i \big(L_{0i}  Q_i P_i \big)$,
%%% CUT £££
%%%\begin{align*}
%%%A_i 
%%%&= -P_i^{-1} \begin{bmatrix}
%%%\mathds{O}_6 & \Theta_i\\
%%%\mathds{I}_6 & \mathds{O}_6
%%%\end{bmatrix}P_i \\
%%%&= - \big(Q_i P_i \big)^{-1} \begin{bmatrix}
%%%\mathds{O}_6 & D_i^2\\
%%%\mathds{I}_6 & \mathds{O}_6
%%%\end{bmatrix} \big(Q_i P_i \big) \\
%%%&= \big(L_{0i}  Q_i P_i \big)^{-1} \mathbf{D}_i \big(L_{0i}  Q_i P_i \big),
%%%\end{align*}
where $L_{0i}$ and its inverse are given by
\begin{linenomath}
\begin{equation*}
L_{0i} = \begin{bmatrix}
\mathds{I}_6 & D_i \\
\mathds{I}_6 & - D_i
\end{bmatrix}, \qquad 
L_{0i}^{-1} = \frac{1}{2} \begin{bmatrix}
\mathds{I}_6 & \mathds{I}_6\\
D_i^{-1} & - D_i^{-1}
\end{bmatrix}.
\end{equation*}
\end{linenomath}
The matrix $L_i$ defined in \eqref{eq:def_bfDi_Li} corresponds to $L_i = L_{0i} Q_i P_i$. Its inverse $L_i^{-1} = P_i^{-1} Q_i^{-1} L_{0i}^{-1}$ takes the form
\begin{linenomath}
\begin{equation*} %\label{eq:gl_change_L_inverse}
L_i^{-1} = \frac{1}{2} \begin{bmatrix}
\mathbf{C}_i^{\sfrac{1}{2}} U_i^\intercal & \mathbf{C}_i^{\sfrac{1}{2}} U_i^\intercal \\
\mathbf{C}_i^{-\sfrac{1}{2}} U_i^\intercal D_i^{-1} & - \mathbf{C}_i^{-\sfrac{1}{2}} U_i^\intercal D_i^{-1}
\end{bmatrix}.\qedhere
\end{equation*}
\end{linenomath}
\end{proof}

\begin{remark}
The regularity of $L_i, L_i^{-1}$ and $\Theta_i$(see \eqref{eq:def_Thetai}) follows from that of $D_i, U_i$  -- provided by Assumption \ref{as:mass_flex} -- and from the fact that for any $m \in \{0, 1, \ldots\}$, $d \in \{1, 2, \ldots\}$ and any function $M \in C^m( [a, b]; \mathbb{R}^{d \times d})$ having values in the set of positive definite symmetric matrices, both $M^{-1}$ and $M^{\sfrac{1}{2}}$ are as regular as $M$.
\end{remark}

\subsection{System written in Riemann invariants}
\label{subsec:riem}

Lemma \ref{lem:diag_Ai} being established, we can write System \eqref{eq:syst_y} in diagonal form by applying the change of variable 
\begin{equation} \label{eq:change_var_Li}
r_i = L_iy_i, \qquad \text{for all }i \in \mathcal{I}.
\end{equation}
The first (resp. last) six components of $r_i$ correspond to the negative (resp. positive) eigenvalues of $A_i$, this is why we use the notation
\begin{linenomath}
\begin{equation*}
r_i = \begin{bmatrix}
r_i^- \\
r_i^+ 
\end{bmatrix}, \qquad r_i^- ,\, r_i^+  \colon [0, \ell_i]\times [0, T] \rightarrow \mathbb{R}^6,
\end{equation*}
\end{linenomath}
for all $i \in \mathcal{I}$.
More explicitly, $y_i$ and $r_i$ are related as follows:
\begin{equation} \label{eq:relation_r_y}
\begin{aligned}
r_i^-  &= U_i \mathbf{C}_i^{-\sfrac{1}{2}}( v_i + \mathbf{C}_i^{\sfrac{1}{2}} \mathbf{M}_i^{-1} \mathbf{C}_i^{\sfrac{1}{2}} z_i), & v_i &= \frac{1}{2}\mathbf{C}_i^{\sfrac{1}{2}} U_i^\intercal(r_i^- +r_i^+ ),\\
r_i^+  &= U_i\mathbf{C}_i^{-\sfrac{1}{2}}( v_i - \mathbf{C}_i^{\sfrac{1}{2}} \mathbf{M}_i^{-1} \mathbf{C}_i^{\sfrac{1}{2}} z_i), &
z_i &= \frac{1}{2} \mathbf{C}_i^{-\sfrac{1}{2}}U_i^\intercal D_{i}^{-1} (r_i^- -r_i^+ ).
\end{aligned}
\end{equation}

When applying \eqref{eq:change_var_Li} to the governing equations of \eqref{eq:syst_y}, we obtain the following governing system for the new unknown state $r = (r_1, \ldots, r_N)$:
\begin{equation} \label{eq:gov_riem}
\partial_t r_i + \mathbf{D}_i(x) \partial_x r_i + B_i(x) r_i = g_i (x,r_i), \qquad \text{for all }i \in \mathcal{I},
\end{equation}
where $B_i$ is defined by $B_i(x) = L_i(x) \overline{B}_i(x) L_i(x)^{-1} + L_i(x) A_i(x) \frac{\mathrm{d}}{\mathrm{d}x}L_i^{-1}(x)$ and the source $g_i$ is defined by $g_i(x,\mathbf{r}) = L_i(x) \overline{g}_i(x,L_i(x)^{-1} \mathbf{r} )$, for all $i \in \mathcal{I}$, $x \in [0, \ell_i]$ and $\mathbf{r} \in \mathbb{R}^{12}$.
The initial datum for this problem is $r_i^0 = L_i y_i^0$. 

\begin{remark}
Note that $B_i \in C^k([0, \ell_i]; \mathbb{R}^{12 \times 12})$ under Assumption \ref{as:mass_flex} with $m = k+1$.
\end{remark}

It remains to apply the change of variable to the nodal conditions.
Later on, in order to study the well-posedness of System \eqref{eq:syst_y}, we will verify that, at each node $n$ of the network, the \emph{outgoing information}, denoted $r^\mathrm{out}_n$, may be expressed explicitly as a function of the \emph{incoming information}, denoted $r^\mathrm{in}_n$. 
We define $r^\mathrm{out}_n$ and $r^\mathrm{in}_n$ at this stage in order to make use of this notation, here, to write the new nodal conditions. Let us first define the notion of outgoing/incoming information.

\begin{definition}\label{def:out_in_info}
Let $\ell >0$. Consider a semilinear hyperbolic system of the form $\partial_t \zeta + \Lambda(x) \partial_x \zeta + M(x, \zeta) \zeta = 0$ when it is written in Riemann invariants. More precisely, $\Lambda = \mathrm{diag}(\Lambda_-, \Lambda_+)$ where $\Lambda_-$ (resp. $\Lambda_+$) has values in the set of negative (resp. positive) definite diagonal matrices of size $m$ (resp. $d-m$), for some $m\leq d$ belonging to $\{1, 2, \ldots\}$. Here, the unknown state is $\zeta \colon [0, \ell]\times[0, T]\rightarrow \mathbb{R}^d$, and $\zeta = ((\zeta^-)^\intercal, (\zeta^+)^\intercal)^\intercal$ where $\zeta^- \in \mathbb{R}^m$ (resp. $\zeta^+ \in \mathbb{R}^{d-m}$) are the components of $\zeta$ corresponding to the negative (resp. positive) diagonal entries of $\Lambda$.

The \emph{outgoing information} consists of the components of $\zeta$ corresponding to characteristics which are outgoing at the boundaries $x=\ell$ and $x=0$ (in other words, going into the domain $[0, \ell]\times [0, T]$ from outside of it): these are $\zeta^{-}(\ell, t)$ and $\zeta^{+}(0, t)$ respectively. Likewise, we mean by \emph{incoming information}, the components of $\zeta$ corresponding to characteristics which are incoming at the boundaries $x=\ell$ and $x=0$ (going out of $[0, \ell]\times[0, T]$ from inside): these are $\zeta^{+}(\ell, t)$ and $\zeta^{-}(0, t)$ respectively. We refer to Fig. \ref{fig:charac} for visualization.
\end{definition}

\begin{figure}
\includegraphics[scale=0.75]{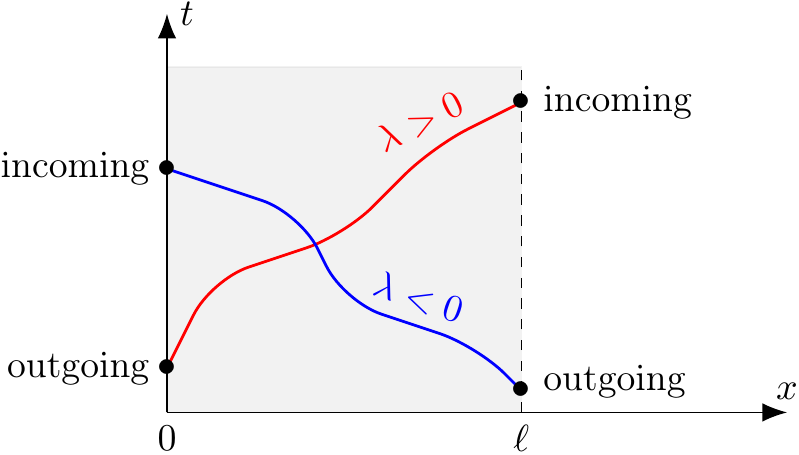}
\caption{Characteristic curves $(\mathbf{x}(t), t)$ with $\frac{\mathrm{d}\mathbf{x}}{\mathrm{d}t}(t) = \lambda(\mathbf{x}(t))$, where either $\lambda(s)>0$ or $\lambda(s)<0$ for all $s \in [0, \ell]$.}
\label{fig:charac}
\end{figure}

In the sense of Definition \ref{def:out_in_info}, for any beam $i\in \mathcal{I}$, the outgoing information is $r_i^- (\ell, t)$ and $r_i^+ (0, t)$, while the incoming information is $r_i^+ (\ell, t)$ and $r_i^- (0, t)$, and consequently, for any node $n \in \mathcal{N}$, the outgoing and incoming information $r^\mathrm{out}_n, r^\mathrm{in}_n \colon [0, T] \rightarrow \mathbb{R}^{6k_n}$ are given by (see \eqref{eq:nota_In_indices}-\eqref{eq:def_kn})
\begin{equation} \label{eq:r_out_in_mult_simp}
r^\mathrm{out}_n =
\begin{cases}
\left[\begin{smallmatrix}
r_n^- (\ell_n, \cdot)\\
r_{i_2}^+ (0, \cdot)\\
\vdots\\
r_{i_{k_n}}^+ (0, \cdot)
\end{smallmatrix}\right], & n \in \mathcal{N}_M\\
r_n^- (\ell_n, \cdot), & n \in \mathcal{N}_S \setminus \{0\}\\
r_1^+(0, \cdot), & n = 0
\end{cases}, \ 
r^\mathrm{in}_n =
\begin{cases}
\left[\begin{smallmatrix}
r_n^+(\ell_n, \cdot)\\
r_{i_2}^- (0, \cdot)\\
\vdots\\
r_{i_{k_n}}^- (0, \cdot)
\end{smallmatrix}\right], & n \in \mathcal{N}_M\\
r_n^+(\ell_n, \cdot), &n \in  \mathcal{N}_S \setminus \{0\}\\
r_1^-(0, \cdot), & n = 0.
\end{cases}
\end{equation}
Furthermore, to write concisely the nodal conditions, we define, for all $n \in \mathcal{N}$, the $6 \times 6$ invertible matrices
\begin{equation} \label{eq:def_gammai}
\begin{aligned}
\gamma_0^0 &= \big(\overline{R}_1 \mathbf{C}_1^{\sfrac{1}{2}} U_1^\intercal\big)\big(0\big),\qquad  \ \text{ if } n = 0,\\
\gamma_n^n &= \big(\overline{R}_n \mathbf{C}_n^{\sfrac{1}{2}} U_n^\intercal\big)\big(\ell_n\big), \qquad  \text{ if } n \neq 0,\\
\gamma_i^n &= \big(\overline{R}_i \mathbf{C}_i^{\sfrac{1}{2}} U_i^\intercal\big)\big(0\big), \quad \qquad \text{for all }i \in \mathcal{I}_n, \text{ if } n \in \mathcal{N}_M,
\end{aligned}
\end{equation}
the $6 \times 6$ symmetric matrices
\begin{equation}\label{eq:def_barKn}
\begin{aligned}
\overline{K}_0 &= \overline{R}_1(0) K_0 \overline{R}_1(0)^\intercal, \qquad  \quad \text{ if } n = 0,\\
\overline{K}_n &= \overline{R}_n(\ell_n) K_n \overline{R}_n(\ell_n)^\intercal,\qquad  \text{ if } n \neq 0,
\end{aligned}
\end{equation}
which are positive definite (resp. null) if and only if $K_n$ is positive definite (resp. $K_n = \mathds{O}_6$), and the $6 \times 6$ positive definite symmetric matrices
\begin{equation}\label{eq:def_sigmai}
\begin{aligned}
\sigma_0^0 &= \big(\overline{R}_1 \mathbf{C}_1^{-\sfrac{1}{2}} U_1^\intercal D_1^{-1} U_1 \mathbf{C}_1^{-\sfrac{1}{2}} \overline{R}_1^\intercal\big)\big(0\big), \ \ \qquad  \text{ if } n = 0,\\
\sigma_n^n &= \big(\overline{R}_n \mathbf{C}_n^{-\sfrac{1}{2}} U_n^\intercal D_n^{-1} U_n \mathbf{C}_n^{-\sfrac{1}{2}} \overline{R}_n^\intercal\big)\big(\ell_n\big),\qquad  \text{ if } n \neq 0,\\
\sigma_i^n &= \big(\overline{R}_i \mathbf{C}_i^{-\sfrac{1}{2}} U_i^\intercal D_i^{-1} U_i \mathbf{C}_i^{-\sfrac{1}{2}} \overline{R}_i^\intercal\big)\big(0\big), \ \quad \qquad \text{for all }i \in \mathcal{I}_n, \text{ if } n \in \mathcal{N}_M.
\end{aligned}
\end{equation}

Let us now apply \eqref{eq:change_var_Li} to the nodal conditions, starting with the transmission conditions. Let $n \in \mathcal{N}_M$.
Injecting \eqref{eq:relation_r_y} in \eqref{eq:cont_cond}, the continuity condition writes as
\begin{equation} \label{eq:cont_Riem_1}
\gamma_i^n (r_i^- +r_i^+ )(0, t) = \gamma_n^n (r_n^- + r_n^+)(\ell_n, t), \qquad \text{for all }t \in [0, T], \ i \in \mathcal{I}_n.
\end{equation}
Since $\overline{R}_i \mathbf{C}_i^{-\sfrac{1}{2}}U_i^\intercal D_{i}^{-1} = \sigma_i^n \gamma_i^n$ (for any $i \in \mathcal{I}_n\cup \{n\}$) and $\overline{R}_n K_n \mathbf{C}_n^{\sfrac{1}{2}} U_n^\intercal = k_n \overline{K}_n \gamma_n^n$ by definition, we deduce that the 
Kirchhoff condition \eqref{eq:Kirchhoff_cond} takes the form
\begin{linenomath}
\begin{equation*}
\sigma_n^n \gamma_n^n (r_n^- - r_n^+)(\ell_n, t) - \sum_{i \in \mathcal{I}_n}  \sigma_i^n \gamma_i^n (r_i^-  - r_i^+ )(0, t) = - \overline{K}_n \gamma_n^n (r_n^-+r_n^+)(\ell_n, t).
\end{equation*}
\end{linenomath}
%%%%%%%%Note that, by \eqref{eq:cont_Riem_1}, one can equivalently replace the term $\gamma_n^n (r_n^-+r_n^+)(\ell_n, t)$ in the above right-hand side by $\gamma_i^n (r_i^- +r_i^+ )(0, t)$ for any $i \in \mathcal{I}_n$.
Gathering the outgoing information on the left-hand side, the equivalent expression
\begin{equation} \label{eq:Kirch_alphan_betan}
\alpha^n r^\mathrm{out}_n (t) = \beta^n r^\mathrm{in}_n(t)
\end{equation}
is obtained, where the \emph{rectangular} matrices $\alpha^n, \beta^n \in \mathbb{R}^{6 \times 6k_n}$ are defined by
\begin{equation}\label{eq:def_alphan_betan}
\begin{aligned}
\alpha^n &= \begin{bmatrix}
(\sigma_n^n + \overline{K}_n)\gamma_n^n & \sigma_{i_2}^n \gamma_{i_2}^n &  \sigma_{i_3}^n \gamma_{i_3}^n & \ldots & \sigma_{i_{k_n}}^n \gamma_{i_{k_n}}^n
\end{bmatrix},\\
\beta^n &= \begin{bmatrix}
(\sigma_n^n - \overline{K}_n)\gamma_n^n & \sigma_{i_2}^n \gamma_{i_2}^n &  \sigma_{i_3}^n \gamma_{i_3}^n & \ldots & \sigma_{i_{k_n}}^n \gamma_{i_{k_n}}^n
\end{bmatrix}.
\end{aligned} %\qquad \text{for }n \in \mathcal{N}_M.
\end{equation}

We now turn to the simple nodes $n \in \mathcal{N}_S$.
Similarly to the Kirchhoff condition, injecting \eqref{eq:relation_r_y} in \eqref{eq:neum_cond_n} and in \eqref{eq:neum_cond_0}, we obtain
\begin{linenomath}
\begin{align}\label{eq:Neu_Riem_1_xelln_before}
\sigma_n^n \gamma_n^n (r_n^- - r_n^+)(\ell_n, t) &= - \overline{K}_n \gamma_n^n (r_n^- + r_n^+)(\ell_n, t),  && \text{if }n \neq 0 \\
\sigma^0_0 \gamma^0_0 (r_1^-  - r_1^+)(0, t) &= \overline{K}_0 \gamma^0_0 (r_1^-  + r_1^+)(0, t),  && \text{if }n = 0, \label{eq:Neu_Riem_1_x0_before}
\end{align}
\end{linenomath}
respectively.
For any $n \in \mathcal{N}$, $(\sigma_n^n + \overline{K}_n)$ is symmetric and positive definite (hence invertible), allowing us to gather the outgoing information on the left-hand side and obtain the equivalent (to \eqref{eq:Neu_Riem_1_xelln_before}, \eqref{eq:Neu_Riem_1_x0_before} resp.) expressions
\begin{linenomath}
\begin{align} \label{eq:Neu_Riem_1_xelln}
r_n^-(\ell_n, t) &=  ((\sigma_n^n + \overline{K}_n)\gamma_n^n)^{-1}(\sigma_n^n - \overline{K}_n)\gamma_n^n \, r_n^+(\ell_n, t), && \text{if }n \neq 0\\
 r_1^+(0, t) &= ( (\sigma_0^0 + \overline{K}_0)\gamma_0^0)^{-1} (\sigma_0^0 - \overline{K}_0)\gamma_0^0 \, r_1^- (0, t), && \text{if }n = 0. \label{eq:Neu_Riem_1_x0}
\end{align}
\end{linenomath}
Then, for a free beam (i.e. $K_n = \mathds{O}_6$), \eqref{eq:Neu_Riem_1_xelln} and \eqref{eq:Neu_Riem_1_x0} take the form $r_n^\mathrm{out}(t) = r_n^\mathrm{in}(t)$.

Using \eqref{eq:gov_riem},\eqref{eq:cont_Riem_1},\eqref{eq:Kirch_alphan_betan},\eqref{eq:Neu_Riem_1_xelln} and \eqref{eq:Neu_Riem_1_x0}, the system in Riemann invariants writes as
\begin{equation} \label{eq:syst_Riem_1}
\begin{dcases}
\partial_t r_i + \mathbf{D}_i(x) \partial_x r_i + B_i(x) r_i = g_i (x,r_i) 
&\text{in } (0, \ell_i)\times(0, T), \, i \in \mathcal{I}\\
\gamma_i^n (r_i^-  + r_i^+ )(0, t) = \gamma_n^n (r_n^- + r_n^+)(\ell_n, t)
&t \in (0, T), \, i \in \mathcal{I}_n, \, n \in \mathcal{N}_M\\
\alpha^n r^\mathrm{out}_n (t) = \beta^n r^\mathrm{in}_n(t)
& t \in (0, T), \, n \in \mathcal{N}_M \\
r^\mathrm{out}_n(t)  = ((\sigma_n^n + \overline{K}_n)\gamma_n^n)^{-1}(\sigma_n^n - \overline{K}_n)\gamma_n^n r^\mathrm{in}_n(t)
&t \in (0, T), \, n \in \mathcal{N}_S\\
r_i(x, 0) = r_i^0(x)  &x\in (0, \ell_i), \, i \in \mathcal{I}.
\end{dcases}
\end{equation}

Note that, for $n \in \mathcal{N}_S$, when injecting \eqref{eq:relation_r_y} in the Dirichlet conditions \eqref{eq:diri_cond_0n}, they take the form $(\overline{R}_1 \mathbf{C}_1^{\sfrac{1}{2}} U_1^\intercal (r_1^-  + r_1^+))(0, t) = 0$ for $n=0$, or $(\overline{R}_n \mathbf{C}_n^{\sfrac{1}{2}} U_n^\intercal (r_n^- + r_n^+))(\ell_n, t) = 0$ for $n \neq 0$, which both write equivalently as $r^\mathrm{out}_n(t) = - r^\mathrm{in}_n(t)$.

\begin{remark}[Transparent conditions] 
For \emph{simple} nodes $n \in \mathcal{N}_S$, the specific choice
\begin{linenomath}
\begin{equation*}
\begin{aligned}
K_n = \big(\mathbf{C}_n^{-\sfrac{1}{2}} (\mathbf{C}_n^{\sfrac{1}{2}} \mathbf{M}_n \mathbf{C}_n^{\sfrac{1}{2}})^{\sfrac{1}{2}}\mathbf{C}_n^{-\sfrac{1}{2}}\big)\big(\ell_n\big), \qquad &\text{if }n\neq 0,\\
K_0 = \big(\mathbf{C}_1^{-\sfrac{1}{2}} (\mathbf{C}_1^{\sfrac{1}{2}} \mathbf{M}_1 \mathbf{C}_1^{\sfrac{1}{2}})^{\sfrac{1}{2}}\mathbf{C}_1^{-\sfrac{1}{2}}\big)\big(0\big), \qquad &\text{if }n = 0,
\end{aligned}
\end{equation*}
\end{linenomath}
leads to the nodal conditions $r_n^\mathrm{out}(t) = 0$, for all $t \in [0, T]$, where there is no outgoing information at the node $n$ -- as \eqref{eq:def_Thetai},\eqref{eq:diag_Thetai},\eqref{eq:def_barKn},\eqref{eq:def_sigmai} imply that $\overline{K}_n = \sigma^n_n$.
\end{remark}

\subsection{Proof of Theorem \ref{thm:well-posedness}}
\label{subsec:wellposedness}

To show that System \eqref{eq:syst_y} is well-posed, we will apply the following theorem, of which a more general version may be found in \cite[Ap. B, Rem. 6.9]{BC2016} and \cite[Thm. 10.1]{bastin2017exponential} -- for the cases $k=1$ and $k=2$ respectively.

\begin{theorem}[Bastin \& Coron 2016]
\label{thm:bastin_coron_well_posed}
Let $\ell>0$, $d \in \{1, 2, \ldots\}$ and $k \in \{1, 2\}$ be fixed, and denote $H^k_x = H^k(0, \ell; \mathbb{R}^d)$. For a matrix $K \in \mathbb{R}^{d\times d}$, and functions $\Lambda \in C^k([0, \ell]; \mathbb{R}^{d\times d})$ and $M \in C^k([0, \ell]\times \mathbb{R}^d; \mathbb{R}^{d\times d})$, consider the following system written in Riemann invariants -- which has been partially introduced in Definition \ref{def:out_in_info} -- where $\zeta^\mathrm{out},\zeta^\mathrm{in}$ contain the outgoing and incoming information:
\begin{equation}\label{eq:coron_thm_syst}
\begin{cases}
\partial_t \zeta + \Lambda(x) \partial_x \zeta + M(x, \zeta)\zeta = 0 &\text{in }(0, \ell) \times (0, T)\\
\zeta^\mathrm{out}(t) = K \zeta^{\mathrm{in}}(t) &t \in (0, T)\\
\zeta(x, 0) = \zeta^0(x) &x \in (0, \ell).
\end{cases}
\end{equation}
Then, there exists $\delta_0>0$ such that for any $\zeta^0 \in {H}_x^k$ satisfying the $(k-1)$-order compatibility conditions of \eqref{eq:coron_thm_syst} (defined analogously to Definition \ref{def:comp_cond}) and $\|\zeta^0\|_{{H}_x^k} \leq \delta_0$, there exists a unique solution $\zeta \in C^0([0, T); {H}_x^k)$ to System \eqref{eq:coron_thm_syst}, for some $T \in (0, +\infty]$. Moreover, if $\|\zeta(\cdot, t)\|_{{H}_x^k} \leq \delta_0$ for all $t \in [0, T)$, then $T = + \infty$.
%%%%%%\begin{align*}
%%%%%%\|\zeta(\cdot, t)\|_{{H}^k(0, \ell; \mathbb{R}^d)} \leq \delta_0, \qquad \text{for all } t \in [0, T).
%%%%%%\end{align*}
\end{theorem}

To apply Theorem \ref{thm:bastin_coron_well_posed}, on the one hand we will write System \eqref{eq:syst_Riem_1} as a single larger hyperbolic system -- hence, we will beforehand apply a change of variable to obtain the same spatial interval for all the beams of the network -- and, on the other hand we will verify that the boundary conditions of \eqref{eq:syst_Riem_1} fit in the framework of Theorem \ref{thm:bastin_coron_well_posed}. We can see in \eqref{eq:syst_Riem_1} that at the simple nodes, the outgoing information $r_n^\mathrm{out}$ is expressed explicitly as a function of the incoming information $r_n^\mathrm{in}$. However, this property remains to be verified for multiple nodes, which is the object of Lemma \ref{lem:calB_n} given below.
Let us first introduce some notation for multiple nodes $n \in \mathcal{N}_M$. We define the positive definite matrix $\sigma^n \in \mathbb{R}^{6 k_n \times 6k_n}$ by
\begin{equation} \label{eq:def_sigman_sum}
\sigma^n = \sigma_n^n + {\textstyle \sum_{j=2}^{k_n}} \sigma_{i_j}^n,
\end{equation}
the block diagonal matrices $\mathbf{G}_n, \mathrm{diag}(\sigma^n + \overline{K}_n), \mathrm{diag}(\sigma_i^n) \in \mathbb{R}^{6k_n \times 6k_n}$ by
\begin{equation}\label{eq:def_block_diag}
\begin{aligned}
&\mathbf{G}_n = \mathrm{diag}\big(\gamma_n^n, \ \gamma^n_{i_2}, \ \gamma^n_{i_3}, \ \ldots, \ \gamma^n_{i_{k_n}} \big), \quad \mathrm{diag}(\sigma_i^n) =\mathrm{diag}\big(\sigma_n^n, \sigma_{i_2}^n, \ldots, \sigma_{i_{k_n}}^n \big)\\
&\hspace{2cm} \mathrm{diag}(\sigma^n+\overline{K}_n) = \mathrm{diag}\big(\sigma^n + \overline{K}_n , \ldots , \sigma^n + \overline{K}_n\big),
\end{aligned}
\end{equation}
and the square block matrix $\textswab{I}_n \in \mathbb{R}^{6k_n \times 6k_n}$ by 
%%%%%\begin{align}\label{eq:def_I_crochet}
%%%%%\textswab{I}_n = \begin{bmatrix}
%%%%%\mathds{I}_6 & \mathds{I}_6 & \ldots & \mathds{I}_6\\
%%%%%\mathds{I}_6 & \mathds{I}_6 & \ldots & \mathds{I}_6 \\
%%%%%\vdots & \vdots & \ddots & \vdots \\
%%%%%\mathds{I}_6 & \mathds{I}_6 & \ldots & \mathds{I}_6
%%%%%\end{bmatrix}.
%%%%%\end{align}
\begin{equation}\label{eq:def_I_crochet}
\textswab{I}_n = \begin{bmatrix}
\mathds{I}_6 &\ldots & \mathds{I}_6\\
\vdots & \ddots  & \vdots \\
\mathds{I}_6 & \ldots & \mathds{I}_6
\end{bmatrix}.
\end{equation}

\begin{lemma} \label{lem:calB_n}
The nodal conditions of System \eqref{eq:syst_Riem_1} are equivalent to 
\begin{equation} \label{eq:nodal_cond_calBn}
r^{\mathrm{out}}_n(t) = \mathcal{B}_n r^{\mathrm{in}}_n(t), \qquad t \in [0, T], \text{ for all } n \in \mathcal{N},
\end{equation}
where $\mathcal{B}_n \in \mathbb{R}^{6 k_n \times 6 k_n}$ is defined by
\begin{equation} \label{eq:def_calBn}
\mathcal{B}_n =
\begin{dcases}
((\sigma_n^n + \overline{K}_n)\gamma_n^n)^{-1}(\sigma_n^n - \overline{K}_n)\gamma_n^n, & \text{if }n \in \mathcal{N}_S,\\
2 \mathbf{G}_n^{-1}\mathrm{diag}(\sigma^n+\overline{K}_n)^{-1}  
\textswab{I}_n \mathrm{diag}(\sigma_i^n) \mathbf{G}_n -  \mathds{I}_{6k_n}, & \text{if } n \in \mathcal{N}_M.
\end{dcases}
\end{equation}
\end{lemma}

\begin{remark} \label{rem:calBn_free_clamped}
As seen in Subsection \ref{subsec:riem}, if the beam is free (or clamped) at the node $n \in \mathcal{N}_S$, then $\mathcal{B}_n$ is rather given by $\mathcal{B}_n = \mathds{I}_6$ (resp. $\mathcal{B}_n = - \mathds{I}_6$).
\end{remark}

\begin{proof}[Proof of Lemma \ref{lem:calB_n}]
The case of simple nodes being already solved, we consider the multiple nodes $n \in \mathcal{N}_M$. The continuity condition also writes as
\begin{linenomath}
\begin{equation*}
- \gamma_n^n r_n^-(\ell_n, t) + \gamma^n_{i_j} r_{i_j}^+(0, t) = \gamma^n_n r_n^+(\ell_n, t) - \gamma^n_{i_j} r_{i_j}^-(0,t), \quad \text{for all }2\leq j \leq k_n.
\end{equation*}
\end{linenomath}
Thenè, the combined continuity and Kirchhoff \eqref{eq:Kirch_alphan_betan}-\eqref{eq:def_alphan_betan} conditions are equivalent to
\begin{linenomath}
\begin{equation*}
\mathbf{A}_n \mathbf{G}_n r_\mathrm{out}^n(t) = \mathbf{B}_n \mathbf{G}_n r_\mathrm{in}^n(t),
\end{equation*}
\end{linenomath}
where $\mathbf{A}_n, \mathbf{B}_n \in \mathbb{R}^{6k_n \times 6k_n}$ are defined by
\begin{linenomath}
\begin{equation*}
\mathbf{A}_n = \begin{bmatrix}
\mathbf{a}_n & \mathbf{b}_n\\
\mathbf{c}_n & \mathds{I}_{6(k_n-1)}
\end{bmatrix}, \qquad
\mathbf{B}_n = \begin{bmatrix}
\mathbf{d}_n & \mathbf{b}_n\\
-\mathbf{c}_n & -\mathds{I}_{6(k_n-1)}
\end{bmatrix},
\end{equation*}
\end{linenomath}
with $\mathbf{a}_n = \sigma_n^n + \overline{K}_n$, $\mathbf{b}_n = \begin{bmatrix}
\sigma^n_{i_2} & \sigma^n_{i_3} & \ldots &   \sigma^n_{i_{k_n}}
\end{bmatrix}$, $\mathbf{c}_n = - \begin{bmatrix}
\mathds{I}_6 & \mathds{I}_6 & \ldots & \mathds{I}_6
\end{bmatrix}^\intercal$, and $\mathbf{d}_n =  \sigma_n^n - \overline{K}_n$.
%%%%%%% CUT £££
%%%%%%%\begin{align*}
%%%%%%%\mathbf{A}_n = \begin{bmatrix}
%%%%%%%\sigma_n^n + \overline{K}_n & \sigma^n_{i_2} & \sigma^n_{i_3} & \ldots &   \sigma^n_{i_{k_n}}\\
%%%%%%%%
%%%%%%%- \mathds{I}_6 & \mathds{I}_6 & \mathds{O}_6 & \ldots &  \mathds{O}_6 \\
%%%%%%%%
%%%%%%%-\mathds{I}_6 & \mathds{O}_6 & \mathds{I}_6 & \ddots & \vdots \\
%%%%%%%\vdots & \vdots & \ddots & \ddots  & \mathds{O}_6 \\
%%%%%%%%
%%%%%%%-\mathds{I}_6 & \mathds{O}_6 & \ldots & \mathds{O}_6  & \mathds{I}_6
%%%%%%%\end{bmatrix}, \quad
%%%%%%%\mathbf{B}_n = \begin{bmatrix}
%%%%%%%\sigma_n^n - \overline{K}_n & \sigma^n_{i_2} & \sigma^n_{i_3} & \ldots &   \sigma^n_{i_{k_n}}\\
%%%%%%%%
%%%%%%%\mathds{I}_6 & -\mathds{I}_6 & \mathds{O}_6 & \ldots &  \mathds{O}_6 \\
%%%%%%%%
%%%%%%%\mathds{I}_6 & \mathds{O}_6 & -\mathds{I}_6 & \ddots & \vdots \\
%%%%%%%\vdots & \vdots & \ddots & \ddots  & \mathds{O}_6 \\
%%%%%%%%
%%%%%%%\mathds{I}_6 & \mathds{O}_6 & \ldots & \mathds{O}_6  & -\mathds{I}_6
%%%%%%%\end{bmatrix}.
%%%%%%%\end{align*}
To show that $\mathbf{A}_n$ is invertible, let $\zeta = (\zeta_1^\intercal, \ldots, \zeta_{k_n}^\intercal)^\intercal$, where $\zeta_j \in \mathbb{R}^6$ for all $j \in \{1, \ldots, k_n\}$, and assume that $\mathbf{A}_n\zeta = 0$. Then, $\zeta_1 = \ldots = \zeta_{k_n}$ and $(\sigma^n + \overline{K}_n)\zeta_1 = 0$ (see \eqref{eq:def_sigman_sum}). The matrix $(\sigma^n + \overline{K}_n)$ has only real positive eigenvalues since it is positive definite and symmetric. In particular, it is invertible, implying that $\zeta = 0$, and $\mathbf{A}_n$ is consequently invertible. We set $\mathcal{B}_n = \mathbf{G}_n^{-1}\mathbf{A}_n^{-1}\mathbf{B}_n \mathbf{G}_n$, and \eqref{eq:def_calBn} follows from basic computations and noticing that $\mathbf{A}_n^{-1}$ is given by
\begin{linenomath}
\begin{equation*}
\mathbf{A}_n^{-1} = 
\begin{bmatrix}
(\mathbf{a}_n-\mathbf{b}_n\mathbf{c}_n)^{-1} & -(\mathbf{a}_n-\mathbf{b}_n \mathbf{c}_n)^{-1} \mathbf{b}_n \\
-\mathbf{c}_n (\mathbf{a}_n - \mathbf{b}_n \mathbf{c}_n)^{-1} & \mathds{I}_{6(k_n-1)} + \mathbf{c}_n(\mathbf{a}_n-\mathbf{b}_n \mathbf{c}_n)^{-1} \mathbf{b}_n
\end{bmatrix}.\qedhere
\end{equation*}
\end{linenomath}
\end{proof}

%%% CUT £££
%%%System \eqref{eq:syst_Riem_1} may now also be written as
%%%\begin{align} \label{eq:syst_riem_Bn}
%%%\begin{cases}
%%%\partial_t r_i + \mathbf{D}_i(x)\partial_x r_i + B_i(x) r_i = g_i(x, r_i) & \text{in }(0, \ell_i)\times(0, T), \ i \in \mathcal{I}\\
%%%r_n^{\text{out}}(t) = \mathcal{B}_n r_n^{\text{in}}(t) & t \in (0, T), \ n \in \mathcal{N}\\
%%%r_i(\xi, 0) = r^0_i(\xi) & x \in (0, \ell_i), \ i \in \mathcal{I}.
%%%\end{cases}
%%%\end{align}

Finally, let us comment on the form of the quadratic  functions $\overline{g}_i$ and $g_i$ ($i \in \mathcal{I}$).
\begin{remark} \label{rem:form_sources}
Let $i \in \mathcal{I}$. For all $x \in [0, \ell_i]$, $\mathbf{u} \in \mathbb{R}^{12}$ and $m \in \{1, \ldots, 12\}$, the components of $\overline{g}_i$ may be written as $(\overline{g}_i(x, \mathbf{u}))_m = \langle \mathbf{u} \,, \overline{G}^m_i(x) \mathbf{u} \rangle$, where $\overline{G}_i^m(x) \in \mathbb{R}^{12 \times 12}$ is a symmetric matrix (whose expression is omitted here). 
Then, the components of $g_i$ also write as $(g_i(x,\mathbf{u}))_m = \left\langle \mathbf{u} \,, G_i^m(x) \mathbf{u} \right \rangle$, for $G_i^m = \sum_{j=1}^{12} (L_i)_{m,j} (L_i^{-1})^\intercal \overline{G}_i^m L_i^{-1}$, denoting by $(L_i)_{m,j}$ the component at the $m$-th row and $j$-th column of $L_i$. 
Consequently, the nonlinearities also write as $\overline{g}_i(x, \mathbf{u}) = \overline{\mathbf{G}}_i(x, \mathbf{u})\mathbf{u}$ and $g_i(x,\mathbf{u}) = \mathbf{G}_i(x,\mathbf{u})\mathbf{u}$, for functions $\overline{\mathbf{G}}_i, \mathbf{G}_i \colon [0, \ell_i]\times \mathbb{R}^{12} \rightarrow \mathbb{R}^{12 \times 12}$ defined by $\overline{\mathbf{G}}_i(x, \mathbf{u}) = [\overline{G}_i^1(x)\mathbf{u}, \ldots, \overline{G}_i^{12}(x)\mathbf{u}]^\intercal$ and $\mathbf{G}_i (x,\mathbf{u})= [ G_i^1(x) \mathbf{u}, \ldots, G_{i}^{12}(x) \mathbf{u}]^\intercal$.
\end{remark}

We are now in position to prove the well-posedness result.

\begin{proof}[Proof of Theorem \ref{thm:well-posedness}]
Let $\ell >0$. For all $i \in \mathcal{I}$, we apply the change of variable $\tilde{r}_i(\xi, t) = r_i(\ell_i \ell^{-1} \xi, t)$ for $\xi \in [0, \ell]$ so that the governing systems have the same spatial domain $[0, \ell]$ for all $i \in \mathcal{I}$.
Then, the dynamics of $\tilde{r}_i \colon (0, \ell)\times(0, T) \rightarrow \mathbb{R}^{12}$ are given by the governing system $\partial_t \tilde{r}_i + \widetilde{\mathbf{D}}_i\partial_\xi \tilde{r}_i + \widetilde{B}_i \tilde{r}_i = \widetilde{g}_i(\cdot, \tilde{r}_i)$, where $\widetilde{\mathbf{D}}_i(\xi) = \ell_i^{-1} \ell \ \mathbf{D}_i(\ell_i \ell^{-1} \xi)$, $\widetilde{B}_i(\xi) = B_i(\ell_i \ell^{-1} \xi)$ and $\widetilde{g}_i(\xi\,, \mathbf{r} \,) = g_i(\ell_i \ell^{-1} \xi\,, \mathbf{r} \,)$, for all $\xi \in (0, \ell)$ and $\mathbf{r} \in \mathbb{R}^{12}$.
The boundary and initial conditions take the form $\tilde{r}_n^{\text{out}}(t) = \mathcal{B}_n \tilde{r}_n^{\text{in}}(t)$ and $\tilde{r}_i(\xi, 0) = \tilde{r}^0_i(\xi)$, for all $t \in (0, T)$ and $\xi \in (0, \ell)$, where the outgoing/incoming information is defined akin to \eqref{eq:r_out_in_mult_simp}, and $\tilde{r}^0_i(\cdot) = r_i^0(\ell_i \ell^{-1} \cdot)$.

Now, \eqref{eq:syst_Riem_1} can be written as a single hyperbolic system of $12N$ equations, with unknown state $\tilde{r} = (\tilde{r}_1^\intercal, \ldots, \tilde{r}_N^\intercal)^\intercal$, 
\begin{equation} \label{eq:syst_samel}
\begin{cases}
\partial_t \tilde{r} + \widetilde{\mathbf{D}}(\xi) \partial_\xi \tilde{r} + \widetilde{B}(\xi) \tilde{r} = \widetilde{g}(\xi, \tilde{r}) &\text{in }(0, \ell) \times (0, T)\\
\tilde{r}^{\text{out}}(t) = \mathcal{B} \, \tilde{r}^{\text{in}}(t) & t \in (0, T)\\
\tilde{r}(\xi, 0) = \tilde{r}^0(\xi) & \xi \in (0, \ell),
\end{cases}
\end{equation}
where $\widetilde{\mathbf{D}} = \mathrm{diag}(\widetilde{\mathbf{D}}_1, \ldots, \widetilde{\mathbf{D}}_N)$, $\widetilde{B} = \mathrm{diag}(\widetilde{B}_1, \ldots, \widetilde{B}_N)$ and $\mathcal{B} = \mathrm{diag}(\mathcal{B}_0, \mathcal{B}_1, \ldots, \mathcal{B}_N)$, the initial datum is $\tilde{r}^0 = ((\tilde{r}^0_1)^\intercal, \ldots, (\tilde{r}^0_N)^\intercal)^\intercal$, the outgoing information $\tilde{r}^\mathrm{out} = ((\tilde{r}^\mathrm{out}_0)^\intercal, \ldots, (\tilde{r}^\mathrm{out}_N)^\intercal)^\intercal$, the incoming information $\tilde{r}^\mathrm{in} = ((\tilde{r}^\mathrm{in}_0)^\intercal, \ldots, (\tilde{r}^\mathrm{in}_N)^\intercal)^\intercal$, and the source is defined by $\widetilde{g}(\xi, \mathbf{r}) = (\widetilde{g}_1(\xi, \mathbf{r}_1)^\intercal,\ldots,\widetilde{g}_N(\xi, \mathbf{r}_N)^\intercal)^\intercal$, for all $\xi \in [0, \ell]$ and $\mathbf{r} = (\mathbf{r}_1^\intercal, \ldots, \mathbf{r}_{N}^\intercal)^\intercal$ with $\mathbf{r}_j \in \mathbb{R}^{12}$ for all $j\in \{1, \ldots, N\}$.
By Remark \ref{rem:form_sources}, the latter function also takes the form $\widetilde{g}(\xi, \mathbf{r}) = \widetilde{\mathbf{G}}(\xi, \mathbf{r})\mathbf{r}$, required in Theorem \ref{thm:bastin_coron_well_posed}, where $\widetilde{\mathbf{G}}(\xi, \mathbf{r}) = \mathrm{diag}(\mathbf{G}_1(\ell_1 \ell^{-1}\xi, \mathbf{r}_1), \ldots, \mathbf{G}_{N}(\ell_N \ell^{-1}\xi, \mathbf{r}_N))$.

Let $k \in \{1, 2\}$.
The $(k-1)$-order compatibility conditions of \eqref{eq:syst_Riem_1} and \eqref{eq:syst_samel} may be defined analogously to Definition \ref{def:comp_cond} and are all equivalent to that of \eqref{eq:syst_y}.
We can now apply Theorem \ref{thm:bastin_coron_well_posed} to System  \eqref{eq:syst_samel}. The obtained result then translates to the existence of $\delta_0>0$ such that for any $r^0 \in \mathbf{H}^k_x$ satisfying $\|r^0\|_{\mathbf{H}^k_x} \leq \delta_0$ and the $(k-1)$-order compatibility conditions of \eqref{eq:syst_Riem_1}, there exists a unique solution $r \in C^0([0, T), \mathbf{H}^k_x)$ to \eqref{eq:syst_Riem_1}, and if $\|r(\cdot, t)\|_{\mathbf{H}^k_x} \leq \delta_0$ for all $t \in [0, T)$ then $T = + \infty$. Since the well-posedness of the diagonal system \eqref{eq:syst_Riem_1} is equivalent to that of the physical system \eqref{eq:syst_y}, we have proved Theorem \ref{thm:well-posedness}.
\end{proof}

\section{Proof of Theorem \ref{thm:stabilization}}
\label{sec:stab}

In Subsection \ref{subsec:stab_P} below, we prove Theorem \ref{thm:stabilization}, using the point of view of the physical system \eqref{eq:syst_y}. We then make some comments about this proof seen from the point of view of the diagonal system \eqref{eq:syst_Riem_1}, in Subsection \ref{subsec:stab_D}.

For any $d \in \{1, 2, \ldots\}$, we denote by $\mathbb{S}^d$ (and $\mathbb{D}^d$) and by $\mathbb{S}^d_{++}$ (and $\mathbb{D}^d_{++}$) the sets of symmetric (resp. diagonal) and positive definite symmetric (resp. positive definite diagonal) real matrices of size $d$, respectively.
Furthermore, we denote $\mathbf{C}^k_{x} = \prod_{i=1}^{N} C^k([0, \ell_i]; \mathbb{R}^{12})$ and $\mathbf{L}^2_{x} = \prod_{i=1}^{N} L^2(0, \ell_i; \mathbb{R}^{12})$, these spaces being endowed with the associated product norms.

\subsection{Point of view of the physical system}
\label{subsec:stab_P}

\subsubsection{Strategy and observations}

Let $k \in \{1, 2\}$.
In order to prove Theorem \ref{thm:stabilization}, we want to find a so-called \emph{quadratic $\mathbf{H}^k_x$ Lyapunov functional}, namely, a functional $\overline{\mathcal{L}}\colon [0, T] \rightarrow [0, +\infty)$ of the form
\begin{equation} \label{eq:def_bar_lyap}
\overline{\mathcal{L}}(t) = \sum_{i \in \mathcal{I}} \sum_{\alpha=0}^k \overline{\mathcal{L}}_{\alpha i}(t), \quad \text{with} \ \ \overline{\mathcal{L}}_{\alpha i}(t) := \int_0^{\ell_i} \left\langle \partial_t^\alpha y_i (x,t) \,, \overline{Q}_i(x) \partial_t^\alpha y_i (x,t) \right\rangle dx,
\end{equation}
where $\overline{Q}_i \in C^1([0, \ell_i]; \mathbb{S}^{12})$ and $y \in C^0([0, T]; \mathbf{H}_x^k)$ is solution to \eqref{eq:syst_y}, such that $\overline{\mathcal{L}}$ fulfills the assumptions of Proposition \ref{prop:existence_Lyap}, given below\footnote{As $y_i \in C^0([0, T]; H^k(0, \ell_i; \mathbb{R}^{12}))$ solves \eqref{eq:syst_y}, the governing system yields that, for all $\alpha \in \{0, \ldots, k\}$, $\partial_t^\alpha y_i$ belongs to $C^0([0, T]; L^2(0, \ell_i; \mathbb{R}^{12}))$. Hence, $\overline{\mathcal{L}}(t)$ is well defined for all $t \in [0, T ]$.}. In other words, the functional should be such that: when the solution $y$ is in some small ball of $C^0([0, T];\mathbf{C}^{k-1}_{x})$, $\overline{\mathcal{L}}(t)$ is equivalent to the $\mathbf{H}_x^k$ norm of $y(\cdot, t)$ and decays exponentially with time.
Since the arguments used to obtain the latter proposition follow closely \cite{BC2016, bastin2017exponential}, we do not provide a proof here.

\begin{proposition} \label{prop:existence_Lyap}
Let $k \in \{1, 2\}$. Assume that for any fixed $T>0$, there exists $\delta>0$ such that the following holds: if $y \in C^0([0, T]; \mathbf{H}^k_x)$ is solution to \eqref{eq:syst_y} and fulfills $\|y(\cdot, t)\|_{\mathbf{C}^{k-1}_{x}} \leq \delta$ for all $t \in [0, T]$, then there exists $\eta \geq 1$ and $\beta >0$ such that 
\begin{linenomath}
\begin{align}
\eta^{-1} \|y(\cdot, t)\|_{\mathbf{H}^k_x}^2 &\leq \overline{\mathcal{L}}(t) \leq \eta \|y(\cdot, t)\|_{\mathbf{H}^k_x}^2,  &&\text{for all } t \in [0, T] \label{eq:lyap_equiv}\\
\overline{\mathcal{L}}(t) &\leq e^{-2 \beta t} \overline{\mathcal{L}}(0), &&\text{for all }t \in [0, T]. \label{eq:lyap_decay}
\end{align}
\end{linenomath}
Then, the steady state $y\equiv 0$ of \eqref{eq:syst_y} is locally $\mathbf{H}^k_x$ exponentially stable.
\end{proposition}

To find such a Lyapunov functional, we start by observing the energy $\mathcal{E}^\mathcal{P}_i$ ($i\in\mathcal{I}$) for beams described by System \eqref{eq:syst_y}. By definition, it is the sum of the kinetic and elastic energies, and -- for the kind of beam considered here -- it takes the form 
\begin{equation} \label{eq:def_calEP}
\mathcal{E}_i^\mathcal{P}(t) = \int_0^{\ell_i} \langle y_i(x,t) \,, Q_i^\mathcal{P}(x) y_i(x,t) \rangle dx
\end{equation}
for $y_i$ ($i\in \mathcal{I}$) solution to \eqref{eq:syst_y}, and where $Q_i^\mathcal{P} = \mathrm{diag}(\mathbf{M}_i, \mathbf{C}_i)$. 
Moreover, one may easily verify that the product $Q_i^\mathcal{P}A_i$ is skew-symmetric and that $Q^\mathcal{P}_i\overline{B}_i$ is equal to
\begin{equation} \label{eq:prod_QPAi}
Q^\mathcal{P}_i\overline{B}_i = \begin{bmatrix}
\mathds{O}_6 & \mathds{I}_6\\ \mathds{I}_6 & \mathds{O}_6
\end{bmatrix}.
\end{equation}
Here, the notation $\mathcal{P}$ stresses that we are referring to System \eqref{eq:syst_y}, whose unknown state is the \emph{"physical variable"} $y_i$, as opposed to the \emph{"diagonal variable"} $r_i$ for the system written in Riemann invariants.
The velocity feedback controls have been introduced in System \eqref{eq:syst_y} in such a way that the energy $\mathcal{E}^\mathcal{P} = \sum_{i \in \mathcal{I}} \mathcal{E}_i^\mathcal{P}$ of the whole network is dissipated. Indeed, one can check that for any positive semi-definite matrices $K_n$ ($n \in \mathcal{N}$) one has $\frac{\mathrm{d}}{\mathrm{d}t}\mathcal{E}^\mathcal{P}(t) \leq 0$, for all $t \in [0, T]$, if $y$ is solution to \eqref{eq:syst_y} in $[0, T]$.

We start by giving, in Lemma \ref{lem:class_barQi} below, a series of properties on $Q_i$ $(i\in\mathcal{I})$, which are sufficient for the associated functional $\overline{\mathcal{L}}$ to fulfill the assumptions of Proposition \ref{prop:existence_Lyap}.

\begin{lemma} \label{lem:class_barQi}
Assume that there exists $\overline{Q}_i \in C^1([0, \ell_i]; \mathbb{R}^{12 \times 12})$ $(i \in \mathcal{I})$ fulfilling:
% the following properties:
\begin{enumerate}[label=(\roman*)]
\item \label{pty:barQi_symm_posDef}
for all $x \in [0, \ell_i]$ and all $i \in \mathcal{I}$, $\bar{Q}_i(x)$ is symmetric and positive definite;
\item \label{pty:barQiA_i_symm}
for all $x \in [0, \ell_i]$ and all $i \in \mathcal{I}$, $(\overline{Q}_iA_i)(x)$ is symmetric;
\item \label{pty:barSi_negDef}
for all $x \in [0, \ell_i]$ and all $i \in \mathcal{I}$, $\overline{S}_i(x)$ is negative definite, \\
where $\bar{S}_i$ is defined by $\overline{S}_i = \frac{\mathrm{d}}{\mathrm{d}x}( \overline{Q}_i A_i ) - \overline{Q}_i\overline{B}_i - \overline{B}_i^\intercal\overline{Q}_i$;
\item \label{pty:barcalR_nonPos}
for all $y \in C^0([0, T]; \mathbf{C}^0_x)$ fulfilling the nodal conditions of \eqref{eq:syst_y} and all $t \in [0, T]$, $\overline{\mathcal{R}}(y,t)\leq 0$, where $\bar{\mathcal{R}}$ is defined by $\overline{\mathcal{R}} = \overline{\mathcal{R}}_0 + \overline{\mathcal{R}}_M + \overline{\mathcal{R}}_S$, with
\begin{linenomath}
\begin{equation*}
\begin{aligned}
&\overline{\mathcal{R}}_0(y,t) = \left\langle y_1 \,, \overline{Q}_1 A_1 y_1 \right\rangle (0, t), \qquad \overline{\mathcal{R}}_S(y,t) = - {\textstyle \sum_{n \in \mathcal{N}_S \setminus \{0\}}} \left\langle y_n \,, \overline{Q}_n A_n y_n\right\rangle (\ell_n, t),\\
&\overline{\mathcal{R}}_M(y,t) = {\textstyle \sum_{n \in \mathcal{N}_M}} \big[- \left\langle y_n \,, \overline{Q}_n A_n y_n \right \rangle (\ell_n, t) + {\textstyle \sum_{i \in \mathcal{I}_n} } \left\langle y_i \,, \overline{Q}_i A_i y_i \right \rangle (0, t)\big].
\end{aligned}
\end{equation*}
\end{linenomath}
\end{enumerate}
Then, the associated $\overline{\mathcal{L}}$ $($see \eqref{eq:def_bar_lyap}$)$ fulfills the assumptions of Proposition \ref{prop:existence_Lyap}.
\end{lemma}

Before the proof of this lemma, let us make an additional remark on the quadratic function $\overline{g}_i$, for $i \in \mathcal{I}$.
%%%
%%%
%%%
We observe that, for any $u \in C^2([0, \ell_i]\times[0, T]; \mathbb{R}^{12})$,
\begin{linenomath}
\begin{align} 
&\partial_t (\overline{g}_i(x, u)) = 2 \overline{\mathbf{G}}_i(x, u)\partial_t u, \quad \partial_t^2 (\overline{g}_i(x, u)) = 2 (\overline{\mathbf{G}}_i(x, u)\partial_t^2 u + \overline{\mathbf{G}}_i(x, \partial_t u)\partial_t u), \nonumber \\
&\partial_x (\overline{g}_i(x,u)) = 2 \overline{\mathbf{G}}_i(x,u)\partial_x u + (\partial_x \overline{\mathbf{G}}_i)(x,u)u \label{eq:deriv_bargi}
\end{align}
\end{linenomath}
where $\overline{\mathbf{G}}_i$ is defined in Remark \ref{rem:form_sources}.
Moreover, there exists a constant $C>0$ depending only on the beam parameters (i.e. $\mathbf{M}_i, \mathbf{C}_i$, for $i \in \mathcal{I}$) such that $\|\overline{\mathbf{G}}_i(x,\mathbf{u})\| \leq C |\mathbf{u}|$ and $\|(\partial_x \overline{\mathbf{G}}_i)(x,\mathbf{u})\| \leq C |\mathbf{u}|$, for any $x \in [0, \ell_i]$ and $\mathbf{u} \in \mathbb{R}^{12}$.

\begin{proof}[Proof of Lemma \ref{lem:class_barQi}]
Similary to Proposition \ref{prop:existence_Lyap}, we use the arguments of \cite{BC2016, bastin2017exponential}, developed for a single hyperbolic system in diagonal form, for our tree-shaped network expressed in physical variable.

Let $k \in \{1, 2\}$. Assume that $y \in C^0([0, T]; \mathbf{H}^k_x)$ is solution to \eqref{eq:syst_y} and that, for some $\delta >0$, $\|y(\cdot, t)\|_{\mathbf{C}^{k-1}_{x}} \leq \delta$ for all $t \in [0, T]$.
Assume temporarily that $y$ is of regularity $C^k$ in $\prod_{i=1}^{N} [0, \ell_i] \times [0, T]$.
Since $\overline{Q}_i(x)$ is positive definite for any $x \in [0, \ell_i]$ and $i \in \mathcal{I}$ (see \ref{pty:barQi_symm_posDef}), there exists $C_1 \geq 1$ (depending only on $Q_i$, $i \in \mathcal{I}$) such that 
\begin{equation}\label{eq:equiv_barLyap_L2timeder}
\frac{1}{C_1} \Big\| \sum_{\alpha=0}^k \partial_t^\alpha y(\cdot, t) \Big\|_{\mathbf{L}^2_x}^2 \leq \overline{\mathcal{L}}(t) \leq C_1 \Big\|\sum_{\alpha=0}^k \partial_t^\alpha y(\cdot, t) \Big\|_{\mathbf{L}^2_x}^2.
\end{equation}
The governing system of \eqref{eq:syst_y} being of first order in space and time, it yields a relationship between $\partial_t y_i$ and $\partial_x y_i$. Indeed, using that $\partial_t y_i = - A_i \partial_x y_i - \overline{B}_i y_i + \overline{g}_i(\cdot, y_i)$ and $\partial_x y_i = A_i^{-1}(-\partial_t y_i - \overline{B}_i y_i + \overline{g}_i(\cdot, y_i))$, and also differentiating these two systems in time and space respectively, one deduces that there exists a constant $C_2\geq 1$ depending only on the beam parameters (i.e. $\mathbf{M}_i$, $\mathbf{C}_i$, for $i\in \mathcal{I}$) and on the $C^0([0, T];\mathbf{C}^0_x)$ norm of $y$ (in particular $C_2$, depends on $\delta$), such that
\begin{linenomath}
\begin{equation*}
\frac{1}{C_2}  \sum_{\alpha=0}^k |\partial_x^\alpha y_i|^2 \leq \sum_{\alpha=0}^k |\partial_t^\alpha y_i|^2 \leq C_2 \sum_{\alpha=0}^k |\partial_x^\alpha y_i|^2
\end{equation*}
\end{linenomath}
in $[0, \ell_i] \times [0, T]$, for all $i \in \mathcal{I}$.
Hence, \eqref{eq:lyap_equiv} is fulfilled with this $y$.
Let $\alpha \in \{1, \ldots k\}$ and $i\in\mathcal{I}$.
One has $\frac{\mathrm{d}}{\mathrm{d}t}\overline{\mathcal{L}}_{\alpha i}(t) = 2 \int_0^{\ell_i} \left\langle \partial_t^\alpha y_i(x, t) \,, \overline{Q}_i(x) \partial_t^{\alpha+1} y_i(x,t)\right\rangle dx$, since $\overline{Q}_i(x)$ is symmetric (see \ref{pty:barQiA_i_symm}).
%%%\begin{align*}
%%%\frac{\mathrm{d}}{\mathrm{d}t}\overline{\mathcal{L}}_{\alpha i}(t) = 2 \int_0^{\ell_i} \left\langle \partial_t^\alpha y_i(x, t) \,, \overline{Q}_i(x) \partial_t^{\alpha+1} y_i(x,t)\right\rangle dx.
%%%\end{align*}
Below, we drop the arguments $t$ and $x$ for clarity. From the governing system, integration by parts, and \ref{pty:barQiA_i_symm}, one obtains
\begin{equation} \label{eq:dtbarL0i}
\frac{\mathrm{d}}{\mathrm{d}t}\overline{\mathcal{L}}_{\alpha i} =  \int_0^{\ell_i} \Big\langle \partial_t^\alpha y_i \,, \overline{S}_i \partial_t^\alpha y_i + 2\overline{Q}_i \partial_t^\alpha \big(\overline{g}_i(\cdot, y_i)\big) \Big\rangle dx - \Big[\left\langle \partial_t^\alpha y_i \,, \overline{Q}_i A_i \partial_t^\alpha y_i \right\rangle \Big]_0^{\ell_i}.
\end{equation}
Taking into account \eqref{eq:deriv_bargi} and recalling that -- because of the governing system -- one has $|\partial_t y_i| \leq C (|\partial_x y_i|+|y_i|+|y_i|^2)$ for some constant $C>0$ depending only on the beam parameters, we deduce that there exists $C_{3i}>0$, depending only on $\overline{Q}_i$ and the beam parameters, such that 
\begin{equation} \label{eq:dtbarL0i_g}
\sum_{\alpha=0}^k \left\langle \partial_t^\alpha y_i \,, 2 \overline{Q}_i \partial_t^\alpha (\overline{g}_i(x, y_i)) \right\rangle  \leq C_{3i} (\delta + \delta^2) \sum_{\alpha=0}^k |\partial_t^\alpha y_i|^2.
\end{equation}
Let the negative constant $C_{4i}<0$ be the maximum in $[0, \ell_i]$ of the largest eigenvalue of $\overline{S}_i$ (see \ref{pty:barSi_negDef}). By \eqref{eq:dtbarL0i}-\eqref{eq:dtbarL0i_g}, the derivative of $\overline{\mathcal{L}}$ fulfills
\begin{equation} \label{eq:barcalL_ineq}
\frac{\mathrm{d}}{\mathrm{d}t}\overline{\mathcal{L}} \leq \max_{i\in\mathcal{I}}(C_{3i}(\delta + \delta^2) + C_{4i}) \Big\|\sum_{\alpha=0}^k \partial_t^\alpha y \Big\|^2_{\mathbf{L}^2_x} + \sum_{i\in\mathcal{I}} \sum_{\alpha=0}^k \big[\left\langle \partial_t^\alpha y_i \,, \overline{Q}_i A_i \partial_t^\alpha y_i \right\rangle\big]_0^{\ell_i}.
\end{equation}
Hence, choosing $\delta>0$ small enough and using \eqref{eq:equiv_barLyap_L2timeder}, we deduce that there exists a positive constant $\beta>0$ such that the first term in \eqref{eq:barcalL_ineq} is less than or equal to $-2\beta \overline{\mathcal{L}}(t)$. Finally, one can observe that the second term in \eqref{eq:barcalL_ineq} is equal to $\sum_{\alpha=0}^k \overline{\mathcal{R}}(\partial_t^{\alpha} y,t)$, which is nonpositive here (see \ref{pty:barcalR_nonPos}) since $\partial_t^\alpha y$ fulfills the nodal conditions of System \eqref{eq:syst_y} for any $\alpha \in \{ 0, \ldots, k\}$. Gronwall's inequality allows to conclude that \eqref{eq:lyap_decay} holds.
By a density argument similar to \cite[Comment 4.6]{BC2016}, the estimates \eqref{eq:lyap_equiv}-\eqref{eq:lyap_decay} remain valid for $y \in C^0([0, T]; \mathbf{H}^k_x)$.
\end{proof}

\subsubsection{Lemmas and Proof of Theorem \ref{thm:stabilization}}

Having clarified which kind of functions $\overline{Q}_i$ (for $i \in \mathcal{I}$) we are looking for, we may now proceed with the main part of the proof of Theorem \ref{thm:stabilization}, which is to show the existence of such functions. Let us first give a short lemma of use in what follows.

\begin{lemma} \label{lem:exist_g}
Let $\eta>0$, $\ell>0$ be fixed. 
\begin{enumerate}[label=\alph*)]
\item \label{item:exist_g_neg}
For any choice of constants $a < b\leq 0$, there exists $q\in C^\infty([0, \ell])$ such that $q(0) = a$, $q(\ell) = b$ and $\frac{\mathrm{d}}{\mathrm{d}x}q(x) > (q(\ell)-q(x)) \eta$, for all $x \in [0, \ell]$.
\item \label{item:exist_g_pos}
For any choice of constants $0 \leq a < b$, there exists $q\in C^\infty([0, \ell])$ such that $q(0) = a$, $q(\ell) = b$ and $\frac{\mathrm{d}}{\mathrm{d}x}q(x) > (q(x) - q(0))\eta$, for all $x \in [0, \ell]$.
\end{enumerate}
\end{lemma}

\begin{proof}[Proof of Lemma \ref{lem:exist_g}]
Starting with \ref{item:exist_g_neg}, $\frac{\mathrm{d}}{\mathrm{d}x}q(x) > \eta (b-q(x))$ is equivalent, for some $\varepsilon>0$, to $\frac{\mathrm{d}}{\mathrm{d}x} \left( e^{\eta x} (q(x)-b) \right) \geq \varepsilon$. Integrating the latter on $[0, x]$, it is equivalent to
\begin{equation} \label{eq:ineq_exist_g_neg}
e^{\eta x} (q(x)-b) - (a-b) \geq \varepsilon x.
\end{equation}
Hence, choosing $q$ defined by $q(x) = e^{-\eta x}(a-b+\varepsilon x) + b$ with $\varepsilon = \ell^{-1}(b - a)$, the function $q$ satisfies both \eqref{eq:ineq_exist_g_neg} (with equality) and $q(0)=a$, $q(\ell) = b$.
The proof of \ref{item:exist_g_pos} is similar: observing that $\frac{\mathrm{d}}{\mathrm{d}x}q(x) > \eta (q(x) - a)$ is equivalent to $\frac{\mathrm{d}}{\mathrm{d}x} \left( e^{-\eta x} (q(x)-a) \right) \geq \varepsilon$ for some $\varepsilon>0$, we integrate the latter inequality on $[0, x]$ and choose $q(x) = a + e^{\eta x}\varepsilon x$ with $\varepsilon = e^{-\eta\ell}\ell^{-1}(b - a)$.
\end{proof}

We may now give the proof of Theorem \ref{thm:stabilization}.

\begin{proof}[Proof of Theorem \ref{thm:stabilization}]
We are looking for functions $\overline{Q}_i$ ($i \in \mathcal{I}$) fulfilling the properties \ref{pty:barQi_symm_posDef}-\ref{pty:barQiA_i_symm}-\ref{pty:barSi_negDef}-\ref{pty:barcalR_nonPos} of Lemma \ref{lem:class_barQi}. Let $i \in \mathcal{I}$.

\subsubsection*{Step 1: Ansatz}
First, we choose an Ansatz for $\overline{Q}_i$.
Our choice rests on the form of the energy of the beam $\mathcal{E}_i^\mathcal{P}$, whose definition (see \eqref{eq:def_calEP}) depends on the \emph{"energy matrix"} $Q^\mathcal{P}_i$: we multiply $Q^\mathcal{P}_i$ by a constant $\rho_i$, and add extradiagonal terms $\mathbf{W}_i$, multiplied by a "\emph{weight function}" $w_i$. More precisely, we look for $\overline{Q}_i$ of the form
\begin{equation} \label{eq:ansatz_Qibar}
\overline{Q}_i = \rho_i Q_i^\mathcal{P} + Q_i^\mathrm{ex}, \qquad \text{with} \ \ Q_i^\mathrm{ex} = w_i \begin{bmatrix}
\mathds{O}_6 & \mathbf{W}_i \\ \mathbf{W}_i^\intercal & \mathds{O}_6
\end{bmatrix},
\end{equation}
for some $\rho_i \in \mathbb{R}$, $w_i \in C^1([0, \ell_i])$ and $\mathbf{W}_i \in C^1([0, \ell_i]; \mathbb{R}^{6 \times 6})$.
Then, $Q_i(x)$ is by definition symmetric, for all $x \in [0, \ell_i]$.
Since the product $Q_i^\mathcal{P} A_i$ takes the form \eqref{eq:prod_QPAi}, one has $\tfrac{\mathrm{d}}{\mathrm{d}x} \left(\overline{Q}_i A_i \right) = \tfrac{\mathrm{d}}{\mathrm{d}x} \left(Q_i^\mathrm{ex} A_i \right)$; and since $Q_i^\mathcal{P}\overline{B}_i$ is skew-symmetric, one has $\overline{Q}_i\overline{B}_i + (\overline{Q}_i\overline{B}_i)^\intercal = Q_i^\mathrm{ex}\overline{B}_i + (Q_i^\mathrm{ex}\overline{B}_i)^\intercal$. Hence, we obtain
\begin{linenomath}
\begin{equation*}
\begin{aligned}
\tfrac{\mathrm{d}}{\mathrm{d}x} \left(\overline{Q}_i A_i \right)
= \tfrac{\mathrm{d}}{\mathrm{d}x} \big( - w_i &\mathrm{diag}\left(\mathbf{W}_i\mathbf{C}_i^{-1}, \mathbf{W}_i^\intercal \mathbf{M}_i^{-1} \right) \big),\\
\overline{Q}_i\overline{B}_i + (\overline{Q}_i\overline{B}_i)^\intercal 
= w_i \mathrm{diag}\big( \mathbf{W}_i \mathbf{C}_i^{-1}\mathbf{E}_i^\intercal &+ \mathbf{E}_i \mathbf{C}_i^{-1} \mathbf{W}_i^\intercal, - \mathbf{W}_i^\intercal \mathbf{M}_i^{-1} \mathbf{E}_i - \mathbf{E}_i^\intercal \mathbf{M}_i^{-1} \mathbf{W}_i \big).
\end{aligned}
\end{equation*}
\end{linenomath}
Consequently, $\overline{S}_i = - \frac{\mathrm{d}w_i}{\mathrm{d}x} \Lambda_i + |w_i| \Xi_i$, where $\Lambda_i$ is defined by $\Lambda_i = \mathrm{diag}\left( \Lambda_i^\mathrm{I}\,, \Lambda_i^\mathrm{II} \right)$, with $\Lambda_i^\mathrm{I} = \mathbf{W}_i \mathbf{C}_i^{-1}$ and  $\Lambda_i^\mathrm{II} = \mathbf{W}_i^\intercal \mathbf{M}_i^{-1}$,
%%%\begin{align*}
%%%\Lambda_i &= \mathrm{diag}\left( \Lambda_i^\mathrm{I}\,, \Lambda_i^\mathrm{II} \right), \qquad \text{with }\Lambda_i^\mathrm{I} = \mathbf{W}_i \mathbf{C}_i^{-1}, \quad \Lambda_i^\mathrm{II} = \mathbf{W}_i^\intercal \mathbf{M}_i^{-1},
%%%\end{align*}
and where $\Xi_i$ is defined by
\begin{linenomath}
\begin{equation*}
\Xi_i = \mathrm{sign}(w_i) \mathrm{diag} \left(\Lambda_i^\mathrm{I}\mathbf{E}_i^\intercal + (\Lambda_i^\mathrm{I}\mathbf{E}_i^\intercal)^\intercal - \tfrac{\mathrm{d}}{\mathrm{d}x} \Lambda_i^\mathrm{I} \,, - \Lambda_i^\mathrm{II} \mathbf{E}_i - (\Lambda_i^\mathrm{II} \mathbf{E}_i)^\intercal  - \tfrac{\mathrm{d}}{\mathrm{d}x}\Lambda_i^\mathrm{II} \right).
\end{equation*}
\end{linenomath}
We have written $\overline{S}_i(x)$ as the sum of two matrices $ - (\frac{\mathrm{d}w_i}{\mathrm{d}x} \Lambda_i)(x) $ and $(|w_i| \Xi_i)(x)$, and the latter matrix may be indefinite because of the presence of $\mathbf{E}_i$ (see \eqref{eq:def_E_bold}) in its expression. This is why we will, latter on, make the first matrix negative definite -- by adding assumptions on $w_i$ and $\mathbf{W}_i$ -- and large enough in comparison to the second matrix, for \ref{pty:barSi_negDef} to hold.

\subsubsection*{Step 2: Constraints on $\mathbf{W}_i$}
In view of \ref{pty:barQiA_i_symm}-\ref{pty:barSi_negDef} and the above observations, we look for a function $\mathbf{W}_i$ satisfying 
\begin{equation} \label{eq:cond_Wi}
\Lambda_i^\mathrm{I}(x) \in \mathbb{S}^6_{++}, \qquad \Lambda_i^\mathrm{II}(x)\in \mathbb{S}^6_{++}, \qquad \text{for all }x \in [0, \ell_i],
\end{equation}
so that not only \ref{pty:barQiA_i_symm} is fulfilled, but also $(\overline{\Lambda}_i)(x)$ is symmetric and positive definite.
One may verify that possible $\mathbf{W}_i$ fulfilling \eqref{eq:cond_Wi} are
\begin{linenomath}
\begin{align}
\mathbf{W}_i &= \mathds{I}_6, \label{eq:boldWi_I}\\
\mathbf{W}_i &= \mathbf{M}_i \mathbf{C}_i, \label{eq:boldWi_MC}\\
\label{eq:boldWi_MCfrac}
\mathbf{W}_i &= \mathbf{C}_i^{-\sfrac{1}{2}}(\mathbf{C}_i^{\sfrac{1}{2}}\mathbf{M}_i\mathbf{C}_i^{\sfrac{1}{2}})^{\sfrac{1}{2}} \mathbf{C}_i^{\sfrac{1}{2}},
\end{align}
\end{linenomath}
%%% CUT £££
%%%as in that case $\Lambda_i^\mathrm{I} = \mathbf{C}_i^{-\sfrac{1}{2}}(\mathbf{C}_i^{\sfrac{1}{2}} \mathbf{M}_i \mathbf{C}_i^{\sfrac{1}{2}})^{\sfrac{1}{2}} \mathbf{C}_i^{-\sfrac{1}{2}}$ and $\Lambda_i^\mathrm{II} = \mathbf{C}_i^{\sfrac{1}{2}}(\mathbf{C}_i^{\sfrac{1}{2}} \mathbf{M}_i \mathbf{C}_i^{\sfrac{1}{2}})^{-\sfrac{1}{2}} \mathbf{C}_i^{\sfrac{1}{2}}$, which are both symmetric and positive definite.
where the latter choice \eqref{eq:boldWi_MCfrac} also writes as $\mathbf{W}_i = \mathbf{M}_i^{\sfrac{1}{2}}\mathbf{C}_i^{\sfrac{1}{2}}$ for commuting $\mathbf{M}_i, \mathbf{C}_i$.

\subsubsection*{Step 3: Constraints on $w_i$}
To render $-(\frac{\mathrm{d}w_i}{\mathrm{d}x} \Lambda_i)(x)$ negative definite for all $x \in [0, \ell_i]$, we require $w_i$ to be increasing. Then, a sufficient conditions for \ref{pty:barSi_negDef} to hold is that 
\begin{equation} \label{eq:cond_barSi_neg}
\frac{\mathrm{d}w_i}{\mathrm{d}x}> C_{\Xi_i} C_{\Lambda_i}^{-1} |w_i|,
\end{equation}
where $C_{\Lambda_i} >0$ and $C_{\Xi_i} \geq 0$ are the maximum over $[0, \ell_i]$ of the smallest eigenvalue of $\Lambda_i(x)$ and largest eigenvalue of $\Xi_i(x)$, respectively.
This follows from Weyl's Theorem  \cite[Th. 4.3.1, Coro. 4.3.15]{horn2012matrix} which provides bounds on the eigenvalues of the sum of Hermitian matrices.
Moreover, \ref{pty:barQi_symm_posDef} is equivalent to assuming that $\rho_i >0$ and the Schur complement $(\overline{Q}_i / (\rho_i \mathbf{C}_i))(x) = (\rho_i \mathbf{M}_i - w_i^2 \rho_i^{-1} \mathbf{W}_i \mathbf{C}_i^{-1} \mathbf{W}_i^\intercal)(x)$ is positive definite for all $x \in [0, \ell_i]$;
the latter being equivalent to the inequality 
%%%$\rho_i > |w_i(x)| \lambda_{\overline{Q}_i(x)}^{\sfrac{1}{2}}$ holding for all $x \in [0, \ell_i]$,
% $(\mathds{I}_6 - w_i^2 \rho_i^{-2} \mathbf{M}_i^{-\sfrac{1}{2}} \mathbf{W}_i \mathbf{C}_i^{-1} \mathbf{W}_i^\intercal \mathbf{M}_i^{-\sfrac{1}{2}})(x) \in \mathbb{S}^6_{++}$ for all $x \in [0, \ell_i]$, which holds if and only if
\begin{linenomath} 
\begin{equation*}
\rho_i > |w_i(x)| \sqrt{\lambda_{\theta_i(x)}}, \qquad \text{for all } x \in [0, \ell_i],
\end{equation*}
\end{linenomath}
where $\lambda_{\theta_i(x)}$ denotes the largest eigenvalue of the matrix $\theta_i(x)$ defined by $\theta_i = \mathbf{M}_i^{-\sfrac{1}{2}} \mathbf{W}_i \mathbf{C}_i^{-1} \mathbf{W}_i^\intercal \mathbf{M}_i^{-\sfrac{1}{2}}$.
Hence, denoting $C_{\theta_i} = \max_{x \in [0, \ell_i]} \lambda_{\theta_i(x)}$, a sufficient condition for \ref{pty:barQi_symm_posDef} to be satisfied is
\begin{equation} \label{eq:cond_barQi_pos}
|w_i(x)| < \rho_i \sqrt{ C_{\theta_i}}, \qquad \text{for all } x \in [0 ,\ell_i]
\end{equation}
Note that if $\mathbf{W}_i$ is defined by \eqref{eq:boldWi_MCfrac}, then $\theta_i(x) = \mathds{I}_6$ and $\lambda_{\theta_i}(x) = 1$ for all $x \in [0, \ell_i]$.

\subsubsection*{Step 4.1: Boundary terms (case of tree-shaped networks)}
In the preceding steps, we saw that the extradiagonal terms $Q_i^\mathrm{ex}(x)$ are of help to make the matrix $\overline{S}_i(x)$ negative definite under some assumptions. However, since they are also present in the boundary terms stored in $\overline{\mathcal{R}}$, the choice of $w_i$ and $\mathbf{W}_i$ is further constrained. 
We start by studying $\overline{\mathcal{R}}$ for a general tree-shaped network, and will afterwards -- in Step 4.2 -- focus on the star-shaped case. 
For any $i \in \mathcal{I}$, one has $\langle y_i \,, \overline{Q}_i A_i y_i \rangle = -2\rho_i \langle v_i \,, z_i \rangle - w_i \langle v_i\,, \Lambda_i^\mathrm{I} v_i \rangle - w_i \langle z_i \,, \Lambda_i^\mathrm{II} z_i\rangle $. Hence,
\begin{linenomath}
\begin{equation*}
\begin{aligned}
\overline{\mathcal{R}}_0(y,t) &= \left(-2\rho_1 \left \langle v_1 \,, z_1 \right\rangle - \left \langle v_1 \,, w_1 \Lambda_i^\mathrm{I} v_1 \right \rangle -  \left \langle z_1 \,, w_1 \Lambda_i^\mathrm{II} z_1 \right \rangle \right) (0, t),\\
\overline{\mathcal{R}}_M(y,t) &= { \textstyle \sum_{n \in \mathcal{N}_M}} \big[ \left(2\rho_n \left \langle v_n \,, z_n\right \rangle + \left \langle v_n \,, w_n \Lambda_n^\mathrm{I} v_n \right \rangle + \left \langle z_n \,, w_n \Lambda_n^\mathrm{II} z_n\right \rangle \right) (\ell_n, t)\\
&\quad - {\textstyle \sum_{i \in \mathcal{I}_n} } \left( 2\rho_i \left \langle v_i \,, z_i \right \rangle + \left \langle v_i \,, w_i \Lambda_i^\mathrm{I} v_i\right \rangle + \left \langle z_i \,, w_i \Lambda_i^\mathrm{II} z_i \right \rangle \right)(0, t)  \big],\\
\overline{\mathcal{R}}_S(y,t) &= { \textstyle \sum_{n \in \mathcal{N}_S \setminus \{0\}}} \left( 2\rho_n \left \langle v_n \,, z_n \right \rangle + \left \langle v_n \,, w_n \Lambda_n^\mathrm{I} v_n\right \rangle  + \left \langle z_n\,, w_n \Lambda_n^\mathrm{II} z_n\right \rangle \right) (\ell_n, t).
\end{aligned}
\end{equation*}
\end{linenomath}
Let us first focus on multiple nodes $n \in \mathcal{N}_M$.  One can use the transmission conditions, to rewrite the term $2\rho_n  \langle v_n \,, z_n \rangle (\ell_n, t) - 2 \sum_{i \in \mathcal{I}_n} \rho_i \langle v_i \,, z_i \rangle (0, t)$, if the constant $\rho_i$ is the same for all incident edges $i \in \mathcal{I}_n \cup \{n\}$.
This is why we henceforth assume that $\rho_i = \rho >0$ for all $i \in \mathcal{I}$.
Indeed, the continuity condition and the fact that $\overline{R}_i$ is unitary, for all $i \in \mathcal{I}$ and $x \in [0, \ell_i]$, yield that
\begin{linenomath}
\begin{equation*}
 \langle v_n \,, z_n \rangle (\ell_n, t) - {\displaystyle \sum_{i \in \mathcal{I}_n} } \langle v_i \,, z_i \rangle (0, t) = \Big\langle (\overline{R}_n v_n)(\ell_n, t) \,,  (\overline{R}_n z_n)(\ell_n, t) - {\displaystyle\sum_{i \in \mathcal{I}_n} } (\overline{R}_i z_i)(0, t) \Big\rangle
\end{equation*}
\end{linenomath}
while the Kirchhoff condition yields that the above right-hand side is equal to $- \left\langle v_n \,, \overline{K}_n v_n \right\rangle (\ell_n, t)$, where $\bar{K}_n$ is defined by \eqref{eq:def_barKn}.
%%%\begin{align*}
%%%2\rho \big[ \langle v_n \,, z_n \rangle (\ell_n, t) - {\textstyle \sum_{i \in \mathcal{I}_n}}  \langle v_i \,, z_i \rangle (0, t) \big] = -2 \rho \left\langle v_n \,, \overline{K}_n v_n \right\rangle (\ell_n, t).
%%%\end{align*}
Hence, in $\overline{\mathcal{R}}_M$, we have replaced quantities of unknown sign by a nonpositive term:
\begin{equation}
\label{eq:barcalR_M}
\begin{aligned}
\overline{\mathcal{R}}_M&(y,t) = {\textstyle \sum_{n \in \mathcal{N}_M}} \big[ \left(\left \langle v_n \,, w_n \Lambda_n^\mathrm{I} v_n\right \rangle  + \left \langle z_n \,, w_n \Lambda_n^\mathrm{II} z_n \right \rangle \right) (\ell_n, t)  \\
&-2 \rho \left \langle v_n \,, \overline{K}_n v_n \right \rangle (\ell_n, t) - {\textstyle \sum_{i \in \mathcal{I}_n}}  \left( \left \langle v_i \,, w_i \Lambda_i^\mathrm{I} v_i\right \rangle + \left \langle z_i \,, w_i \Lambda_i^\mathrm{II} z_i\right \rangle \right) (0, t) \big].
\end{aligned}
\end{equation}
Let us now consider the node $n=0$. If this node is clamped then
\begin{equation} \label{eq:barcalR_0_clamped}
\overline{\mathcal{R}}_0(y,t) = - \left \langle z_1 \,, w_1 \Lambda_1^\mathrm{II} z_1\right \rangle (0, t).
\end{equation}
If a control is applied, $\overline{\mathcal{R}}_0(y,t) = \left \langle v_1 \,, (-2\rho K_0 - w_1 \Lambda_1^\mathrm{I} - w_1 K_0 \Lambda_1^\mathrm{II} K_0) v_1 \right\rangle (0, t)$ due to the nodal condition.
Finally, for remaining simple nodes $n \in \mathcal{N}_S\setminus \{0\}$, the nodal conditions yield that
\begin{equation} \label{eq:barcalR_S_not0}
\overline{\mathcal{R}}_S(y,t) = {\textstyle \sum_{n \in \mathcal{N}_S \setminus \{0\}} } \left \langle v_n \,, (- 2 \rho K_n + w_n \Lambda_n^\mathrm{I} + w_n K_n \Lambda_n^\mathrm{II} K_n) v_n\right \rangle (\ell_n, t).
\end{equation}

\subsubsection*{Step 4.2: Boundary terms (case of star-shaped networks)}
We now focus on the specific network considered inTheorem \ref{thm:stabilization}. As a star-shaped network, it is such that $\mathcal{N}_M = \{1\}$ and $\mathcal{I}_1 = \{2, 3, \ldots, N\}$. Furthermore, we have assumed that $K_1= \mathds{O}_6$ and $K_n \in \mathbb{S}_{++}^6$ for all $n \in \mathcal{N}_S$. Hence, the boundary terms take the form
\begin{linenomath}
\begin{equation*}
\begin{aligned}
\overline{\mathcal{R}}(y,t) &= 
\left \langle v_1 \,, (-2\rho K_0 - w_1 \Lambda_1^\mathrm{I} -w_1 K_0 \Lambda_1^\mathrm{II} K_0) v_1 \right \rangle (0, t) + \left \langle v_1 \,, w_1 \Lambda_1^\mathrm{I} v_1\right \rangle (\ell_1, t) \\
&\quad   + \left \langle z_1 \,, w_1 \Lambda_1^\mathrm{II} z_1 \right \rangle (\ell_1, t) + \sum_{i=2}^N \Big[ \left \langle v_i \,, (- 2 \rho K_i + w_i \Lambda_i^\mathrm{I} + w_i K_i \Lambda_i^\mathrm{II} K_i ) v_i\right \rangle (\ell_i, t)\\
&\quad - \left \langle v_i \,, w_i \Lambda_i^\mathrm{I} v_i \right \rangle (0, t) - \left \langle z_i \,, w_i \Lambda_i^\mathrm{II} z_i \right \rangle (0, t) \Big].
\end{aligned}
\end{equation*}
\end{linenomath}
Since $\Lambda_i^\mathrm{I}$ and $\Lambda_i^\mathrm{II}$ ($i\in\mathcal{I}$) have values in $\mathbb{S}^6_{++}$, we assume that
\begin{equation} \label{eq:assup_weights_sign}
w_1(\ell_1) \leq 0, \qquad w_i(0) \geq 0, \quad \text{for all }i\geq 2,
\end{equation}
in order to render nonpositve the scalar products containing neither $-2\rho K_0$ nor $-2\rho K_i$ ($i \geq 2$). Then, defining for all $x\in [0, \ell_i], \ i \in \mathcal{I}$ and $K \in \mathbb{S}^6_{++}$, the matrix
\begin{linenomath}
\begin{equation*}
\mu_i(x,K) = (-2\rho K  + |w_i| \Lambda_i^\mathrm{I} + |w_i| K \Lambda_i^\mathrm{II} K)(x),
\end{equation*}
\end{linenomath}
we obtain $\overline{\mathcal{R}}(y,t) \leq \left \langle v_1 \,, \mu_1(0, K_0) v_1 \right \rangle (0, t) + \sum_{i=2}^N \left \langle v_i \,, \mu_i(\ell_i, K_i) v_i \right \rangle (\ell_i, t)$.

No additional assumption on the sign of $w_1(0)$ and $w_i(\ell_i)$ for $i\geq 2$ is needed to estimate the remaining boundary terms. This comes from the fact that, for any fixed $x \in [0, \ell_i]$ and $K \in \mathbb{S}^6_{++}$, the matrix $\mu_i(x,K)$ is negative semi-definite if $\rho$ is large enough in comparison to $|w_i(x)|$. Indeed, for this matrix to be negative semi-definite, it is sufficient to have $|w_i(x)| \leq \rho C_{\mu_i(x,K)}^{-1}$, where $C_{\mu_i(x,K)}$ denotes the largest eigenvalue of $(K^{-\sfrac{1}{2}} \Lambda_i^\mathrm{I} K^{-\sfrac{1}{2}} + K^{\sfrac{1}{2}} \Lambda_i^\mathrm{II} K^{\sfrac{1}{2}})(x)$. Hence, $\overline{\mathcal{R}}(y,t) \leq 0$ if 
\begin{equation}
\label{eq:cond_boundtermfb_neg}
|w_1(0)| \leq \frac{\rho}{ C_{\mu_1(0,K_0)}}, \qquad |w_i(\ell_i)| \leq \frac{\rho }{ C_{\mu_i(\ell_i,K_i)}}, \quad \text{for all } i \in \{2, \ldots, N\}.
\end{equation}

\subsubsection*{Step 5: Existence of $\overline{Q}_i$}
In summary, our first aim is to find $\mathbf{W}_i \in C^1([0, \ell_i]; \mathbb{R}^{6 \times 6})$ ($i \in \mathcal{I}$) fulfilling \eqref{eq:cond_Wi} -- examples of such functions being \eqref{eq:boldWi_I}, \eqref{eq:boldWi_MC} and \eqref{eq:boldWi_MCfrac}. Secondly, once that $\mathbf{W}_i$ ($i\in \mathcal{I}$) have been chosen, our aim is to find $\rho>0$ and increasing functions $w_i \in C^1([0, \ell_i])$ $(i \in \mathcal{I})$ such that $w_1$ is nonpositive while $w_i$ is nonnegative for all $i \geq 2$, and which fulfill \eqref{eq:cond_barSi_neg} as well as
\begin{equation} \label{eq:rho_large_enough}
\begin{aligned}
&w_1(0) > - \rho \beta_1, \quad \text{with }\beta_1= \min \left\{C_{\theta_1}^{-\sfrac{1}{2}}, C_{\mu_1(0,K_0)}^{-1}\right\},\\
& w_i(\ell_i) < \rho \beta_i, \quad \ \ \, \text{with } \beta_i = \min \left\{C_{\theta_i}^{-\sfrac{1}{2}}, C_{\mu_i(\ell_i,K_i)}^{-1}\right\}, \quad \text{for all }i \in \{2, \ldots, N\}.
\end{aligned}
\end{equation}
Here, \eqref{eq:rho_large_enough} is equivalent to \eqref{eq:cond_barQi_pos} and \eqref{eq:cond_boundtermfb_neg} (with strict inequalities), due to the monoticity and sign assumptions made on the weight functions.

With the help of Lemma \ref{lem:exist_g} \ref{item:exist_g_neg} and \ref{item:exist_g_pos}, we can conclude that there exists such weight functions, example of which are illustrated in Fig. \ref{fig:star_weights}. Indeed, for any $\rho>0$, one may choose $w_i = q_i - q_i(\ell_i)$, where $q_1$ is a function provided by \ref{item:exist_g_neg} with $\eta = C_{\overline{\Theta}_1} C_{\overline{\Lambda}_1}^{-1}$, with $\ell = \ell_1$ and with $a,b$ such that $(b- \rho \beta_1) < a < b \leq 0$, while $q_i$ ($i\geq 2$) is a function provided by \ref{item:exist_g_pos} with $\eta = C_{\overline{\Theta}_i} C_{\overline{\Lambda}_i}^{-1}$, with $\ell = \ell_i$ and with $a,b$ such that $0\leq a<b<(a+ \rho \beta_i)$.
\end{proof}

\begin{figure}
\centering
\includegraphics[scale=0.85]{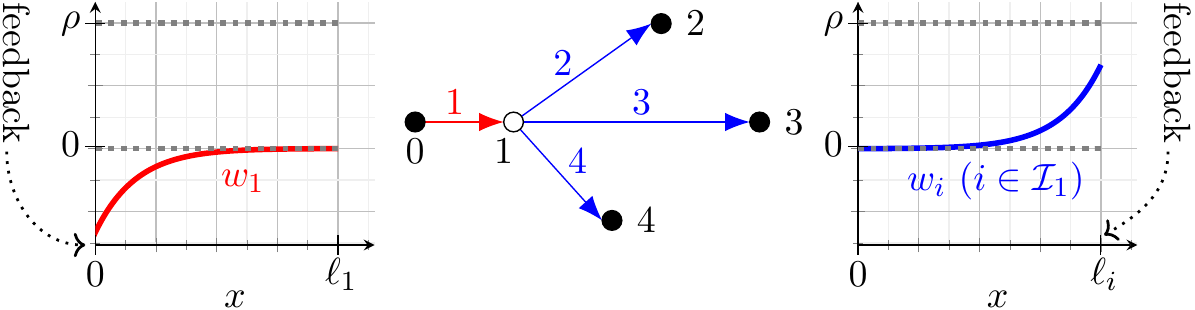}
\caption{Example of choice of $\rho$ and the weight functions.}
\label{fig:star_weights}
\end{figure}

\begin{remark}  \label{rem:remove_a_fb}
In Step 4.2 of the proof of Theorem \ref{thm:stabilization}, we could assume that \eqref{eq:assup_weights_sign} holds without disagreeing with the fact that each $w_i$ ($i \in \mathcal{I}$) has to be increasing, and this assumption enabled us to estimate boundary terms at the multiple node $n=1$.
However, out of the setting of star-shaped networks controlled at all simple nodes, it is not clear how one may obtain the property \ref{pty:barcalR_nonPos} of Lemma \ref{lem:class_barQi} without contradicting the monoticity assumption on $w_i$ ($i \in \mathcal{I}$).

An ensuing natural question is: what would happen if one of the controls is removed?
More precisely, we may assume that the beam of index $i=1$ is clamped or free at the node $n=0$, and we may possibly apply a feedback at the multiple node -- meaning that $K_1 \in \mathbb{S}^6_{++}$. Then, the steps 1 to 3 of the proof of Theorem \ref{thm:stabilization} remain unchanged since they concern the governing system, but the boundary terms stored in $\overline{\mathcal{R}}(y,t)$ are different. One has, from \eqref{eq:barcalR_M}-\eqref{eq:barcalR_0_clamped}-\eqref{eq:barcalR_S_not0}, that
\begin{linenomath}
\begin{equation*}
\begin{aligned}
\overline{\mathcal{R}}&(y,t) = 
- \left \langle z_1 \,, w_1 \Lambda_1^\mathrm{II} z_1\right \rangle (0, t) + \left \langle v_1 \,, (- 2\rho\overline{K}_1 + w_1 \Lambda_1^\mathrm{I}) v_1\right \rangle (\ell_1, t)  \\
&+ \left \langle z_1 \,, w_1 \Lambda_1^\mathrm{II} z_1 \right \rangle (\ell_1, t) + { \textstyle \sum_{i=2}^N} \big[ \left \langle v_i \,, (- 2 \rho K_i + w_i \Lambda_i^\mathrm{I} + w_i K_i \Lambda_i^\mathrm{II} K_i ) v_i\right \rangle (\ell_i, t) \\
&- \left \langle v_i \,, w_i \Lambda_i^\mathrm{I} v_i \right \rangle (0, t) + \left \langle z_i \,, w_i \Lambda_i^\mathrm{II} z_i \right \rangle (0, t) \big],
\end{aligned}
\end{equation*}
\end{linenomath}
Here, one cannot both assume that $w_1(0) \geq 0$ -- in order to estimate the first term in the above expression -- and that $w_1(\ell_1) \leq 0$ -- in order to estimate the term $\left \langle z_1 \,, w_1 \Lambda_1^\mathrm{II} z_1 \right \rangle (\ell_1, t)$ -- without contradicting the fact that $w_1$ should be increasing.
\end{remark}

\subsection{Point of view of the diagonal system}
\label{subsec:stab_D}

Now that Theorem \ref{thm:stabilization} has been proved from the point of view of \eqref{eq:syst_y}, let us consider the point of view of the diagonal system \eqref{eq:syst_Riem_1}. Even though computations for the latter are more involved, we are interested in observing how the proof unfolds and understanding if -- with the same Lyapunov functional -- the boundary terms may be treated in a manner that also allows for the removal of one of the controls, as mentioned in Remark \ref{rem:remove_a_fb}.

The physical and diagonal systems are related by the change of variable \eqref{eq:change_var_Li}. Hence, the energy $\mathcal{E}_i^\mathcal{D}$ of a beam (of index $i \in \mathcal{I}$) described by \eqref{eq:syst_Riem_1}, takes the form  
\begin{equation} \label{eq:def_energy_i_D}
\mathcal{E}_i^\mathcal{D}(t) = \int_0^{\ell_i} \langle
r_i(x,t) \,, Q_i^\mathcal{D}(x) r_i(x,t) \rangle dx,
\end{equation}
for $r_i$ solution to \eqref{eq:syst_Riem_1}, and where $Q_i^\mathcal{D} = (L_i^{-1})^\intercal Q_i^\mathcal{P} L_i^{-1}$; and one may compute that $Q_i^\mathcal{D} = \frac{1}{2}\mathrm{diag}\left(D_i^{-2}, D_i^{-2} \right)$.
%%%\begin{align} \label{eq:energy_matr_D}
%%%Q_i^\mathcal{D} = \frac{1}{2}\mathrm{diag}\left(D_i^{-2}, D_i^{-2} \right).
%%%\end{align}
Just as $\mathcal{P}$ refers to the physical system, here the subscript $\mathcal{D}$ refers to the diagonal system.
Let $k \in \{1, 2\}$, and define $\mathcal{L}\colon [0, T] \rightarrow [0, +\infty)$ by
\begin{equation} \label{eq:def_bar_lyap_Riem}
\mathcal{L}(t) = \sum_{i \in \mathcal{I}} \sum_{\alpha=0}^k \mathcal{L}_{\alpha i}(t), \quad \text{with} \ \ \mathcal{L}_{\alpha i}(t) := \int_0^{\ell_i} \left\langle \partial_t^\alpha r_i (x,t) \,, Q_i(x) \partial_t^\alpha r_i (x,t) \right\rangle dx,
\end{equation}
where $Q_i \in C^1([0, \ell_i]; \mathbb{R}^{12 \times 12})$ for all $i \in \mathcal{I}$, and $r \in C^0([0, T]; \mathbf{H}_x^k)$ is solution to \eqref{eq:syst_Riem_1}. 
As for the energy, if $Q_i$ and $\overline{Q}_i$ are such that
\begin{equation} \label{eq:rel_Q_barQ}
Q_i = (L_i^{-1})^\intercal \overline{Q}_i L_i^{-1}, \qquad \text{for all }i \in\mathcal{I},
\end{equation}
then $\mathcal{L}$ and $\overline{\mathcal{L}}$ are equivalent expressions -- the former from the point of view of the diagonal system and the latter from that of the physical system.

We can show, equivalently, that the zero steady state of \eqref{eq:syst_y} or that the zero steady state of \eqref{eq:syst_Riem_1} is locally $\mathbf{H}^k_x$ exponentially stable. As in Subsection \ref{subsec:stab_D}, in order to prove the latter, one may look for a quadratic $\mathbf{H}^k_x$ Lyapunov functional, namely, a functional of the form \eqref{eq:def_bar_lyap_Riem} which fulfills the assumptions of the Proposition \ref{prop:existence_Lyap_Riem} below -- whose proof is identical to that of Proposition \ref{prop:existence_Lyap}.

\begin{proposition} \label{prop:existence_Lyap_Riem}
Let $k \in \{1, 2\}$. Assume that for any fixed $T>0$, there exists $\delta>0$ such that the following holds: if $r \in C^0([0, T]; \mathbf{H}^k_x)$ is solution to \eqref{eq:syst_Riem_1} and if $\|r(\cdot, t)\|_{\mathbf{C}^{k-1}_{x}} \leq \delta$ for all $t \in [0, T]$, then there exists $\eta \geq 1$ and $\alpha >0$ such that
\begin{linenomath}
\begin{equation*}
\begin{aligned}
\eta^{-1} \|r(\cdot, t)\|_{\mathbf{H}^k_x}^2 &\leq \mathcal{L}(t) \leq \eta \|r(\cdot, t)\|_{\mathbf{H}^k_x}^2,  &&\text{for all } t \in [0, T] \\
\mathcal{L}(t) &\leq e^{-2 \beta t} \mathcal{L}(0), &&\text{for all }t \in [0, T]. 
\end{aligned}
\end{equation*}
\end{linenomath}
Then, the steady state $r \equiv 0$ of \eqref{eq:syst_y} is locally $\mathbf{H}^k_x$ exponentially stable.
\end{proposition}

As before, we may look for functions $Q_i$ ($i \in \mathcal{I}$) such that the associated $\mathcal{L}$ fulfills the assumptions of Proposition \ref{prop:existence_Lyap_Riem}, and a lemma -- Lemma \ref{lem:class_barQi_Riem} given below -- provides a class of such functions.
Let us introduce additional notation for functions $Q_i$ ($i \in \mathcal{I}$) having values in $\mathbb{D}^{12}$. We denote $Q_i = \mathrm{diag}(Q_{i}^-, \ Q_{i}^+)$, where $Q_{i}^-, Q_{i}^+ \in C^1([0, \ell_i]; \mathbb{D}^6)$. Moreover, for all $n \in \mathcal{N}$, we define the matrices $Q_n^\mathrm{out}, Q_n^\mathrm{in}, \bar{D}_n \in \mathbb{R}^{6k_n \times 6 k_n}$ as follows (see \eqref{eq:nota_In_indices}-\eqref{eq:def_kn}):
\begin{equation} \label{eq:def_Qnout_Qnin_1}
Q_n^\mathrm{out} =
\begin{cases}
Q_1^+(0) & n=0\\
Q_{n}^-(\ell_n) & n \in \mathcal{N}_S \setminus \{0\}\\
\mathrm{diag}(Q_{n}^-(\ell_n), \ Q_{i_2}^+(0), \ Q_{i_3}^+(0) , \ \ldots, \ Q_{i_{k_n}}^+(0) ) & n \in \mathcal{N}_M,
\end{cases}
\end{equation}
\begin{equation} \label{eq:def_Qnout_Qnin_2}
\begin{aligned}
Q_n^\mathrm{in} &= 
\begin{cases}
Q_1^-(0) & n=0\\
Q_n^+(\ell_n) & n\in\mathcal{N}_S \setminus\{0\}\\
\mathrm{diag}(Q_{n}^+(\ell_n), \ Q_{i_2}^-(0), \ Q_{i_3}^-(0) , \ \ldots, \ Q_{i_{k_n}}^-(0) ) &n\in \mathcal{N}_M,
\end{cases}\\
\bar{D}_n &= 
\begin{cases}
D_1(0) & n=0\\
D_n(\ell_n) & n \in \mathcal{N}_S\setminus \{0\}\\
\mathrm{diag} ( D_n(\ell_n), \ D_{ i_2}(0), \ D_{i_3}(0) , \ \ldots, \ D_{i_{k_n}}(0) ) & n\in\mathcal{N}_M.
\end{cases}
\end{aligned}
\end{equation}

\begin{lemma} \label{lem:class_barQi_Riem}
Assume that there exists $Q_i  \in C^1([0, \ell_i]; \mathbb{R}^{12\times 12})$ $(i \in \mathcal{I})$, fulfilling
\begin{enumerate}[label=(\roman*)]
\item \label{pty:Qi_diag_posDef}
for all $x \in [0, \ell_i]$ and all $i\in \mathcal{I}$, $Q_i(x)$ is diagonal and positive definite;
\item \label{pty:Si_negDef}
for all $x \in [0, \ell_i]$ and all $i\in \mathcal{I}$, $S_i(x)$ is negative definite, \\
where $S_i$ is defined by $S_i = \frac{\mathrm{d}}{\mathrm{d}x}( Q_i \mathbf{D}_i ) - Q_i B_i - B_i^\intercal Q_i$;
\item \label{pty:calMn_negSemiDef}
for all $n \in \mathcal{N}$, $\mathcal{M}_n \in \mathbb{R}^{6 k_n \times 6 k_n}$ is negative semi-definite, where $\mathcal{M}_n$ is defined by $($see \eqref{eq:def_calBn}$)$
\begin{equation} \label{eq:def_calMn}
\mathcal{M}_n = \mathcal{B}_n^\intercal Q_n^\mathrm{out} \bar{D}_n \mathcal{B}_n - Q_n^\mathrm{in} \bar{D}_n.
\end{equation}
\end{enumerate}
Then, the associated $\mathcal{L}$ $($see \eqref{eq:def_bar_lyap_Riem}$)$ fulfills the assumptions of Proposition \ref{prop:existence_Lyap}.
\end{lemma}

\begin{remark}
In Lemma \ref{lem:class_barQi_Riem}, instead of looking for $Q_i \in C^1([0, \ell_i]; \mathbb{D}^{12}_{++})$ $(i \in \mathcal{I})$, we could rather look for $Q_i \in C^1([0, \ell_i]; \mathbb{S}^{12}_{++})$ such that the product $(Q_i \mathbf{D}_i)(x)$ is symmetric for all $x \in [0, \ell_i]$, as in Lemma \ref{lem:class_barQi}. However, it is sufficient to assume the former, and is even equivalent if $\mathbf{D}_i(x)$ has distinct diagonal entries -- as a matrix commuting with a diagonal matrix with distinct diagonal entries is itself diagonal.
\end{remark}

\begin{proof}[Proof of Lemma \ref{lem:class_barQi_Riem}]

The proof of this lemma is identical to that of Lemma \ref{lem:class_barQi}, except the treatment of the boundary terms. Indeed, following the same procedure, one deduces that if the functions $Q_i$ ($i \in \mathcal{I}$) fulfill the properties \ref{pty:Qi_diag_posDef}, \ref{pty:Si_negDef} of Lemma \ref{lem:class_barQi_Riem} and if ${\mathcal{R}}(r,t)\leq 0$, where $\mathcal{R}$ is defined by 
\begin{linenomath}
\begin{equation*}
\begin{aligned}
\mathcal{R}(r,t) &= \left \langle r_1\,, Q_1 {\mathbf{D}}_1 u_1 \right \rangle(0, t) - {\textstyle \sum_{n \in \mathcal{N}_S \setminus \{0\}}} \left \langle r_n \,, Q_n \mathbf{D}_n r_n \right \rangle (\ell_n, t),\\
&\quad+ { \textstyle \sum_{n \in \mathcal{N}_M}} \big[ - \left \langle r_n \,, Q_n \mathbf{D}_n r_n \right \rangle (\ell_n, t) + {\textstyle \sum_{j=2}^N} \left \langle r_{i_j} \,, Q_{i_j} \mathbf{D}_{i_j} r_{i_j} \right \rangle (0, t)\big],
\end{aligned}
\end{equation*}
\end{linenomath} 
for all $t \in [0, T]$ and all $r \in C^0([0, T]; \mathbf{C}^0_x)$ fulfilling the nodal conditions of \eqref{eq:syst_Riem_1}, then the associated $\mathcal{L}$ satisfies the assumptions of Proposition \ref{prop:existence_Lyap_Riem}. However, one can compute further the term $\mathcal{R}(r,t)$. Noticing that $\langle r_i \,, Q_i \mathbf{D}_i r_i \rangle = -\langle r_i^-  \,, Q_{i}^-D_ir_i^-  \rangle + \langle r_i^+  \,, Q_{i}^+ D_i r_i^+  \rangle$ for all $i \in \mathcal{I}$, and using the definition \eqref{eq:r_out_in_mult_simp} of $r^\mathrm{in}_n, r^\mathrm{out}_n$ as well as the nodal conditions \eqref{eq:nodal_cond_calBn}, we obtain that
$\mathcal{R}(r,t) = \sum_{n \in \mathcal{N}} \langle r^\mathrm{in}_n(t), \mathcal{M}_n r^\mathrm{in}_n(t) \rangle$ with $\mathcal{M}_n$ defined by \eqref{eq:def_calMn}, which is nonpositive by \ref{pty:calMn_negSemiDef}.
\end{proof}

Note that functions $Q_i$ ($i \in\mathcal{I}$) defined by \eqref{eq:rel_Q_barQ} with $\overline{Q}_i$ fulfilling the assumptions of Lemma \ref{lem:class_barQi}, do fulfill the assumptions of Lemma \ref{lem:class_barQi_Riem}. Indeed, basic computations yield the following proposition.

\begin{proposition} \label{prop:rel_Lyap_Q_barQ}
Let $i \in \mathcal{I}$ and $\overline{Q}_i, Q_i \in C^1([0, \ell_i]; \mathbb{S}^{12})$ be such that \eqref{eq:rel_Q_barQ}. Then, 
\begin{enumerate}[label=\alph*)]
\item \label{item:equiv_barSi_Si}
for any $x \in [0, \ell_i]$, if $(\overline{Q}_iA_i)(x)$ is symmetric, then $\overline{S}_i(x)$ is negative definite if and only if $S_i(x)$ is negative definite;
\item \label{item:equi_barQi_Qi}
for any $x \in [0, \ell_i]$, $\overline{Q}_i(x)$ is positive definite if and only if $Q_i(x)$ is positive definite;
\item for any $y,r \in C^0([0, T]; \mathbf{C}^0_x)$ fulfilling the nodal conditions of \eqref{eq:syst_y} and \eqref{eq:syst_Riem_1} respectively, and any $t \in [0, T]$, $\overline{\mathcal{R}}(y,t)\leq 0$ if and only if $\mathcal{R}(r,t) \leq 0$.
\end{enumerate}
\end{proposition}

In view of Proposition \ref{prop:rel_Lyap_Q_barQ}, let $\overline{Q}_i$ ($i \in \mathcal{I}$) be defined by \eqref{eq:ansatz_Qibar}, where each $w_i\in C^1([0, \ell_i])$ is increasing with its derivative fulfilling \eqref{eq:cond_barSi_neg} for all $x \in [0, \ell_i]$, where the functions $\mathbf{W}_i$ are chosen among \eqref{eq:boldWi_I}-\eqref{eq:boldWi_MC}-\eqref{eq:boldWi_MCfrac}, and where $\rho_i = \rho >0$ for all $i \in \mathcal{I}$; furthermore, let $Q_i$ ($i \in \mathcal{I}$) be the associated functions defined by \eqref{eq:rel_Q_barQ}.
Then, by Proposition \ref{prop:rel_Lyap_Q_barQ} \ref{item:equiv_barSi_Si}-\ref{item:equi_barQi_Qi}, the properties \ref{pty:Qi_diag_posDef}-\ref{pty:Si_negDef} of Lemma \ref{lem:class_barQi_Riem} are fulfilled by $Q_i$ ($i \in \mathcal{I}$).
After some computations, one can obtain that
\begin{linenomath}
\begin{equation*}
Q_i = \rho Q_i^\mathcal{D} + \frac{w_i}{4} \begin{bmatrix}
J_i + J_i^\intercal & -J_i + J_i^\intercal\\
J_i - J_i^\intercal & -(J_i+J_i^\intercal)
\end{bmatrix},
\end{equation*}
\end{linenomath}
where $J_i = U_i \mathbf{C}_i^{\sfrac{1}{2}} \mathbf{W}_i \mathbf{C}_i^{-\sfrac{1}{2}} U_i^\intercal D_i^{-1}$.
Depending on $\mathbf{W}_i$, $J_i$ and $Q_i$ take the form\footnote{In this form, it is straightforward that $Q_i$ has values in $\mathbb{S}_{++}^{12}$ if and only if $\rho > |w_i(x)| C_{Q_i}(x)$ for all $x \in [0, \ell_i]$, where $C_{Q_i}(x)>0$ is the largest diagonal entry of $D_i(x)$ if \eqref{eq:boldWi_I}, the largest diagonal entry of $D_i(x)^{-1}$ if \eqref{eq:boldWi_MC}, and $C_{Q_i}\equiv 1$ if \eqref{eq:boldWi_MCfrac}.}
\begin{linenomath}
\begin{equation*}
J_i = \begin{cases}
D_i^{-1} &\text{if }\eqref{eq:boldWi_I}\\
D_i^{-3} &\text{if }\eqref{eq:boldWi_MC} \\
D_i^{-2} &\text{if }\eqref{eq:boldWi_MCfrac}
\end{cases}, \quad 
Q_i = \begin{cases}
\frac{1}{2} \mathrm{diag}\big(\rho \mathds{I}_6 + w_i D_i \,, \rho \mathds{I}_6 - w_i D_i \big) D_i^{-2} &\text{if }\eqref{eq:boldWi_I}\\
\frac{1}{2} \mathrm{diag}\big(\rho \mathds{I}_6 + w_i D_i^{-1} \,, \rho \mathds{I}_6 - w_i D_i^{-1} \big) D_i^{-2} &\text{if }\eqref{eq:boldWi_MC}\\
\frac{1}{2} \mathrm{diag} \big ((\rho + w_i)\mathds{I}_6 \,, (\rho - w_i) \mathds{I}_6 \big ) D_i^{-2} &\text{if }\eqref{eq:boldWi_MCfrac}.
\end{cases}
\end{equation*}
\end{linenomath}

Henceforth, we focus on the latter case \eqref{eq:boldWi_MCfrac} for simplicity, since $Q_i$ writes as the energy matrix $Q_i^\mathcal{D}$ (see \eqref{eq:def_energy_i_D}) multiplied by some weight functions $\rho+w_i$ and $\rho-w_i$:
\begin{equation} \label{eq:Qi_multEnergy}
Q_i = \mathrm{diag}((\rho + w_i)\mathds{I}_6 \,, (\rho - w_i) \mathds{I}_6) Q_i^\mathcal{D}.
\end{equation}
We now analyze the property \ref{pty:calMn_negSemiDef} of Lemma \ref{lem:class_barQi_Riem} -- i.e. the matrices $\mathcal{M}_n$ ($n\in\mathcal{N}$) -- by means of Proposition \ref{prop:calMn_form} below, which is proved in Appendix \ref{app:calMn_form}.
Beforehand, for any $n \in \mathcal{N}$, let us define the constant matrix $\overline{w}_n \in \mathbb{R}^{6k_n \times 6k_n}$ by
\begin{equation}\label{eq:def_Jbarn_wbarn}
\begin{aligned}
\overline{w}_n &= \begin{cases}
-w_1(0)\mathds{I}_6, & \text{if } n=0\\
w_n(\ell_n)\mathds{I}_6, & \text{if } n \in \mathcal{N}_S\setminus \{0\}\\
\mathrm{diag}(w_n(\ell_n) \mathds{I}_6\, , \, -w_{i_2}(0)\mathds{I}_6\, , \, \ldots\, , \, -w_{i_{k_n}}(0)\mathds{I}_6), & \text{if } n\in\mathcal{N}_M.
\end{cases}
\end{aligned}
\end{equation}

\begin{proposition}\label{prop:calMn_form}
Let $Q_i$ $(i \in \mathcal{I})$ be given by \eqref{eq:Qi_multEnergy}.
\begin{enumerate}[label=\arabic*)]
\item \label{item:form_calMn_mult}
Let $n \in \mathcal{N}_M$. If $K_n$ is symmetric positive definite or $K_n = \mathds{O}_6$, then $\mathcal{M}_n$ is congruent\footnote{
For any $d \in \{1, 2, \ldots\}$, $A, B \in \mathbb{R}^{d \times d}$ are called congruent if there exists $P \in \mathbb{R}^{d\times d}$, invertible, such that $A = P^\intercal B P$. In particular, $A$ is negative semi-definite if and only if $B$ is also negative semi-definite.
} to the matrix $($see \eqref{eq:def_kn}, \eqref{eq:def_block_diag}, \eqref{eq:def_I_crochet}$)$
\begin{equation} \label{eq:def_calMn_tilda_mult}
\begin{aligned}
\widetilde{\mathcal{M}}_n &= - 2 \rho k_n^{-1} \textswab{I}_{n} \mathrm{diag}(\overline{K}_n)  \textswab{I}_n + 2
\textswab{I}_n \overline{w}_n \mathrm{diag}(\sigma_i^n) \textswab{I}_n - \textswab{I}_n \overline{w}_n\mathrm{diag}(\sigma^n+\overline{K}_n)  \\
& -  \mathrm{diag}(\sigma^n+\overline{K}_n)\overline{w}_n \textswab{I}_n + \overline{w}_n \mathrm{diag}(\sigma^n+\overline{K}_n)\mathrm{diag}(\sigma_i^n)^{-1}\mathrm{diag}(\sigma^n+\overline{K}_n).
\end{aligned}
\end{equation}
\item \label{item:form_calMn_simp}
Let $n\in \mathcal{N}_S$. If $K_n$ is symmetric positive definite (controlled node), then $\mathcal{M}_n$ is congruent to the matrix
\begin{equation} \label{eq:def_calMn_tilda_simp}
\widetilde{\mathcal{M}}_n = \rho  \big[ (\mathds{I}_6 + \Upsilon^n)^{-2}(\mathds{I}_6 - \Upsilon^n)^2 - \mathds{I}_6 \big] + \overline{w}_n \big[
(\mathds{I}_6 + \Upsilon^n)^{-2}(\mathds{I}_6 - \Upsilon^n)^2 + \mathds{I}_6 \big],
\end{equation}
where $\Upsilon^n \in \mathbb{D}_{++}^6$ has the eigenvalues of $\bar{D}_n^{\sfrac{1}{2}} (\gamma_n^n)^\intercal \overline{K}_n \gamma_n^n \bar{D}_n^{\sfrac{1}{2}}$ as diagonal entries.\\
If $K_n = \mathds{O}_6$ or if the beam clamped then $\mathcal{M}_n = \overline{w}_n \bar{D}_n^{-1}$ (see Remark \ref{rem:calBn_free_clamped}).
\end{enumerate}
\end{proposition}

Let us now make some observations by using Proposition \ref{prop:calMn_form}.
For any simple node $n$ at which a control is applied, we can see in \ref{item:form_calMn_simp} that $\mathcal{M}_n$ is negative semi-definite if and only if the largest diagonal entry of $\widetilde{M}_n$ is nonpositive; namely, it is equivalent to the inequality $\rho C_{K_n} - w_1(0) \leq 0$ if $n=0$, and to $\rho C_{K_n} + w_n(\ell_n) \leq 0$ if $n\neq 0$,
%%%\begin{align*}
%%%&\rho C_{K_n} - w_1(0) \leq 0, \qquad \ \, \text{if }n=0,\\
%%%&\rho C_{K_n} + w_n(\ell_n) \leq 0, \qquad \text{if }n \neq 0,
%%%\end{align*}
where the negative constant $C_{K_n}<0$ is defined by
\begin{linenomath}
\begin{equation*}
C_{K_n} = \max_{1\leq j \leq 6} \frac{(1- \Upsilon^n_j)^2(1+\Upsilon^n_j)^{-2} - 1 }{(1- \Upsilon^n_j)^2(1+\Upsilon^n_j)^{-2} + 1},
\end{equation*}
\end{linenomath}
$\{\Upsilon^n_j\}_{j=1}^6$ denoting the diagonal entries of $\Upsilon^n$.
These inequalities hold for any choice of weight functions, provided that $\rho>0$ is large enough. For any simple node $n$ at which the beam is free or clamped, \ref{item:form_calMn_simp} yields that $\mathcal{M}_n$ is negative semi-definite if and only if $w_1(0)\geq 0$ for $n=0$, and $w_n(\ell_n) \leq 0$ for $n \neq 0$.

For any multiple node $n$, one deduces from \ref{item:form_calMn_mult} that, for any $K_n \in \mathbb{S}^6_{++}$ or $K_n = \mathds{O}_6$, the matrix $\mathcal{M}_n$ is necessarily negative semi-definite if
\begin{equation}\label{eq:remcalMn_signWeights}
w_n(\ell_n)\leq 0, \qquad w_i(0)\geq 0, \quad \text{ for all }i \in \mathcal{I}_n.
\end{equation}
In particular, we recover (in a different manner) Theorem \ref{thm:stabilization} -- for which the network contains only one multiple node $n=1$ and all simple nodes are controlled -- since one can then assume \eqref{eq:remcalMn_signWeights} (with $n=1$) without being in contradiction with the fact that all $w_i$ ($i \in \mathcal{I}$) are increasing.

\begin{remark}
For different networks, estimating $\mathcal{M}_n$ ($n \in \mathcal{N}_M$) is less evident.
In particular (see Remark \ref{rem:remove_a_fb}), one may consider the star-shaped network of Theorem \ref{thm:stabilization} and remove the control applied at the node $n=0$. 
By the Step 3 of the proof of Theorem \ref{thm:stabilization}, $w_1$ is increasing, while by \ref{item:form_calMn_simp}, $\mathcal{M}_0$ is negative semi-definite if and only if $w_1(0)\geq 0$. Hence, $w_1(\ell_1)>0$.

However, in the expression \eqref{eq:def_calMn_tilda_mult} of $\widetilde{\mathcal{M}}_1$, the term $w_1(\ell_1)>0$ appears in $P_1 := \overline{w}_1 \mathrm{diag}(\sigma^1+\overline{K}_1)\mathrm{diag}(\sigma_i^1)^{-1}\mathrm{diag}(\sigma^1+\overline{K}_1)$ and turn the estimation of $P_1$ into a difficult task.
One may think that applying a velocity feedback control at the multiple node $n=1$ could be of help, since it leads to the presence of $P_2 := - 2 \rho k_1^{-1} \textswab{I}_1 \mathrm{diag}(\overline{K}_1)  \textswab{I}_1$ in the expression of $\widetilde{\mathcal{M}}_1$. Indeed, had $P_2$ been negative definite, choosing $\rho$ large enough in comparison to the weight functions would be sufficient to estimate the other terms composing $\widetilde{\mathcal{M}}_1$. However, due to the presence of $\textswab{I}_1$ in its expression $($see \eqref{eq:def_I_crochet}$)$, $P_2$ is only negative semi-definite, which is not enough to estimate $P_1$.
\end{remark}

On a side note, let us comment on a case which amounts to a single beam having been divided into several shorter beams by placing nodes at different locations of its spatial domain. Namely, consider a network of beams which are serially connected (i.e. $k_n=2$) at $n \in \mathcal{N}_M$, without angle (i.e. $R_n(\ell_n) = R_{i_2}(0)$) and having the same material and geometrical properties at the multiple nodes (i.e. $\mathbf{M}_n(\ell_n) = \mathbf{M}_{i_2}(0)$ and $\mathbf{C}_n(\ell_n) = \mathbf{C}_{i_2}(0)$). Then, choosing $w_n(\ell_n) = w_{i_2}(0)$, we obtain after some computations that $\widetilde{\mathcal{M}}_n$ (and consequently $\mathcal{M}_n$) is equal to the null matrix if $K_n = \mathds{O}_6$. Hence, one can stabilize these beams by applying a feedback at one of the simple nodes of the overall network (see Fig. \ref{fig:serial_beams}).

\begin{figure}
\includegraphics[scale=0.85]{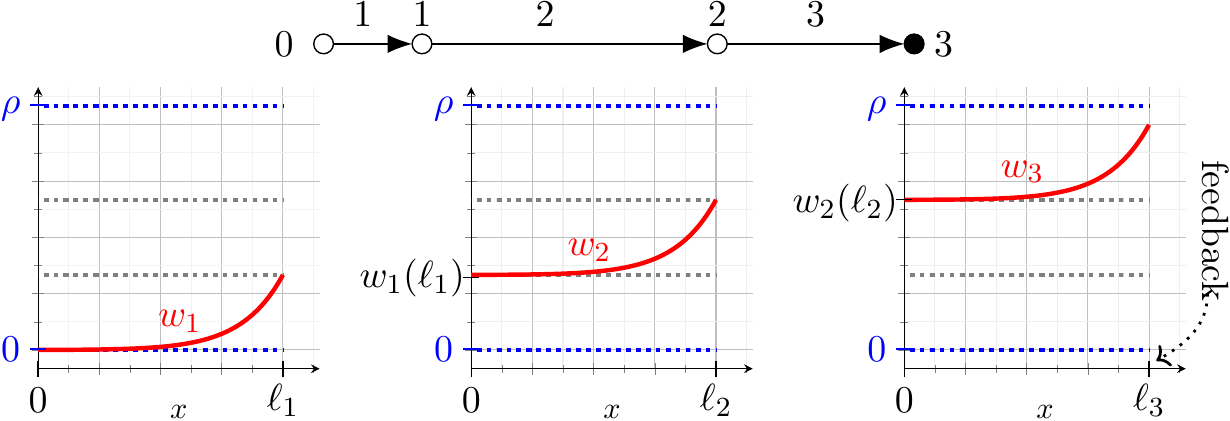}
\caption{Recovering the single beam case: example of choice of $\rho$ and weight functions.}
\label{fig:serial_beams}
\end{figure}

\section{Conclusion}

In this work, we have studied well-posedness (Theorem \ref{thm:well-posedness}) and stabilization (Theorem \ref{thm:stabilization}) for tree- and star-shaped networks of geometrically exact beams whose dynamics are given in terms of velocities and internal forces and moments, by the IGEB model.

Notably, using a quadratic Lyapunov functional (along the same lines as \cite{BC2016, bastin2017exponential}), we proved that stabilization of a star-shaped network can be achieved if the controls are applied at all the simple nodes. 
To construct this functional, we built an ansatz around the energy of the beam, taking into account the benefits and drawbacks inherent to form and properties of the model's coefficients.

A naturally ensuing question is whether or not this stability is preserved when one of the controls is removed. As the IGEB model is hyperbolic, it may be written in Riemann invariants (diagonal form) and stabilization may be equivalently analyzed both from the point of view of the physical \eqref{eq:syst_y} and diagonal \eqref{eq:syst_Riem_1} systems. We did so in Section \ref{sec:stab}, using the same Lyapunov functional, and observed that the diagonal point of view did not provide any edge in removing one of the feedback (see Subsection \ref{subsec:stab_D}). The difficulty may be only technical and the question remains open.

Let us also point out, without going into detail, that the local in time well-posedness result Theorem \ref{thm:well-posedness} and the stabilization result Theorem \ref{thm:stabilization} are likely to yield analogous results for corresponding networks in which the dynamics of each beam are given by the GEB model \eqref{eq:GEB}. In our previous work \cite{RL2019}, the transformation \eqref{eq:transfo} between the GEB and IGEB models was inverted under some assumptions, and the obtained solution $(\mathbf{p}, \mathbf{R})$ was shown to fulfill the overall system governed by GEB. Here, we believe that arguments similar to \cite{RL2019} apply, though after having inverted the transformation, one would first have to show the rigid joint property \eqref{eq:rigid_joint} before any of the other nodal conditions.

\subsection*{Acknowledgements}
The author is grateful to her PhD advisor G\"unter Leugering for his advice and encouragement.

\appendix

\section{Proof of Proposition \ref{prop:calMn_form}}

\label{app:calMn_form}
Let $n \in \mathcal{N}$, and let us compute the matrices $\mathcal{M}_n$ defined in \eqref{eq:def_calMn}.
In the context of Proposition \ref{prop:calMn_form}, $Q_i$ ($i \in \mathcal{I}$) is given by \eqref{eq:Qi_multEnergy}, or in other words $Q_i^- = \frac{1}{2} \left(\rho D_i^{-2} + w_i D_i^{-2} \right)$ and $Q_i^+ = \frac{1}{2}\left(\rho D_i^{-2} - w_i D_i^{-2} \right)$. As a consequence, the matrices $Q_n^\mathrm{out}, Q_n^\mathrm{in}$ involved in the expression of $\mathcal{M}_n$ take the form $Q_n^\mathrm{out} = \frac{1}{2}(\rho \bar{D}_n^{-2} + \overline{w}_n \bar{D}_n^{-2})$ and $Q_n^\mathrm{in} = \frac{1}{2}(\rho \bar{D}_n^{-2} - \overline{w}_n \bar{D}_n^{-2})$ (see \eqref{eq:def_Qnout_Qnin_1},\eqref{eq:def_Qnout_Qnin_2},\eqref{eq:def_Jbarn_wbarn}).

First, assume that $n$ is a multiple node. By \eqref{eq:def_calBn}, $\mathcal{B}_n = \mathbf{G}_n^{-1} \mathbf{H}_n \mathbf{G}_n$, with $\mathbf{H}_n = 2 \mathrm{diag}(\sigma^n+\overline{K}_n)^{-1} \textswab{I}_n \mathrm{diag}(\sigma_i^n) -  \mathds{I}_{6k_n}$. Hence, \eqref{eq:def_calMn} becomes
\begin{equation*}
\mathcal{M}_n = \frac{1}{2} \mathbf{G}_n^{\intercal} \mathbf{H}_n^\intercal (\mathbf{G}_n^{-1})^\intercal (\rho \mathds{I}_{6k_n} + \overline{w}_n) \bar{D}_n^{-1} \mathbf{G}_n^{-1} \mathbf{H}_n \mathbf{G}_n - \frac{1}{2} (\rho \mathds{I}_6 - \overline{w}_n) \bar{D}_n^{-1}.
\end{equation*}
We then obtain $\mathcal{M}_n = \frac{1}{2}\mathbf{G}_n^{\intercal} \left[ \mathbf{H}_n^\intercal (\rho \mathds{I}_{6k_n} + \overline{w}_n) \mathrm{diag}(\sigma_i^n) \mathbf{H}_n  - \rho \mathds{I}_6 + \overline{w}_n \mathrm{diag}(\sigma_i^n) \right]\mathbf{G}_n$ by using that $\overline{w}_n$ commutes with $\mathbf{G}_n$, and that $\mathrm{diag}(\sigma_i^n) =  (\mathbf{G}_n^{-1})^\intercal \bar{D}_n^{-1} \mathbf{G}_n^{-1}$ by definition (see \eqref{eq:def_gammai}-\eqref{eq:def_sigmai}-\eqref{eq:def_block_diag}).
Next, we replace $\mathbf{H}_n$ by its value given above and expand the product. Taking notice of the fact that the terms containing $\pm \rho$ are canceled, and that $(\rho \mathds{I}_{6k_n}  + \overline{w}_n)$ commutes with $\mathrm{diag}(\sigma_i^n)$, while $\textswab{I}_n$ commutes with $\mathrm{diag}(\sigma^n + \overline{K}_n)^{-1}$, we obtain the following expression for $\mathcal{M}_n$:
%
%%%Replacing $\mathbf{H}_n$ by its value, expanding the product, and noticing that the terms $\pm \rho \mathrm{diag}(\sigma_i^n)$ are canceled, we obtain
%%%\begin{align*}
%%%&\mathcal{M}_n = \mathbf{G}_n^{\intercal} \big[  
%%%-   \mathrm{diag}(\sigma_i^n)
%%%\textswab{I}_n  \mathrm{diag}(\sigma^n+\overline{K}_n)^{-1} (\rho \mathds{I}_{6k_n}  + \overline{w}_n) \mathrm{diag}(\sigma_i^n)\\
%%%& - (\rho \mathds{I}_{6k_n}  + \overline{w}_n) \mathrm{diag}(\sigma_i^n) \mathrm{diag}(\sigma^n+\overline{K}_n)^{-1}  
%%%\textswab{I}_n \mathrm{diag}(\sigma_i^n)\\
%%%&+2  \mathrm{diag}(\sigma_i^n)
%%%\textswab{I}_n  \mathrm{diag}(\sigma^n+\overline{K}_n)^{-1} (\rho \mathds{I}_{6k_n}  + \overline{w}_n) \mathrm{diag}(\sigma_i^n) \mathrm{diag}(\sigma^n+\overline{K}_n)^{-1}  
%%%\textswab{I}_n \mathrm{diag}(\sigma_i^n)\\
%%%& + \overline{w}_n \mathrm{diag}(\sigma_i^n) \big]\mathbf{G}_n.
%%%\end{align*}
%%%Using that $(\rho \mathds{I}_{6k_n}  + \overline{w}_n)$ commutes with $\mathrm{diag}(\sigma_i^n)$, while $\textswab{I}_n$ commutes with $\mathrm{diag}(\sigma^n + \overline{K}_n)^{-1}$, one may rewrite $\mathcal{M}_n$ as 
\begin{linenomath}
\begin{equation*}
\begin{aligned}
\mathcal{M}_n = \mathbf{P}_n^\intercal \big[&  2
\textswab{I}_n (\rho \mathds{I}_{6k_n}  + \overline{w}_n) \mathrm{diag}(\sigma_i^n)  
\textswab{I}_n - \textswab{I}_n  (\rho \mathds{I}_{6k_n}+ \overline{w}_n) \mathrm{diag}(\sigma^n+\overline{K}_n)  \\
&    -  \mathrm{diag}(\sigma^n+\overline{K}_n) (\rho \mathds{I}_{6k_n}  + \overline{w}_n)   \textswab{I}_n \\
& + \overline{w}_n \mathrm{diag}(\sigma^n+\overline{K}_n)\mathrm{diag}(\sigma_i^n)^{-1}\mathrm{diag}(\sigma^n+\overline{K}_n) \big] \mathbf{P}_n,
\end{aligned}
\end{equation*}
\end{linenomath}
where $\mathbf{P}_n \in \mathbb{R}^{6k_n \times 6k_n}$ is defined by $\mathbf{P}_n = \mathrm{diag}(\sigma^n + \overline{K}_n)^{-1} \mathrm{diag}(\sigma_i^n)\mathbf{G}_n$ which is clearly invertible.
Finally, from $\textswab{I}_n^2 = k_n \textswab{I}_n$, one can obtain the identities $\textswab{I}_n \mathrm{diag}(\sigma_i^n)  
\textswab{I}_n = k_n^{-1} \textswab{I}_n \mathrm{diag}(\sigma^n) \textswab{I}_n$ and $\textswab{I}_n \mathrm{diag}(\sigma^n+\overline{K}_n) = k_n^{-1} \textswab{I}_n \mathrm{diag}(\sigma^n+\overline{K}_n) \textswab{I}_n$. This enables us to deduce that $\mathcal{M}_n = \mathbf{P}_n^\intercal \widetilde{\mathcal{M}}_n \mathbf{P}_n$, for $\widetilde{\mathcal{M}}_n$ defined by \eqref{eq:def_calMn_tilda_mult}, concluding the proof of \ref{item:form_calMn_mult}.

Now, assume that $n$ is a simple node at which a control is applied. By definition, $\sigma_n^n = ((\gamma_n^n)^{-1})^\intercal \bar{D}_n^{-1} (\gamma_n^n)^{-1}$ and, as a consequence, one can write $\mathcal{B}_n = (\mathds{I}_6 + \bar{D}_n (\gamma_n^n)^\intercal \overline{K}_n \gamma_n^n)^{-1}(\mathds{I}_6 - \bar{D}_n (\gamma_n^n)^\intercal \overline{K}_n \gamma_n^n)$.
%%%If a velocity feedback control is applied at this node then $\mathcal{B}_n$ may be written in the form $\mathcal{B}_n = (\mathds{I}_6 + \bar{D}_n (\gamma_n^n)^\intercal \overline{K}_n \gamma_n^n)^{-1}(\mathds{I}_6 - \bar{D}_n (\gamma_n^n)^\intercal \overline{K}_n \gamma_n^n)$,
%%% CUT £££
%%%\begin{align*}
%%%\mathcal{B}_n 
%%%%&= (\sigma_n^n \gamma_n^n + \overline{K}_n \gamma_n^n)^{-1}(\sigma_n^n \gamma_n^n - \overline{K}_n \gamma_n^n)\\
%%%&= (\mathds{I}_6 + \bar{D}_n (\gamma_n^n)^\intercal \overline{K}_n \gamma_n^n)^{-1}(\mathds{I}_6 - \bar{D}_n (\gamma_n^n)^\intercal \overline{K}_n \gamma_n^n),
%%%\end{align*}
%%%where we used that $\sigma_n^n = ((\gamma_n^n)^{-1})^\intercal \bar{D}_n^{-1} (\gamma_n^n)^{-1}$ by definition.
Let the positive definite diagonal matrix $\Upsilon^n$ and the unitary matrix $Z_n$ be defined by requiring that 
\begin{linenomath}
\begin{equation*}
Z_n^\intercal \Upsilon^n Z_n = \bar{D}_n^{\sfrac{1}{2}} (\gamma_n^n)^\intercal \overline{K}_n \gamma_n^n \bar{D}_n^{\sfrac{1}{2}}.
\end{equation*}
\end{linenomath}
In particular, the diagonal entries of $\Upsilon^n$ are the eigenvalues of the above right-hand side.
One can compute that $\mathcal{B}_n = \bar{D}_n^{\sfrac{1}{2}}Z_n^\intercal (\mathds{I}_6 + \Upsilon^n)^{-1}(\mathds{I}_6 - \Upsilon^n) Z_n \bar{D}_n^{-\sfrac{1}{2}}$,
%%% CUT £££
%%%\begin{align*}
%%%\mathcal{B}_n = \bar{D}_n^{\sfrac{1}{2}}Z_n^\intercal (\mathds{I}_6 + \Upsilon^n)^{-1}(\mathds{I}_6 - \Upsilon^n) Z_n \bar{D}_n^{-\sfrac{1}{2}}, 
%%%\end{align*}
and $\mathcal{M}_n$ then takes the form $\mathcal{M}_n = \frac{1}{2}\bar{D}_n^{-\sfrac{1}{2}} Z_n^\intercal  \widetilde{\mathcal{M}}_n Z_n \bar{D}_n^{-\sfrac{1}{2}}$ for $\widetilde{\mathcal{M}}_n$ defined by \eqref{eq:def_calMn_tilda_simp}.
%%\begin{align*}
%%\mathcal{M}_n 
%%&= \frac{1}{2}\bar{D}_n^{-\sfrac{1}{2}} Z_n^\intercal  \Big[  \rho  \big( (\mathds{I}_6 + \Upsilon^n)^{-2}(\mathds{I}_6 - \Upsilon^n)^2 - \mathds{I}_6 \big) \\
%%&\quad + \overline{w}_n \big(
%%(\mathds{I}_6 + \Upsilon^n)^{-2}(\mathds{I}_6 - \Upsilon^n)^2 + \mathds{I}_6 \big) \Big] Z_n \bar{D}_n^{-\sfrac{1}{2}}.
%%\end{align*}
Finally, the cases of a clamped or free (i.e. $K_n=\mathds{O}_6$) beam at the node $n$ follow from basic computations, as the matrix $\mathcal{B}_n$ is then equal to $- \mathds{I}_6$ or $\mathds{I}_6$, respectively. Hence, we have proved \ref{item:form_calMn_simp}.

\bibliographystyle{plain}
\bibliography{bibli}

\end{document}